\documentclass[amstex,12pt]{amsart}
\usepackage{geometry}
\usepackage{mathtools,amssymb,amsthm,mathrsfs,color,lineno,paralist,graphicx,float}
\usepackage[linktocpage,backref=page,colorlinks,
linkcolor=red,
anchorcolor=red,
citecolor=red,
]{hyperref}
\usepackage{color}
\usepackage{mathrsfs}
\usepackage{cleveref}
\usepackage{mathscinet}
\usepackage{calc}
\usepackage{etoolbox}

\newtheorem{thm}{Theorem}
\theoremstyle{plain}
\theoremstyle{theorem}
\newtheorem{proposition}[thm]{Proposition}
\newtheorem{cor}[thm]{Corollary}
\newtheorem{lemma}{Lemma}[section]
\newtheorem{definition}{Definition}[section]

\newtheorem{rem}{Remark}[section]
\newtheorem{example}{Example}[section]

\numberwithin{equation}{section}

\definecolor{lgray}{gray}{0.9}

\def\leq{\leqslant}
\def\geq{\geqslant}
\def\tilde{\widetilde}
\def\hat{\widehat}

\def\dif{\mathop{}\!\mathrm{d}}

\allowdisplaybreaks
\renewcommand*{\backref}[1]{}
\renewcommand*{\backrefalt}[4]{\quad \tiny
	\ifcase #1 (\textbf{NOT CITED.})
	\or    (Cited on page~#2.)
	\else   (Cited on page~#2.)
	\fi}
\makeatletter

\author{Nian Liu}
\address[Nian Liu]
{Department of Mathematics\\
 Pennsylvania State University\\
 State College, PA 16801, USA}
\email[N.~Liu]{nkl5330@psu.edu}

\author{Xue Liu}
\address[Xue Liu]
{School of Mathematics\\ Southeast University\\
 Nanjing 211189, PR China}
\email[X.~Liu]{xueliuseu@seu.edu.cn}


\begin{document}

\begin{abstract}
In this paper, we study the irregular set of any continuous observable for a class of skew product transformations, which is driven by a uniquely ergodic homeomorphism system $(\Omega,\mathbb{P},\theta)$ and satisfies Anosov and topological mixing on fibers property. The orbit-gluing technique provided by the fiber specification property are utilized. For such systems, we prove the following rigidity phenomenon that if the irregular set of any continuous observable is nonempty on a fiber, then the irregular set must be nonempty, residual, and carry full fiber topological entropy on $\mathbb{P}$-a.e. fibers. This result can be applied to a class of partially hyperbolic systems.
\end{abstract}

\title{Irregular sets for a class of skew product transformations}
\maketitle
\parskip 10pt         

\section{Introduction}\label{sec introduction}

\subsection{Background}
Multifractal analysis is a powerful tool in dynamical systems for understanding the statistical and geometric complexity of orbits. A central focus lies in studying the multifractal structure associated with the Birkhoff average of continuous observables and the associated level sets. In \cite{variationalprincipleforrandomanosovsystems}, we have investigated the conditional level sets of Birkhoff average of any continuous function from the viewpoint of fiber Bowen's topological entropy, i.e., some kind of \emph{dimensional} complexity of these level sets, for a class of skew product transformations. This class of skew product transformations is driven by a uniquely ergodic homeomorphism and satisfies  Anosov and topological mixing on fibers property. Therefore, it's natural to investigate that whether irregular sets would carry large \emph{topological} and \emph{dimensional} complexity for such skew product transformations.  This question has been thoroughly studied in deterministic and symbolic settings, particularly in systems with specification or almost specification properties. However, such tools are not directly applicable to skew product systems, especially when the base dynamics lacks mixing or the global system lacks uniform hyperbolicity. Our goal in this paper is to address this gap. Before we move further into this topic, some classical results and limitations in \emph{deterministic} cases need to be discussed first.

Given a continuous transformation $f$ on a compact metric space $M$, and a continuous observable $\varphi:M\to\mathbb{R}$, 
\emph{the irregular set of $\varphi$} is defined by
\begin{displaymath}
	I_{\varphi}=\left\{x\in M:\lim_{n\to\infty}\frac{1}{n}\sum_{i=0}^{n-1}\varphi(f^i(x))\ \mbox{does not exist}\right\},
\end{displaymath}and \emph{the irregular set of the system} $(M,f)$ is defined by
\begin{displaymath}
	I=\cup_{\varphi\in C(X,\mathbb{R})} I_{\varphi}.
\end{displaymath}Points in $I_{\varphi}$ are called \emph{$\varphi$-irregular points} (also known as points with \emph{historic behavior} in \cite{Ruelle2001Historicbehaviour} and \cite{Takens2008Historicbehaviour}), or \emph{irregular points} if there is no ambiguity.
By the well-known Birkhoff's ergodic theorem, for any $f$-invariant probability measure $\mu$ on $M$, $I_{\varphi}$ carries zero $\mu$-measure. Hence both $I_{\varphi}$ and $I$ can be neglected from the perspective of $f$-invariant measures. 
However, points exhibiting historic behavior do have large complexity, and can be fairly large compared with the whole systems $M$ from both \emph{dimensional} and \emph{topological} perspectives, 
provided that the topological dynamics $(M,f)$ has some special properties(e.g., some specification-like property including specification property \cite{Thompson2010irregularspecification}, almost specification property \cite{Thompson2012almostspecification} and $g$-almost product property \cite{PfisterSullivan2007galmost}).

From the \emph{dimensional} perspective, the irregular set $I_{\varphi}$ and $I$ can carry full dimension. Notice that dimensional characteristics include Hausdorff dimension, Bowen's topological entropy \cite{Bowen1973Topentropy} and topological pressure \cite{PesinPitskel1984toppressure}. Firstly, Pesin and Pitskel \cite {PesinPitskel1984toppressure} pointed out  that the irregular set $I$ for full shift on two symbols carries full topological entropy. Later on, Barriera and Schmeling \cite{BarreiraSchmeling2000} studied the irregular set of some uniformly hyperbolic systems including conformal repellers and conformal horseshoes using symbolic dynamics and thermodynamical formalism. 
 The results in \cite{BarreiraSchmeling2000} have been extended to more general settings by Fan \emph{et al} in \cite{FanFeng2001} and Feng \emph{et al} in \cite{FengLauWu2002}. 
Besides, the unified
multifractal framework introduced by Olsen and Winter in a sequence of papers \cite{Olsen2003deformedmeasure,Olsen2006deformedmeasure3,Olsen2007deformedmeasure2,Olsen2008deformedmeasure4} provided a systematic basis for further studies on irregular points. 
Chen \emph{et al} \cite{Chenercai2005topologicalentropyfordivergencepoints} proved that irregular sets for systems with specification property is either empty or carries full topological entropy. 
Similar results about topological pressure are obtained by Thompson in \cite{Thompson2010irregularspecification}. 

It's noticeable that the \emph{specification property} has played an important role in most of the above results, since specification property can help one to construct irregular points as many as possible. For systems without specification property, there are also some innovations, see \cite{PfisterSullivan2007galmost,Thompson2012almostspecification,ChenercaiandZhou2012TopologicalpressureforZdaction, ChenercaiandZhou2013Multifractalanalysisforthehistoricset,Tian2018shadowing} and references therein.

From \emph{topological} perspective, systems with specification-like property can exhibit abundance historic behavior. Recall that the complement of any set of first category is called a \emph{residual set}. In a complete metric space $X$, a set $E$ is residual if and only if $E$ contains a dense $G_{\delta}$ subset of $X$ (see more details in \cite{OxtobyMeasureandCategorybook1980}). Barreira \emph{et al} \cite{BarreiraLiValls2012fullshiftirregular,BarreiraLiValls2014residual} showed that the irregular set $I_{\varphi}$ for full shifts and subshifts with weak specification property is either empty or residual. Similar results for systems with specification property were obtained by Li and Wu \cite{LiWu2014}.
Recent work concentrating on irregular sets can be found in \cite{AraujoPinheiro2021Abundanceofhistoricbehavior,MariaPaulo2021genericityofhistoricbehavior,NakanoYamamoto2021irregularsetofrpiecewisemonotonicmaps,BarrientosNakanoetal2022entropydimensionofirregularsetsfornonhyperbolicsystems,Gelfertetal2022AxiomAandC1+alphasystems} and references therein.

We note that the construction of Moran fractals has played an important role in the proof of \cite{TakensVerbitsky2003variationalprinciple,Thompson2009Variationalprinciplepressure,Thompson2010irregularspecification,Thompson2012almostspecification}, since the estimation of topological entropy is based on entropy distribution principle, and the structure of Moran fractals enable one to construct a suitable measure to apply this principle. To construct suitable fractals contained in the irregular set $I_{\varphi}$, one usually need the orbit-gluing technique provided by the specification property or a weak form specification property. However, for the following concrete skew product systems $\Theta:\mathbb{T}\times \mathbb{T}^2\to \mathbb{T}\times \mathbb{T}^2$, specification property can not be achieved. Since that the classical specification property implies topologically mixing and the factor system
of topologically mixing system is also topologically mixing, but$(\mathbb{T},\theta)$ as factor of $(\mathbb{T}\times \mathbb{T}^2,\Theta)$ is not
topologically mixing. Therefore, the system $(\mathbb{T}\times\mathbb{T}^2, \Theta)$ does not have classical specification property.
\begin{example}[fiber Anosov maps on 2-dimension tori]\label{fiber Anosov maps on 2-d tori}
  The skew product $\Theta:\mathbb{T}\times\mathbb{T}^2\to\mathbb{T}\times\mathbb{T}^2$
is given by
\begin{displaymath}
	\Theta(\omega, x) = (\theta\omega, F_{\omega}x):=(\omega + \alpha, Tx+h(\omega)),
\end{displaymath}
where $\alpha\in\mathbb{R}\setminus\mathbb{Q}, T$ is any hyperbolic toral automorphism, and $h:\mathbb{T}\to\mathbb{T}^2$
is a continuous map.
\end{example}
Moreover, it is unknown whether the skew product transformations considered in this paper have weak form of specification property. It is also worth to point out that the Example \ref{fiber Anosov maps on 2-d tori} is a class of partially hyperbolic systems. This partially hyperbolic systems satisfies the Anosov and topological mixing on fibers property, which generalized the classical topological mixing Anosov system. The definition of the Anosov and topological mixing on fibers property can be found in the next subsection. Recently, Huang, Lian and Lu \cite{HLL} stated the random specification property for such class of skew product transformations. The random specification property can not be applied directly in investigating the multifractal structure of the Birkhoff average, but its fiber realization plays a crucial role in \cite{variationalprincipleforrandomanosovsystems}.

\subsection{Contributions and main results}
To formally state our main results, we give some notations. Let $M$ be a connected closed smooth Riemannian manifold, and $d_M$ be the induced Riemannian metric on $M$.  Let $\theta:\Omega\to\Omega$ be a homeomorphism on the compact metric space $(\Omega,d_{\Omega}).$ The product space $\Omega\times M$ is a compact metric space endowed with the distance $d((\omega_1,x_1),(\omega_2,x_2))=d_{\Omega}(\omega_1,\omega_2)+d_M(x_1,x_2)$ for any $\omega_1,\omega_2\in \Omega$ and $x_1,x_2\in M$.  Let $F:\Omega\to\mathrm{Diff}^2(M)$ be a continuous map, where $\mathrm{Diff}^2(M)$ is endowed with $C^2$-topology. The skew product $\Theta:\Omega\times M\to\Omega\times M$ induced by $F$ and $\theta$ is defined by
\begin{displaymath}
	\Theta(\omega,x)=(\theta\omega,F_{\omega}x),\quad\forall\omega\in\Omega,\forall x\in M.
\end{displaymath}
where we rewrite $F(\omega)$ as $F_{\omega}$. The inductively:
\begin{equation*}
	\Theta^n(\omega,x)=(\theta^n\omega,F_{\omega}^nx):=\begin{cases}
		(\theta^n\omega,F_{\theta^{n-1}\omega}\circ\cdots\circ F_{\omega}x),&\mbox{if}\ n>0,\\ (\omega,x),&\mbox{if}\ n=0,\\ (\theta^{n}\omega,(F_{\theta^{n}\omega})^{-1}\circ\cdots\circ (F_{\theta^{-1}\omega})^{-1}x),&\mbox{if}\ n<0.
	\end{cases}
\end{equation*}

In this paper, we consider the target system $(\Omega\times M,\Theta)$ satisfying:	
\begin{enumerate}
	\item \label{C1} the driven system $(\Omega,\mathcal{B}_{\mathbb{P}}(\Omega),\mathbb{P},\theta)$ is a uniquely ergodic homeomorphism endowed with the unique invariant Borel probability measure $\mathbb{P}$ and $\sigma$-algebra $\mathcal{B}_{\mathbb{P}}(\Omega)$, which is the completion of Borel $\sigma$-algebra $\mathcal{B}(\Omega)$ with respect to $\mathbb{P}$;
	\item \label{C2} Anosov on fibers:  for any $(\omega,x)\in \Omega\times M$, there is a splitting of the tangent bundle of $M_\omega=\{\omega\}\times M$ at $x$, i.e.
	\begin{equation*}
		T_xM_\omega=E_{(\omega,x)}^s\oplus E_{(\omega,x)}^u,
	\end{equation*}which depends continuously on $(\omega,x)\in\Omega\times M$ with $\dim E^s_{(\omega,x)},\ \dim E_{(\omega,x)}^u>0$, and satisfies that
	\begin{equation*}
		D_xF_\omega E_{(\omega,x)}^u=E_{\Theta(\omega,x)}^u,\ \ D_xF_\omega E_{(\omega,x)}^s=E_{\Theta(\omega,x)}^s
	\end{equation*}and
	\begin{equation*}
		\begin{cases}
			|D_xF_\omega\xi|\geq e^{\lambda_0}|\xi|, & \ \forall \xi\in E_{(\omega,x)}^u, \\
			|D_xF_\omega\eta|\leq e^{-\lambda_0}|\eta|, &\ \forall \eta\in E_{(\omega,x)}^s,
		\end{cases}
	\end{equation*}where $\lambda_0>0$ is a constant and $|\cdot|$ is the norm on $T_xM_{\omega}$ induced by the Riemannian metric on $M_{\omega}$;
	\item \label{C3} topological mixing on fibers: for any nonempty open sets $U,V\subset M$, there exists  $N>0$ such that for any $n\geq N$ and $\omega\in\Omega$
	\begin{equation*}
		\Theta^n(\{\omega\}\times U)\cap (\{\theta^n\omega\}\times V)\not=\varnothing.
	\end{equation*}
\end{enumerate}
Bowen showed that topologically mixing Axiom A systems come with specification in \cite{Bowen1971specification}. Therefore, it's quite reasonable to expect our target systems are endowed with uniformly hyperbolic structures and some kind of mixing property on fibers. Examples under consideration are given in the next subsection.

\begin{definition}
	For any $\varphi\in C(\Omega\times M,\mathbb{R})$, the irregular set of $\varphi$ is defined by
	\begin{equation*}
		I_{\varphi}=\left\{(\omega,x)\in\Omega\times M:\lim_{n\to\infty}\frac{1}{n}\sum_{i=0}^{n-1}\varphi(\Theta^i(\omega,x))\mbox{ does not exist} \right\},
	\end{equation*}
and the $\omega$-fiber of $I_{\varphi}$ is defined by
\begin{displaymath}
	I_{\varphi}(\omega)=\left\{x\in M:(\omega,x)\in I_{\varphi}\right\}.
\end{displaymath}
\end{definition}
Now we are ready to state our main theorems. 
\begin{thm}[Dimensional perspective]\label{thm irregular set}
	Let $(\Omega\times M,\Theta)$ be a skew product transformation satisfying Anosov and topological mixing on fiber properties, and driven by a uniquely ergodic homeomorphism $(\Omega,\mathbb{P},\theta)$. Given a continuous observable $\varphi:\Omega\times M\to\mathbb{R},$ if $I_{\varphi}\neq\varnothing$, we then have the following results:
	\begin{enumerate}
		\item for $\mathbb{P}$-a.e. $\omega\in\Omega,$ $I_{\varphi}(\omega)$ is nonempty. 
		\item for $\mathbb{P}$-a.e. $\omega\in\Omega, I_{\varphi}(\omega)$ carries full fiber topological entropy, i.e.
		\begin{displaymath}
			h_{top}(F,I_{\varphi}(\omega),\omega)=h_{top}(F).
		\end{displaymath}
	\end{enumerate}
Here, $h_{top}(F,I_{\varphi}(\omega),\omega)$ refers to the fiber Bowen's topological entropy of $F$ restricted on $I_{\varphi}(\omega)$ and the fiber $\{\omega\}\times M$, and $h_{top}(F)$ refers to the fiber topological entropy of $F$. Definitions of these notions are given in subsection \ref{subsection fiber topological entropy}.
\end{thm}

\begin{thm}[Topological perspective]\label{thm residual}
	Let $(\Omega\times M,\Theta)$ be a skew product transformation satisfying Anosov and topological mixing on fiber properties, and driven by a uniquely ergodic homeomorphism $(\Omega,\mathbb{P},\theta)$. Given a continuous observable $\varphi:\Omega\times M\to\mathbb{R}$, if $I_{\varphi}\neq \varnothing,$ then for $\mathbb{P}$-a.e. $\omega\in\Omega$, $ I_{\varphi}(\omega)$ is residual in $M$.
\end{thm}

We remark that these results are highly nontrivial. Since $I_{\varphi}\neq\varnothing$ only implies that there exists some $\omega^*\in \Omega$, such that $I_{\varphi}(\omega^*)\neq\varnothing$. Our main results reveal a rigidity phenomenon: if the irregular set of a continuous observable is nonempty on any single fiber, then it must be nonempty, residual, and carry full fiber topological entropy on almost every fiber. This provides the first such results in the literature for skew product systems with partially hyperbolic or non-specification dynamics. 


 Theorem \ref{thm residual} describes the fiber topological complexity of $I_\varphi$ in $M$. One may ask whether $I_\varphi$ is dense in $\Omega\times M$. The following corollary gives a sufficient condition, and its proof is addressed in section \ref{sec proof of lemmas}.
\begin{cor}\label{corollary dense}
	Let $(\Omega\times M,\Theta)$ be a skew product transformation satisfying Anosov and topological mixing on fiber properties and driven by a uniquely ergodic homeomorphism $(\Omega,\mathbb{P},\theta)$. Assume that every nonempty open subset $U$ of $\Omega$ is assigned \emph{positive} $\mathbb{P}$-measure, then the irregular set $I_{\varphi}$ is either empty or dense in $\Omega\times M$.
\end{cor}

\subsection{Examples and Applications}\label{subsection application}
\subsubsection{Fiber Anosov maps on 2-dim tori}
In \cite[Section 8.1.2]{HLL} or \cite[Section 7]{HuangLianLu2023SCICHINA}, Huang, Lian and Lu proved the
Example \ref{fiber Anosov maps on 2-d tori} is Anosov and topological mixing on fibers. Note
that the irrational rotation on circle $\theta:\omega\mapsto\omega+\alpha$ is a uniquely ergodic system with the unique
$\theta$-invariant measure $\mathbb{P}$, i.e. Lebesgue measure on $\mathbb{T}$.
Results in \cite{Chenercai2005topologicalentropyfordivergencepoints,Thompson2010irregularspecification} can not apply to $\Theta$.

Note that $T$ is linear map, for any $\omega\in\mathbb{T}$ and $n\in\mathbb{N}$, the fiber Bowen’s metric
\begin{displaymath}
	\begin{split}
		d_{\omega}^n(x,y)&=\max_{0\leq i\leq n-1}\{d_{\mathbb{T}^2}(F_{\omega}^i(x),F_{\omega}^i(y))\}\\&=\max_{0\leq i\leq n-1}\{d_{\mathbb{T}^2}(T^i(x),T^i(y))\}\\&=d_n^T(x,y),
	\end{split}
\end{displaymath}
where $d_n^T$ is the usual Bowen's metric induced by $T$ on $\mathbb{T}^2$. As a consequence, for all $\omega\in\Omega$, we have
$B_n(\omega, x,\epsilon)=B^T_n(x, \epsilon):= \{y \in M : d^T_n(x, y)<\epsilon\}$, and therefore, $h_{top}(F, Z, \omega) = h_{top}(T, Z)$ for any nonempty subset $Z\subset\mathbb{T}^2$, where $h_{top}(T, Z)$ is the Bowen topological entropy defined on
noncompact set for deterministic system $(\mathbb{T}^2,T)$ (see \cite{Bowen1973Topentropy,Pesinbook}). Therefore, for such $\Theta:\mathbb{T}\times\mathbb{T}^2\to\mathbb{T}\times\mathbb{T}^2$, $h_{top}(F,I_{\varphi}(\omega),\omega)$ in Theorem \ref{thm irregular set} can be replaced by $h_{top}(T, I_{\varphi}(\omega))$.

Denote by $h_{\mu}(F)$ the fiber measure-theoretical entropy of $\mu$ (see more details in subsection \ref{subsection fiber measure entropy}). Combining theorem \ref{thm irregular set} and theorem \ref{thm residual}, one conclude that
\begin{cor}
	Let $\varphi\in C(\mathbb{T}\times\mathbb{T}^2,\mathbb{R}),\mathbb{P}$ denote the Lebesgue measure on $\mathbb{T}$. Assume that $I_{\varphi}\neq\varnothing$, then for $\mathbb{P}$-a.e. $\omega\in\mathbb{T}$, we have
	\begin{enumerate}
		\item $h_{top}(T, I_{\varphi}(\omega))= h_{top}(F)=h_{top}(\Theta),$
		where the second equality follows from the Abramov-Rohlin formula $h_{\mu}(\Theta)=h_{\mu}(F)+h_{\mathbb{P}}(\theta)$ for $\mu\in I_{\Theta}(\mathbb{T}\times\mathbb{T}^2)$ and the variational principle between (fiber) topological entropy and (fiber) measure-theoretical entropy;
		\item $I_{\varphi}(\omega)$ is residual in $\mathbb{T}^2$.
	\end{enumerate}
	
\end{cor}

Notice that every nonempty open set of $\mathbb{T}$ is assigned positive Lebesgue measure. The following corollary follows directly from corollary \ref{corollary dense}.
\begin{cor}
	Let $\varphi\in C(\mathbb{T}\times\mathbb{T}^2,\mathbb{R}),$ then $I_{\varphi}$ is either empty or dense in $\mathbb{T}\times\mathbb{T}^2$.
\end{cor}
\subsubsection{Random composition of $2\times 2$ area-preserving positive matrices driven by uniquely ergodic subshift}

Let
\begin{equation*}
	\left\{B_i=\begin{pmatrix}
		a_i & b_i  \\
		c_i & d_i  \\
	\end{pmatrix}\right\}_{1\leq i\leq k}
\end{equation*}be $2\times2$ matrices with $a_i,b_i,c_i,d_i\in\mathbb{Z}^+$, and $|a_id_i-c_ib_i|=1$ for any $i\in\{1,...,k\}$.  Let $\Omega=\{1,...,k\}^\mathbb{Z}$, with the left shift operator $\theta$, be the symbolic dynamical system with $k$ symbols.
For any $\omega=(...,\omega_{-1},\omega_0,\omega_1,...)\in \Omega$, we define $F_\omega=B_{\omega_0}.$ Then the skew product $\tilde{\Theta}:\Omega\times\mathbb{T}^2\rightarrow \Omega\times\mathbb{T}^2$ defined by
\begin{equation*}
	\tilde{\Theta}(\omega, x)=\big(\theta\omega, F_\omega x\big)
\end{equation*}
is an Anosov and topological mixing on fibers system by Proposition 8.2 and Theorem 8.2 in \cite{HLL}.

Now we let $\Omega_s\subset \Omega$ be any uniquely ergodic subshift. For instance, when $k\geq 2$, let $(\Omega_s,\theta)$ be the Arnoux-Rauzy subshift, which is uniquely ergodic and minimal (see \cite[Proposition 2.17]{David2007strictlyergodicsubshifts}). Especially, when $k=2$, the Arnoux-Rauzy subshift is the Sturmian subshift. Then the following system is one of target systems:
\begin{equation*}
	\Theta:\Omega_s\times \mathbb{T}^2\to \Omega_s\times \mathbb{T}^2\mbox{ by }\Theta=\tilde{\Theta}|_{\Omega_s\times \mathbb{T}^2}.
\end{equation*}

We denote the uniquely invariant probability measure on the Arnoux-Rauzy subshift  $(\Omega_s,\theta)$ by $\mathbb{P}$. Notice that the support of $\mathbb{P}$, named $\mathrm{supp}(\mathbb{P})$, is a $\theta$-invariant closed subset. By the minimality of $\theta$, we have $\mathrm{supp}(\mathbb{P})=\Omega_{s}$ or $\Omega_{s}=\varnothing$. However, $\mathrm{supp}(\mathbb{P})=\varnothing$ implies that $\mathbb{P}$ is zero, which leads to a contradiction. Hence, we conclude that $\mathrm{supp}(\mathbb{P})=\Omega_s$, and every nonempty open subset of $\Omega_s$ is assigned positive $\mathbb{P}$-measure. The following conclusion follows directly from the above reasoning.

\begin{cor}
	Let $\varphi\in C(\Omega_s\times\mathbb{T}^2,\mathbb{R}),\mathbb{P}$ denote the uniquely $\theta$-invariant measure on $\Omega_s$. Assume that $I_{\varphi}\neq\varnothing$, the for $\mathbb{P}$-a.e. $\omega\in\Omega_s$, we have
	\begin{enumerate}
		\item $h_{top}(F, I_{\varphi}(\omega),\omega)=h_{top}(F);$
		\item $I_{\varphi}(\omega)$ is residual in $\mathbb{T}^2$.
	\end{enumerate}
	Moreover, $I_{\varphi}$ is dense in $\Omega_s\times \mathbb{T}^2$.
\end{cor}

We note that our two main examples given in Sec \ref{subsection application}, including \emph{fiber Anosov maps on 2-dimensional tori} and \emph{random composition of $2\times 2$ area-preserving positive matrices driven by Arnoux-Rauzy subshifts}, satisfy the conditions of corollary \ref{corollary dense}. For these two examples, our corollary \ref{corollary dense} implies that the irregular set of any continuous observable is either empty or dense in $\Omega\times M$. There are some innovations need to be pointed out. If one regards the skew product system $(\Omega\times M,\Theta)$ as a deterministic system, then our results can not follow directly from results in \cite{LiWu2014}. This is because \emph{fiber Anosov maps on 2-dimensional tori} does not have specification property (see more details in subsection \ref{subsection application}). Besides, if $\theta$ is a diffeomorphism on a compact differentiable manifold $\Omega$ such that the contraction and the expansion of $D\theta$ are weaker than those of $DF_\omega$, then the skew product system $(\Omega\times M,\Theta)$  is actually a partially hyperbolic system, see \cite[Proposition 5.1]{liu2021}.

This paper is organized as follows.  Section \ref{sec preliminary}  briefly recalls the fiber specification property, notions of fiber topological entropy and fiber measure-theoretical entropy.  Section \ref{sec proof} and  section \ref{sec proof residual} address the proof of Theorem \ref{thm irregular set} and Theorem \ref{thm residual} respectively. The proof of all lemmas and corollaries is collected in section \ref{sec proof of lemmas}. 


We fix some notation here:
\begin{equation*}
	\begin{split}	
		\mathcal{M}(M)&:\mbox{the set of all Borel probability measures on $M$}.\\
		\mathcal{M}(\Omega\times M)&:\mbox{the set of all Borel probability measures on $\Omega\times M$}.\\
		\mathcal{M}_{\mathbb{P}}(\Omega\times M)&: \mbox{the set of all Borel probability measures on $\Omega\times M$ with marginal $\mathbb{P}$ on $\Omega$. }\\
		I_{\Theta}(\Omega\times M)&:\mbox{the set of all $\Theta$-invariant Borel probability measures on $\Omega\times M$.}\\
		I_{\Theta}^e(\Omega\times M)&:\mbox{the set of all ergodic $\Theta$-invariant Borel probability measures on $\Omega\times M$.}
	\end{split}
\end{equation*}
Notice that all above spaces are endowed with weak* topology.


\section{Preliminary Lemmas}\label{sec preliminary}

\subsection{Fiber specification}\label{subsection fiber specification}
In this subsection, we introduce the definition of fiber specification property, which is extremely useful to find some specific point shadowing pieces of orbit segments within any given precision. The orbit-gluing techniques provided by the fiber specification are utilized to construct fiber Moran-like fractals contained in the irregular set, see \ref{sec proof}. Recall that $\Theta$ is the skew product over product space $\Omega\times M$.

\begin{definition}[Fiber specification]\label{def fiber specification}
	For any $\omega\in\Omega$, an $\omega$-specification $S_{\omega}=(\omega,\tau, P_{\omega})$ consists of a finite collection of intervals $\tau=\{I_1,\cdots,I_k\}$, where $I_i=[a_i,b_i]\subset\mathbb{Z}$, and a map $P_{\omega}:\cup_{i=1}^k I_i\to M$, such that for any $t_1,t_2\in I\in\tau$, we have
	\begin{displaymath}
		F_{\theta^{t_1}\omega}^{t_2-t_1}(P_{\omega}(t_1))=P_{\omega}(t_2).
	\end{displaymath}
	An $\omega$-specification $S_{\omega}$ is called $m$-spaced if $a_{i+1}>b_i+m$ for every $i\in\{1,\dots,k-1\}$.
	
	The system $\Theta$ is said to have the fiber specification property if for any $\epsilon>0$, there exists some $m=m(\epsilon)>0$, such that for any $\omega\in\Omega$ and any $m$-spaced $\omega$-specification $S_{\omega}=(\omega,\tau, P_{\omega})$, we can find a point $x\in M$ such that $S_{\omega}$ is $(\omega,\epsilon)$-shadowed by $x$, i.e.
	\begin{equation*}
		d_M(P_{\omega}(t),F_{\omega}^t(x))<\epsilon,\quad\forall t\in I\in \tau.
	\end{equation*}
\end{definition}

For our target system $(\Omega\times M,\Theta)$, the fiber specification property has already established in \cite{variationalprincipleforrandomanosovsystems}. For the sake of completeness, we give the statement of fiber specification theorem here.

\begin{lemma}{\cite[Theorem A]{variationalprincipleforrandomanosovsystems}}\label{thm fiber specification}
	Let $(\Omega\times M,\Theta)$ be a skew product transformation satisfying Anosov and topological mixing on fiber properties and driven by a uniquely ergodic homeomorphism $(\Omega,\theta)$, then  $\Theta$ has the fiber specification property in definition \ref{def fiber specification}.
\end{lemma}

\subsection{Fiber Bowen's topological entropy and classical fiber topological entropy}\label{subsection fiber topological entropy}In this subsection, we generalize the concept of Bowen's topological entropy to non-compact or non-invariant sets for skew product transformations. This is a special case of Carath\'eodory dimension structure. See the  definition of Carath\'eodory dimension structure and more details in \cite[Chapter 1]{Pesinbook}. We will compare the definition of fiber Bowen's topological entropy with the classical definition of fiber topological entropy in remark \ref{remark comparsion}. See the classical definition of fiber topological entropy in \cite[Definition 1.2.3]{KL06}.

For any $n\in\mathbb{N}$ and $\omega\in\Omega$, define the fiber Bowen's metric by
\begin{equation*}
	d_{\omega}^n(x,y)=\max_{0\leq i\leq n-1}\{d_{M}(F_{\omega}^ix,F_{\omega}^iy)\},\quad\forall x,y\in M.
\end{equation*}
For any $\epsilon>0$, we denote by $\mathcal{B}_n(\omega,x,\epsilon)$ the open ball of radius $\epsilon$ around $x$ in the $d_{\omega}^n$ metric, i.e.
\begin{equation*}
	\mathcal{B}_n(\omega,x,\epsilon)=\{y\in M:d_{\omega}^n(x,y)<\epsilon\}.
\end{equation*}
Similarly, we denote by $\overline{\mathcal{B}}_n(\omega,x,\epsilon)$ the closed ball of radius $\epsilon$ around $x$ in the $d_{\omega}^n$ metric, i.e.
\begin{equation*}
	\overline{\mathcal{B}}_n(\omega,x,\epsilon)=\{y\in M:d_{\omega}^n(x,y)\leq \epsilon\}.
\end{equation*}

Suppose we are given some subset $Z\subset M$. For any $\epsilon>0$ and $s\in\mathbb{R}$, on the fiber $\{\omega\}\times M$, we define
\begin{equation}\label{def of m(Z,s,omega,N,epsilon)}
	m(Z,s,\omega,N,\epsilon)=\inf_{\Gamma^{\epsilon}_{\omega}}\sum_{i}e^{-sn_i},
\end{equation}
where the infimum is taken over all finite or countable collections $\Gamma_{\omega}^\epsilon=\{\mathcal{B}_{n_i}(\omega,x_i,\epsilon)\}_{i}$
with $Z\subset\cup_{i}\mathcal{B}_{n_i}(\omega,x_i,\epsilon)$ and $\min\{n_i\}\geq N.$ Note that $m(Z,s,\omega,N,\epsilon)$ does not decrease as $N$ increases, hence the following limit exists
\begin{equation}\label{def of m(Z,s,omega,epsilon)}
	m(Z,s,\omega,\epsilon)=\lim_{N\to\infty}m(Z,s,\omega,N,\epsilon).
\end{equation}
By standard arguments, we can show the existence of a critical value of the parameter $s$, named $h_{top}(F,Z,\omega,\epsilon)$, such that
\begin{equation}\label{def of htop(F,Z,omega,epsilon)}
	m(Z,s,\omega,\epsilon)=
	\begin{cases}
		+\infty,&\mbox{if}\ s<h_{top}(F,Z,\omega,\epsilon),\\ 0,&\mbox{if}\ s>h_{top}(F,Z,\omega,\epsilon).
	\end{cases}
\end{equation}
Note that $h_{top}(F,Z,\omega,\epsilon)\in [0,+\infty)$, and $m(Z,h_{top}(F,Z,\omega,\epsilon),\omega,\epsilon)$ could be $+\infty,0$ or some positive finite number.
\begin{lemma}\cite[Lemma 1.1]{variationalprincipleforrandomanosovsystems}\label{lemma htop exists}
	The value $h_{top}(F,Z,\omega,\epsilon)$ does not decrease as $\epsilon$ decreases. Therefore, the following limit exists
	\begin{equation}\label{def of htop(F,Z,omega)}
		h_{top}(F,Z,\omega)=\lim_{\epsilon\to 0}h_{top}(F,Z,\omega,\epsilon)=\sup_{\epsilon>0} h_{top}(F,Z,\omega,\epsilon).
	\end{equation}
\end{lemma}

The quantity $h_{top}(F,Z,\omega)$ is called \emph{the fiber Bowen's topological entropy of $F$ restricted on $Z$ and the fiber $\{\omega\}\times M$}.

We give some basic properties of fiber Bowen's topological entropy $h_{top}(F,\cdot,\omega).$ 

\begin{lemma}{\cite[Theorem 1.1]{Pesinbook}}\label{lemma topological entropy property}
	The fiber Bowen's topological entropy $h_{top}(F,\cdot,\omega)$ has the following properties
	\begin{enumerate}[(1)]
		\item $h_{top}(F,Z_1,\omega)\leq h_{top}(F,Z_2,\omega)$ for any $Z_1\subset Z_2\subset M$;
		\item $h_{top}(F,Z,\omega)=\sup_ih_{top}(F,Z_i,\omega)$, where $Z=\cup_{i=1}^{\infty}Z_i\subset M$.
	\end{enumerate}
\end{lemma}

Let us recall the classical definition of fiber topological entropy. Readers are referred to \cite[Section 1.2]{KL06} for more details. A subset $Q\subset M$ is called an \emph{$(\omega,\epsilon,n)$-separated set} of $M$ if for any two different points $x,y\in Q, d_{\omega}^n(x,y)>\epsilon$. For any $n\geq 1$ and $\epsilon>0$, we denote by $Z(\omega,\epsilon,n)$ the largest cardinality of any $(\omega,\epsilon,n)$-separated set of $M$.
\begin{definition}[Fiber topological entropy]
	The fiber topological entropy of $F$ $($or the relative topological entropy of $\Theta)$ is defined by
	\begin{equation}
		h_{top}(F)= h_{top}^{(r)}(\Theta)=\lim_{\epsilon\to 0}\limsup_{n\to\infty}\frac{1}{n}\int\log Z(\omega,\epsilon,n)\dif\mathbb{P}.
	\end{equation}
	In particular, if $\mathbb{P}$ is ergodic with respect to $\theta$, then 
	\begin{displaymath}
		h_{top}(F)=h_{top}^{(r)}(\Theta)=\lim_{\epsilon\to 0}\limsup_{n\to\infty}\frac{1}{n}\log Z(\omega,\epsilon,n).
	\end{displaymath}
\end{definition}

\begin{rem}
	We note that $F_{\omega}\in \mathrm{Diff}^2(M)$ and continuously depends on $\omega$. By the Margulis-Ruelle inequality of random dynamical systems \cite[Theorem 1]{Rulleinequality}, in the settings of this paper, one has
	\begin{equation}\label{entropy is finite}
		h_{top}(F)<\infty.
	\end{equation}
\end{rem}

\begin{rem}\label{remark comparsion}
Let us compare the definition of \emph{fiber Bowen's topological entropy} with the classical definition of \emph{fiber topological entropy}. Note that the  discussion are confined in our target systems. For the classical definition of fiber topological entropy, the following relativized variational principle has already been established (the proof can be found in \cite[Proposition 2.2]{Kifervariationalprinciple})
\begin{equation}\label{classical variational principle}
	\begin{split}
		h_{top}(F)&=\sup\{h_{\mu}(F):\mu\in I_{\Theta}(\Omega\times M)\cap \mathcal{M}_{\mathbb{P}}(\Omega\times M)\}\\&=\sup\{h_{\mu}(F):\mu\in I_{\Theta}(\Omega\times M)\},
	\end{split}
\end{equation}
where $h_{\mu}(F)$ refers to the fiber measure-theoretical entropy of $F$ (see subsection \ref{subsection fiber measure entropy}) and the second equality is due to the unique ergodicity of the driven system $(\Omega,\mathbb{P},\theta)$. Besides, the variational principle between the fiber Bowen's topological entropy and fiber measure-theoretical entropy is a special case of \cite[Theorem B]{variationalprincipleforrandomanosovsystems}. To be more specific, by taking continuous observable $\varphi=0$, we have for $\mathbb{P}$-a.e. $\omega\in\Omega$ the following equality
\begin{displaymath}
	h_{top}(F,M,\omega)=\sup\{h_{\mu}(F):\mu\in I_{\Theta}(\Omega\times M)\}.
\end{displaymath}Therefore, we conclude that the definition of fiber Bowen's topological entropy and classical fiber topological entropy coincides, i.e., for $\mathbb{P}$-a.e. $\omega\in\Omega$, the following equality holds
\begin{equation}\label{equation whole fiber Bowen equals to fiber}
	h_{top}(F,M,\omega)=h_{top}(F).
\end{equation}
\end{rem}

For simplicity, \emph{fiber topological entropy} always refers to the classical definition in \cite{KL06}, and our definition of fiber topological entropy induced by Carathe\'odory structure is always called \emph{fiber Bowen's topological entropy}.

\subsection{Fiber measure-theoretical entropy and upper semi-continuity of entropy map}\label{subsection fiber measure entropy}
In this subsection, we briefly introduce the concept of fiber measure-theoretical entropy, and state the semi-continuity of the entropy map for our target systems.

For any $\mu\in I_{\Theta}(\Omega\times M)$, we denote by $h_{\mu}(\Theta)$ the classical measure-theoretical entropy of dynamical system $(\Omega\times M,\Theta,\mu)$. Denote by $\pi_{\Omega}:\Omega\times M\to\Omega$ the projection onto $\Omega$.

Let $\mathcal{R}=\{R_i\}\subset \mathcal{B}_{\mathbb{P}}(\Omega)\otimes\mathcal{B}(M)$ be a finite or countable partition of $\Omega\times M$, then $\mathcal{R}(\omega)=\{R_i(\omega)\}\subset \mathcal{B}(M)$ is a finite or countable partition of $M$, where $R_i(\omega)=\{x\in M:(\omega,x)\in R_i\}$. For $\mu\in I_{\Theta}(\Omega\times M)$, the conditional entropy of $\mathcal{R}$ given $\sigma$-algebra $\pi_{\Omega}^{-1}(\mathcal{B}_{\mathbb{P}}(\Omega))$ is defined by
\begin{equation*}
	H_{\mu}(\mathcal{R}|\pi_{\Omega}^{-1}(\mathcal{B}_{\mathbb{P}}(\Omega)))=\int_{\Omega} H_{\mu_\omega}(\mathcal{R}(\omega))\dif\mathbb{P}(\omega),
\end{equation*}
where $H_{\mu_{\omega}}(\mathcal{A})$ denotes the usual entropy of $\mathcal{A}$, and $\mu\mapsto\mu_{\omega}$ is the disintegration of $\mu$ with respect to $\mathbb{P}$. \emph{The fiber measure-theoretical entropy of $F$ $($or the relative measure theoretical entropy of $\Theta)$ with respect to $\mu\in I_{\Theta}(\Omega\times M)$} is defined by
\begin{equation*}
	h_{\mu}(F)=h_{\mu}^{(r)}(\Theta)=\sup_{\mathcal{Q}}h_{\mu}(F,\mathcal{Q}),
\end{equation*}
where
\begin{equation*}
	h_{\mu}(F,\mathcal{Q}):=\lim_{n\to\infty}\frac{1}{n}H_{\mu}\left(\bigvee_{i=0}^{n-1}(\Theta)^{-i}\mathcal{Q}\Bigg|\pi^{-1}_{\Omega}(\mathcal{B}_{\mathbb{P}}(\Omega))\right),
\end{equation*}
and the supreme is taken over all finite measurable partitions $\mathcal{Q}$ of $\Omega\times M$.

\begin{rem}
	$\quad$
	\begin{enumerate}
		\item Note the we have $h_{top}(F)<\infty$ in \eqref{entropy is finite} for our target systems. Therefore, according to the relativized variational principle \eqref{classical variational principle}, for any $\mu\in I_{\Theta}(\Omega\times M)$, we have
		\begin{equation*}
			h_{\mu}(F)<\infty.
		\end{equation*}
		\item Besides, we also have the following equality linking the fiber measure-theoretical entropy $h_{\mu}(F)$ of any $\mu\in I_{\Theta}(\Omega\times M)$ with the  measure-theoretical entropy $h_{\mu}(\Theta)$, i.e.
		\begin{equation}\label{Rohlin formula}
			h_{\mu}(\Theta)=h_{\mu}(F)+h_{\mathbb{P}}(\theta),
		\end{equation}which is the celebrated Abramov-Rohlin formula, see \cite{Rohlinformula}.
	\end{enumerate}
	
\end{rem}

To introduce the upper semi-continuity of the entropy map, we need the following lemma. 
\begin{lemma}\cite[Lemma 3.3]{HLL}\label{lemma fiber expansive}
	The target system $\Theta$ is fiber-expansive, i.e., there exists some constant $\eta>0$ such that for any $\omega\in\Omega$, if $d_{M}(F_{\omega}^nx,F_{\omega}^ny)<\eta$ for all $n\in\mathbb{Z}$, then $x=y$. The constant $\eta$ is called a fiber-expansive constant.
\end{lemma}

We have known that our target system $\Theta$ is fiber-expansive. Therefore, the following lemma is a direct corollary of \cite[Theorem 1.3.5]{KL06}.
\begin{lemma}[Entropy map is upper semi-continuous]\label{lemma upper semi-continuous}
	The entropy map of our target system, i.e.
	\begin{displaymath}
		\mu\longmapsto h_{\mu}(F)\mbox{ for }\mu\in I_{\Theta}(\Omega\times M)
	\end{displaymath}
is upper semi-continuous with respect to the weak* topology on $I_{\Theta}(\Omega\times M)$. In particular, there exists a measure of maximal fiber measure-theoretical entropy, named $\mu_0$.
\end{lemma}

\section{Fiber Bowen's topological entropy of irregular set}\label{sec proof}
In this section, we give the proof of theorem \ref{thm irregular set}. The proof of all lemmas are collected in section \ref{sec proof of lemmas}.

Without loss of generality, we assume that $I_{\varphi}(\omega^*)\neq\varnothing$ for some $\omega^*\in\Omega$, i.e., there exists some $x^*\in M$ such that
\begin{equation}\label{irregular set is nonempty}
	(\omega^*,x^*)\in I_{\varphi}.
\end{equation}  Therefore, we remain to prove that for $\mathbb{P}$-a.e. $\omega\in\Omega,$ the $\omega$-fiber irregular set $I_{\varphi}(\omega)$ is nonempty and the irregular set $I_{\varphi}$ carries full fiber  topological entropy, i.e.
\begin{displaymath}
	h_{top}(F, I_{\varphi}(\omega),\omega)= h_{top}(F).
\end{displaymath}

According to \eqref{equation whole fiber Bowen equals to fiber}, $h_{top}(F)$ coincides with $h_{top}(F,M,\omega)$ for $\mathbb{P}$-a.e. $\omega\in\Omega$. Notice that  $I_{\varphi}(\omega)$ is a subset of $M$. Therefore, by applying lemma \ref{lemma topological entropy property}, one has for $\mathbb{P}$-a.e. $\omega\in\Omega$
\begin{displaymath}
h_{top}(F, I_{\varphi}(\omega),\omega)\leq h_{top}(F,M,\omega)=h_{top}(F).
\end{displaymath}Hence, it suffices to prove that $h_{top}(F, I_{\varphi}(\omega),\omega)\geq h_{top}(F)$ for $\mathbb{P}$-a.e. $\omega\in\Omega$. Combining lemma \ref{lemma upper semi-continuous} with \eqref{classical variational principle}, there exists a measure $\mu_0$ with maximal fiber measure-theoretical entropy, i.e.
\begin{equation}\label{def of hmu0}
	h_{top}(F)= h_{\mu_0}(F).
\end{equation}
Hence, it remains to prove that  for $\mathbb{P}$-a.e. $\omega\in\Omega$
\begin{equation}\label{htop is greater that hmu0}
	h_{top}(F,I_{\varphi}(\omega),\omega)\geq h_{\mu_0}(F).
\end{equation}

For any $\varphi\in C(\Omega\times M,\mathbb{R}),\alpha\in\mathbb{R},\delta>0,\omega\in\Omega$ and $n\in\mathbb{N}$, the \emph{fiber deviation set} $P(\alpha,\delta,n,\omega)$ is defined by
\begin{equation*}
	P(\alpha,\delta,n,\omega)=\left\{x\in M:\left|\frac{1}{n}\sum_{i=0}^{n-1}\varphi(\Theta^i(\omega,x))-\alpha\right|<\delta\right\}.
\end{equation*}
For any $\epsilon>0$, denote by $M(\alpha,\delta,n,\epsilon,\omega)$ the largest cardinality of $(\omega,\epsilon,n)$-separated set in the fiber deviation set $P(\alpha,\delta,n,\omega)$.

The standard techniques to prove \eqref{htop is greater that hmu0} is to construct Moran-like fractals contained in $I_{\varphi}(\omega)$ and apply entropy distribution argument (for example, see \cite[Section 5]{TakensVerbitsky2003variationalprinciple}). However, we will not just follow this standard procedure in the remaining proof. Our innovation is the following lemma \ref{lemma deviation set}, which is devoted to investigate how to construct some appropriate deviation set and estimate the largest cardinality of specific separated set in it. The proof can be found in \cite[Subsection 5.3]{variationalprincipleforrandomanosovsystems}.
\begin{lemma}[Fiber deviation set lemma]\label{lemma deviation set}
Given $\varphi\in C(\Omega\times M,\mathbb{R})$ and $\mu\in I_{\Theta}(\Omega\times M)$, let $\alpha=\int\varphi \dif\mu$. For any $\gamma>0$ and positive number $\delta<\frac{1}{2}\min\{\gamma,\eta\}$, there exists a $\mathbb{P}$-full measure set $\Omega_{\delta}$ such that for any $N\in\mathbb{N},$ there exists a measurable function $\hat{n}:\Omega_{\delta}\to\mathbb{N}$ satisfying
\begin{enumerate}
	\item $\hat{n}$ has a lower bound: for any $\omega\in\Omega_{\delta}$, we have $\hat{n}(\omega)\geq N$;
	\item  fiber deviation set is nonempty: for any $\omega\in\Omega_{\delta}, P(\alpha, 4\delta,\hat{n}(\omega),\omega)\neq\varnothing$;
	\item estimation for the largest cardinality of $(\omega,\hat{n}(\omega),\eta/2)$-separated set in the fiber deviation set:
	\begin{equation}\label{cardinality of fiber deviation set}
		\frac{1}{\hat{n}(\omega)}\log M\left(\alpha,4\delta,\hat{n}(\omega),\frac{\eta}{2},\omega\right)\geq h_{\mu}(F)-4\gamma.
	\end{equation}
\end{enumerate}
\end{lemma}

\begin{rem}\label{remark difficultis and methods}
The main difference between proof for skew product transformations and proof for deterministic systems lies on the length function $\hat{n}$. In \emph{deterministic} cases, length functions, corresponding to length of orbits of points in separated sets of deviation sets, are constant on each level of Moran fractals. Therefore, these length functions are automatically bounded. However, things become totally different in cases of \emph{skew product systems}. Notice that it's proven that $\hat{n}$ is only measurably depends on $\omega$. Hence, there is no hope to bound these length functions if one just follow classical constructions of Moran fractals in \cite{TakensVerbitsky2003variationalprinciple}. By utilizing Birkhoff's ergodic theorem and Lusin's theorem, one can see that our construction of fiber Moran-like fractals in subsection \ref{subsetion construction of fiber Moran-like fractals} is essentially different from those in \cite{TakensVerbitsky2003variationalprinciple,Thompson2009Variationalprinciplepressure}.

\end{rem}
We continue our proof with the following lemma. The proof of the lemma is addressed in section \ref{sec proof of lemmas}. Recall that $\mu_0$ is the measure of maximal fiber measure-theoretical entropy in \eqref{def of hmu0}.
\begin{lemma}\label{lemma different integral values}
	Given $\varphi\in C(\Omega\times M,\mathbb{R})$, for any $\gamma>0$, there exists some $\mu_1\in I_{\Theta}(\Omega\times M)$ such that
	\begin{enumerate}
		\item $\varphi$ has different integral values with respect to $\mu_0$ and $\mu_1$, i.e.
		\begin{displaymath}
			\alpha_0=\int\varphi\dif\mu_0\neq\int\varphi\dif\mu_1=\alpha_1,
		\end{displaymath}
	where $\mu_0$ is a measure with maximal fiber measure-theoretical entropy.
		\item the fiber measure-theoretical entropies of $\mu_1$ and $\mu_0$ can be arbitrarily close, i.e.
		\begin{displaymath}
		h_{\mu_0}(F)-\gamma\leq h_{\mu_1}(F)\leq h_{\mu_0}(F).
		\end{displaymath}
	\end{enumerate}
\end{lemma}
\begin{rem}
	One of the difficulties in our proof is that measures $\mu_0$ and $\mu_1$ in lemma \ref{lemma different integral values} are only $\Theta$-invariant, and we can not guarantee that they are ergodic. While in the proof of irregular sets carry full topological entropy for systems with specification property (see \cite{Thompson2010irregularspecification}) or with almost specification property (see \cite{Thompson2012almostspecification}), $\mu_0$ and $\mu_1$ can be chosen to be ergodic since ergodic measures are entropy-dense for systems with approximate product property(see \cite[Definition 2.7, Theorem 2.1
	]{PfisterSullivanlargedeviationwithoutspecification2005}).
	
	The advantage of choosing ergodic measures is that they can make the construction of Moran fractals easier. To be more specific, each ergodic measure has generic point. Therefore, one can construct irregular points by shadowing pieces of orbits of generic points consecutively. However, for our target systems, we can not guarantee that ergodic measures are entropy dense. Therefore, for $\Theta$-invariant measures $\mu_0$ and $\mu_1$, one must utilize the \emph{fiber deviation set lemma} to construct irregular points.
\end{rem}

 The main proof are divided into several steps. In the following subsections, we will formulate our proof and construction in details.
\subsection{Construction of fiber Moran-like fractals}\label{subsetion construction of fiber Moran-like fractals}
We begin our construction of points in the irregular set in the following proof, which is inspired by the construction of Moran fractals in \cite{TakensVerbitsky2003variationalprinciple} and the construction of irregular points in \cite{Thompson2010irregularspecification}.

For simplicity, we define $\rho(k):\mathbb{N}\to\{0,1\}$ by $\rho(k)\equiv k(\mathrm{mod}\ 2)$.
Let $\{\delta_k\}_{k=1}^{\infty}$ be a sequence of positive numbers strictly decreasing to $0$, where $\delta_1<\frac{1}{2}\min\{\gamma,\eta\}$. By theorem \ref{thm fiber specification}, there exists a sequence $\{m_k\}_{k=1}^{\infty}\subset\mathbb{N}$, such that any $m_k=m(\eta/2^{4+k})$-spaced $\omega$-specification is $(\omega,\eta/{2^{4+k}})$-shadowed by some point in $M$. Without loss of generality, we assume that
\begin{equation}\label{mk tends to infinity}
\mbox{$m_k\to\infty$  as $k\to\infty$.}
\end{equation}

From now on, we fix some sufficiently small $\gamma>0$. By utilizing lemma \ref{lemma deviation set} and lemma \ref{lemma different integral values}, we have a sequence of $\mathbb{P}$-full measure set $\{\Omega_{\delta_k}\}_{k=1}^{\infty}$ and a sequence of measurable functions $\hat{n}_k:\Omega_{\delta_k}\to\mathbb{N}$, such that
\begin{equation}\label{lower bound of nk hat}
	\hat{n}_k\geq 2^{m_k}.
\end{equation}
Furthermore, for any $k\in\mathbb{N}$ and $\omega\in\Omega_{\delta_k}$, we have
\begin{equation}
	\begin{cases}
		P(\alpha_{\rho(k)},4\delta_k,\hat{n}_k(\omega),\omega)\neq\varnothing,\\\dfrac{1}{\hat{n}_k(\omega)}\log M(\alpha_{\rho(k)},4\delta_k,\hat{n}_k(\omega),\omega)\geq\min\{h_{\mu_0}(F)-4\gamma,h_{\mu_1}(F)-4\gamma\}\geq h_{\mu_0}(F)-5\gamma.
	\end{cases}\label{estimation for odd even k}
\end{equation}
Note also that $\hat{n}_k:\cap_{j\geq 1}\Omega_{\delta_j}\to\mathbb{N}$ is measurable and $\mathbb{P}(\cap_{j\geq 1}\Omega_{\delta_j})=1$.

Let $\{\xi_k\}_{k=1}^\infty$ be a sequence of positive numbers strictly decreasing to 0 with $\xi_1<1$. Now according to Lusin's theorem, for any $k\in\mathbb{N},$ there exists a compact set  $\Omega_k\subset\cap_{j\geq 1}\Omega_{\delta_j}$ with $\mathbb{P}(\Omega_k)\geq 1-{\xi_k}/{2}$, such that $\hat{n}_k(\omega)$ is continuous on $\Omega_k$ . Denote
\begin{equation}\label{nk hat has maximum}
	\hat{n}_k^M=\max_{\omega\in\Omega_k}\hat{n}_k(\omega).
\end{equation}
By Birkhoff's ergodic theorem, there exists a $\mathbb{P}$-full measure set $\Omega_k^\prime,$ such that for any $\omega\in\Omega_k^\prime$, we have
\begin{equation}\label{probability of Omega k is greater than}
	\lim_{n\to\infty}\frac{1}{n}\sum_{i=0}^{n-1}1_{\Omega_k}(\theta^i\omega)=\mathbb{P}(\Omega_k)\geq 1-\frac{\xi_k}{2}.
\end{equation}
We define
\begin{equation}\label{Omega gamma hat}
\widehat{\Omega}_\gamma:=\cap_{k\geq1}\Omega_k^\prime,
\end{equation} in the left of this section, we are going to prove theorem \ref{thm irregular set} by constructing Moran-like fractals contained in the irregular set on each fiber $\{\omega\}\times M$ for $\omega\in\widehat{\Omega}_\gamma$. Note that $\widehat{\Omega}_{\gamma}$ is a $\mathbb{P}$-full measure as a countable intersection of $\mathbb{P}$-full measure sets.

From now on, we fix any $\omega\in\widehat{\Omega}_\gamma$.

\emph{Step 1. Construction of intermediate sets.} We start by choosing two sequences of positive integers depending on $\omega$, named $\{L_k(\omega)\}_{k\in\mathbb{N}}$ and $\{N_k(\omega)\}_{k\in\mathbb{N}}.$

We define $\{L_k(\omega)\}_{k\in\mathbb{N}}$ inductively. For $k=1,$ by (\ref{probability of Omega k is greater than}), we can pick $L_1(\omega)\in\mathbb{N}$ such that
\begin{equation}\label{Omega1,xi1}	\frac{1}{n}\sum_{i=0}^{n-1}1_{\Omega_1}(\theta^i\omega)\geq 1-\xi_1,\quad\forall n\geq L_1(\omega).
\end{equation}
Once $L_k(\omega)$ is defined, by (\ref{probability of Omega k is greater than}), we can pick $L_{k+1}(\omega)$ satisfying $L_{k+1}(\omega)> L_k(\omega),$ and
\begin{equation}\label{pick L k+1}
	\dfrac{1}{n}\displaystyle\sum_{i=0}^{n-1}1_{\Omega_{k+1}}(\theta^i\omega)\geq 1-\xi_{k+1},\quad\forall n\geq L_{k+1}(\omega).
\end{equation}
Now, let $N_1(\omega)=L_2(\omega)$, and we pick $N_k(\omega)\in\mathbb{N}$ such that
\begin{equation}\label{construction of Nk}
	\begin{cases}
		N_k(\omega)\geq \max\{2^{\hat{n}_{k+1}^M+m_{k+1}},L_{k+1}(\omega)\}, &\forall k\geq 2,\\ N_{k+1}(\omega)\geq 2^{\sum_{j=1}^k(N_j(\omega)\cdot(\hat{n}_j^M+m_j)\cdot\prod_{i=j}^{k}(1-\xi_i)^{-1})},&\forall k\geq 1.
	\end{cases}
\end{equation}The (\ref{construction of Nk}) ensures that
\begin{align*}
	\lim_{k\to\infty}\frac{\hat{n}_{k+1}^M+m_{k+1}}{N_k(\omega)} =0,\quad \lim_{k\to \infty}\frac{\sum_{j=1}^k(N_j(\omega)\cdot(\hat{n}_j^M+m_j)\cdot\prod_{i=j}^{k}(1-\xi_i)^{-1})}{N_{k+1}(\omega)}=0.
\end{align*}

First, we construct intermediate set $D_1(\omega)$. Let $l^{0,1}(\omega)$ be the first integer $i\geq 0$ such that $\theta^i\omega\in\Omega_1.$ We denote $T_0^1(\omega)=l^{0,1}(\omega).$ Let $l_1^1(\omega)$ be the first integer $i\geq 0$ such that
\begin{displaymath}
	\theta^{T^1_0(\omega)+\hat{n}_1(\theta^{T^1_0(\omega)}\omega)+m_1+i}\omega\in\Omega_1,
\end{displaymath}
where $\hat{n}_1(\theta^{T_0^1(\omega)}\omega)$ is well defined and bounded by $\hat{n}_1^M$ since $\theta^{T_0^1(\omega)}\omega\in\Omega_1.$ Denote $T_1^1(\omega)=T_0^1(\omega)+\hat{n}_1(\theta^{T_0^1}\omega)+m_1+l_1^{1}(\omega)$. For $k\in\{1,\dots,N_1(\omega)-2\}$, suppose $T_k^1(\omega)$ is already defined and $\theta^{T_k^1(\omega)}\omega\in\Omega_1$, we let $l_{k+1}^1(\omega)$ be the first integer $i\geq 0$ such that
\begin{displaymath} \theta^{T_k^1(\omega)+\hat{n}_1(\theta^{T_k^1(\omega)}\omega)+m_1+i}\in\Omega_1,
\end{displaymath}
and denote $$T_{k+1}^1(\omega)=T_k^1(\omega)+\hat{n}_1({\theta^{T_k^1}\omega})+m_1+l_{k+1}^1(\omega).$$ Finally, we define $$T_{N_1(\omega)}^1(\omega)=T_{N_1(\omega)-1}^1(\omega)+\hat{n}_1(\theta^{T^1_{N_1(\omega)-1}(\omega)}\omega).$$
We point out that  in orbit $\{\omega,\cdots,\theta^{T_{N_1(\omega)}^1(\omega)}\omega\}$ , there are at least $T_0^1(\omega)+l_1^1(\omega)+\cdots+l_{N_1(\omega)-1}^1(\omega)$ times not lying in $\Omega_1.$ We note that $T_{N_1(\omega)}^1(\omega)\geq N_1(\omega)\geq L_2(\omega)\geq L_1(\omega)$, therefore, by (\ref{Omega1,xi1}),
\begin{equation}\label{estimate 1 bad point}
	T_0^1(\omega)+l_1^1(\omega)+\cdots+l_{N_1(\omega)-1}^1(\omega)\leq \xi_1 T_{N_1(\omega)}^1(\omega).
\end{equation}

By our construction, we have $\theta^{T_k^1(\omega)}\omega\in\Omega_1\subset\cap_{j\geq 1}\Omega_{\delta_j}$ for $k\in\{0,\dots,N_1(\omega)-1\}$. Therefore, combining (\ref{lower bound of nk hat}), (\ref{estimation for odd even k}) and (\ref{nk hat has maximum}) together, one concludes that
\begin{displaymath}
\begin{cases}
	2^{m_1}\leq \hat{n}_1(\theta^{T_k^1(\omega)}\omega)\leq\hat{n}_1^M;\\
	P(\alpha_{\rho(1)},4\delta_1,\hat{n}_1(\theta^{T_k^1(\omega)}\omega),\theta^{T_k^1(\omega)}\omega)\neq\varnothing;\\ 	\dfrac{1}{\hat{n}_1(\theta^{T_k^1(\omega)}\omega)}\log M(\alpha_{\rho(1)},4\delta_1,\hat{n}_1(\theta^{T_k^1(\omega)}\omega),{\eta}/{2},\theta^{T_k^1(\omega)}\omega)\geq h_{\mu_0}(F)-5\gamma.
	\end{cases}
\end{displaymath}
Let $C_1(\theta^{T_k^1(\omega)}\omega)\subset P(\alpha_{\rho(1)},4\delta_1,\hat{n}_1(\theta^{T_k^1(\omega)}\omega),\theta^{T_k^1(\omega)}\omega)$ be some $(\theta^{T_k^1(\omega)}\omega,{\eta}/{2},\hat{n}_1(\theta^{T_k^1(\omega)}\omega))$-separated set with largest cardinality, i.e. $$\# C_1(\theta^{T_k^1(\omega)}\omega)=M\left(\alpha_{\rho(1)},4\delta_1,\hat{n}_1(\theta^{T_k^1(\omega)}\omega),\frac{\eta}{2},\theta^{T_k^1(\omega)}\omega\right).$$

Now for any $N_1(\omega)$-tuple $(x_0^1,\dots,x_{N_1(\omega)-1}^1)\in C_1(\theta^{T_0^1(\omega)}\omega)\times\cdots\times C_1(\theta^{T_{N_1(\omega)-1}^1(\omega)}\omega)$, by Theorem \ref{thm fiber specification}, there exists a point $y=y(x_0^1,\dots, x^1_{N_1(\omega)-1})\in M$, which $(\omega,{\eta}/{2^{4+1}})$-shadows pieces of orbits
\begin{displaymath}
	\begin{split}
		&\left\{x_0^1,F_{\theta^{T_0^1(\omega)}\omega}(x_0^1),\dots,F^{\hat{n}_1(\theta^{T_0^1(\omega)}\omega)-1}_{\theta^{T_0^1(\omega)}\omega}(x_0^1)\right\},\cdots,\\&\left\{x_{N_1(\omega)-1}^1,F_{\theta^{T_{N_1(\omega)-1}^1(\omega)}\omega}(x_{N_1(\omega)-1}^1),\dots,F^{\hat{n}_1(\theta^{T_{N_1(\omega)-1}^1(\omega)}\omega)-1}_{\theta^{T_{N_1(\omega)-1}^1(\omega)}\omega}(x_{N_1(\omega)-1}^1)\right\}
\end{split} \end{displaymath}
with gaps $m_1+l_1^1(\omega),m_1+l_2^1(\omega),\dots,m_1+l_{N_1-1}^1(\omega)$, i.e., for any $k\in\{0,\dots,N_1(\omega)-1\}$
\begin{equation}\label{eq shadowingproperty y}
	d_{\theta^{T_k^1(\omega)}\omega}^{\hat{n}_1(\theta^{T_k^1(\omega)}\omega)}\left(x_k^1, F_{\omega}^{T_k^1(\omega)}y\right)<\frac{\eta}{2^{4+1}}.
\end{equation}

Let $D_1(\omega)=\{y=y(x_0^1,\dots,x_{N_1(\omega)-1}^1):x_k^1\in C_1(\theta^{T_k^1(\omega)}\omega),k\in \{0,\dots,N_1(\omega)-1\}\}$, which is the first intermediate set. 
Similar as the proof of lemma 4.3 in \cite{variationalprincipleforrandomanosovsystems}, we can prove the following lemma.
\begin{lemma}\label{lemma separated metric D1}
	If $(x_0^1,\dots,x_{N_1(\omega)-1}^1), (z_0^1,\dots,z_{N_1(\omega)-1}^1)\in C_1(\theta^{T_0^1(\omega)}\omega)\times\cdots\times C_1(\theta^{T_{N_1(\omega)-1}^1(\omega)}\omega)$ are different $N_1(\omega)$-tuples, then
	\begin{displaymath}
		d_{\omega}^{T_{N_1(\omega)}^1(\omega)}\left(y(x_0^1,\dots,x_{N_1(\omega)-1}^1),y(z_0^1,\dots,z_{N_1(\omega)-1}^1)\right)> \frac{\eta}{2}-\frac{\eta}{2^{4+1}}\times 2\geq \frac{3\eta}{8}.
	\end{displaymath}
	Hence $\# D_1(\omega)=\prod_{i=0}^{N_1(\omega)-1}\# C_1(\theta^{T_i^1(\omega)}\omega).$
\end{lemma}

Next, we define $D_k(\theta^{T_0^k(\omega)}\omega)$ for $k\geq 2$ inductively. Suppose now that $T_{N_k(\omega)}^k(\omega)> N_k(\omega)$ is already defined for $k\geq 1,$ let $l^{k,k+1}(\omega)$ be the first integer $i\geq 0$ such that
\begin{equation}\label{gap larger than mk+1}
	\theta^{T_{N_k(\omega)}^k(\omega)+m_{k+1}+i}\omega\in\Omega_{k+1}.
\end{equation}
Denote $T_0^{k+1}(\omega)=T_{N_k(\omega)}^k(\omega)+m_{k+1}+l^{k,k+1}(\omega)$. Let $l_1^{k+1}(\omega)$ be the first integer $i\geq 0$ such that
\begin{displaymath}
	\theta^{T_0^{k+1}(\omega)+\hat{n}_{k+1}(\theta^{T_0^{k+1}(\omega)}\omega)+m_{k+1}+i}\omega\in \Omega_{k+1}
\end{displaymath}
and denote
\begin{displaymath}
	T_1^{k+1}(\omega)=T_0^{k+1}(\omega)+\hat{n}_{k+1}(\theta^{T_0^k(\omega)}\omega)+m_{k+1}+l_{1}^{k+1}(\omega).
\end{displaymath}
Once $T_{j}^{k+1}(\omega)$ is defined for $j\in\{1,\dots,N_{k+1}(\omega)-2\}$, we let $l_{j+1}^{k+1}(\omega)$ be the first integer $i\geq 0$ such that
\begin{displaymath}
	\theta^{T_j^{k+1}(\omega)+\hat{n}_{k+1}(\theta^{T_j^{k+1}(\omega)}\omega)+m_{k+1}+i}\omega\in\Omega_{k+1},
\end{displaymath}
and denote
\begin{displaymath}
	T_{j+1}^{k+1}(\omega)=T_j^{k+1}(\omega)+\hat{n}_{k+1}(\theta^{T_j^{k+1}(\omega)}\omega)+m_{k+1}+l_{j+1}^{k+1}(\omega).
\end{displaymath}For $j=N_{k+1}(\omega)$, we define
\begin{equation*}
	T_{N_{k+1}(\omega)}^{k+1}(\omega)=T_{N_{k+1}(\omega)-1}^{k+1}(\omega)+\hat{n}_{k+1}(\theta^{T_{N_{k+1}(\omega)-1}^{k+1}(\omega)}\omega).
\end{equation*}It is clear that $T_{N_{k+1}(\omega)}^{k+1}(\omega)>N_{k+1}(\omega)$ since $\hat{n}_{k+1}\geq 1$ and $m_{k+1}>0$. By construction,
in the orbit $\{\theta^{T_{N_k(\omega)}^k(\omega)}\omega,...,\theta^{T_{N_{k+1}(\omega)-1}^{k+1}(\omega)-1}\omega\}$, there are at least $l^{k,k+1}(\omega)+\sum_{i=1}^{N_{k+1}(\omega)-1}l_{i}^{k+1}(\omega)$ times not lying in $\Omega_{k+1}$ and $T_{N_k(\omega)}^k(\omega)\geq N_k(\omega)\geq L_{k+1}(\omega).$ Therefore \eqref{pick L k+1} implies
\begin{equation}\label{estimate of k bad point} \Big(l^{k,k+1}(\omega)+\sum_{i=1}^jl_{i}^{k+1}(\omega)\Big)\leq\xi_{k+1}T_{j}^{k+1}(\omega),\quad\forall j\in\{1,\dots,N_{k+1}(\omega)-1\}.
\end{equation}

By our construction, we have $\theta^{T^{k+1}_j(\omega)}\omega\in\Omega_{k+1}\subset\cap_{l\geq 1}\Omega_{\delta_l}$ for $j\in\{0,\dots,N_{k+1}(\omega)-1\}$. Therefore, combining (\ref{lower bound of nk hat}), (\ref{estimation for odd even k}) and (\ref{nk hat has maximum}), one concludes that
\begin{displaymath}
	\begin{cases}
		2^{m_{k+1}}\leq \hat{n}_{k+1}(\theta^{T_j^{k+1}(\omega)}\omega)\leq\hat{n}_{k+1}^M;\\
		P(\alpha_{\rho(k+1)},4\delta_{k+1},\hat{n}_{k+1}(\theta^{T_j^{k+1}(\omega)}\omega),\theta^{T_j^{k+1}(\omega)}\omega)\neq\varnothing;\\ 	\dfrac{1}{\hat{n}_{k+1}(\theta^{T_j^{k+1}(\omega)}\omega)}\log M(\alpha_{\rho(k+1)},4\delta_{k+1},\hat{n}_{k+1}(\theta^{T_j^{k+1}(\omega)}\omega),{\eta}/{2},\theta^{T_j^{k+1}(\omega)}\omega)\geq h_{\mu_0}(F)-5\gamma.
	\end{cases}
\end{displaymath}
For $j\in\{0,1,...,N_{k+1}(\omega)-1\}$, let $C_{k+1}(\theta^{T^{k+1}_j(\omega)}\omega)\subset P(\alpha_{\rho(k+1)},4\delta_{k+1},\hat{n}_{k+1}(\theta^{T_j^{k+1}(\omega)}\omega),\theta^{T_j^{k+1}(\omega)}\omega)$ be some $(\theta^{T_j^{k+1}(\omega)}\omega,{\eta}/{2},\hat{n}_{k+1}(\theta^{T_j^{k+1}(\omega)}\omega))$-separated set with largest cardinality, i.e.
\begin{displaymath}
	\# C_{k+1}(\theta^{T_j^{k+1}(\omega)}\omega)=M\left(\alpha_{\rho(k+1)},4\delta_{k+1},\hat{n}_{k+1}(\theta^{T_j^{k+1}(\omega)}\omega),\frac{\eta}{2},\theta^{T_j^{k+1}(\omega)}\omega\right).
\end{displaymath}

We begin to illustrate how to glue pieces of orbits in the remaining step. For any $N_{k+1}(\omega)$-tuple $$(x_0^{k+1},\dots,x_{N_{k+1}(\omega)-1}^{k+1})\in C_{k+1}(\theta^{T_0^{k+1}(\omega)}\omega)\times\cdots\times C_{k+1}(\theta^{T_{N_{k+1}(\omega)-1}^{k+1}(\omega)}\omega),$$there exists a point $y=y(x_0^{k+1},\dots,x_{N_{k+1}(\omega)-1}^{k+1})\in M$, which $(\theta^{T_0^{k+1}(\omega)}\omega,{\eta}/{2^{4+k+1}})$-shadows pieces of orbits
\begin{displaymath}
	\begin{split}
		&\left\{x_0^{k+1}, F_{\theta^{T_0^{k+1}(\omega)}\omega}(x_0^{k+1}),\dots,F^{\hat{n}_{k+1}(\theta^{T_0^{k+1}(\omega)}\omega)-1}_{\theta^{T_0^{k+1}(\omega)}\omega}(x_0^{k+1})\right\},\cdots,\\
		&\left\{x_{N_{k+1}(\omega)-1}^{k+1}, F_{\theta^{T_{N_{k+1}(\omega)-1}^{k+1}(\omega)}\omega}(x_{N_{k+1}(\omega)-1}^{k+1}),\dots,F^{\hat{n}_{k+1}(\theta^{T_{N_{k+1}(\omega)-1}^{k+1}(\omega)}\omega)-1}_{\theta^{T_{N_{k+1}(\omega)-1}^{k+1}(\omega)}\omega}(x_{N_{k+1}(\omega)-1}^{k+1})\right\}
	\end{split}
\end{displaymath}
with gaps $m_{k+1}+l_1^{k+1}(\omega),\dots,m_{k+1}+l_{N_{k+1}(\omega)-1}^{k+1}(\omega)$, respectively, i.e. for any $j\in\{0,\dots,N_{k+1}(\omega)-1\},$ we have
\begin{equation} \label{eq y shadowing xk+1} d_{\theta^{T_j^{k+1}(\omega)}\omega}^{\hat{n}_{k+1}(\theta^{T_j^{k+1}(\omega)}\omega)}\Big(x_j^{k+1}, F_{\theta^{T_0^{k+1}(\omega)}\omega}^{T_{j}^{k+1}(\omega)-T_0^{k+1}(\omega)}y\Big)<\frac{\eta}{2^{4+k+1}}.
\end{equation}Such point exists according to theorem \ref{thm fiber specification}.
We collect all such points into
\begin{equation*}
	D_{k+1}(\theta^{T_0^{k+1}(\omega)}\omega)=\{y(x_0^{k+1},\dots,x_{N_{k+1}(\omega)-1}^{k+1}):\ x_{i}^{k+1}\in C_{k+1}(\theta^{T_{i}^{k+1}(\omega)}\omega)\mbox{ for }i=1,...,N_{k+1}(\omega)-1\}.
\end{equation*}
Notice that different $N_{k+1}(\omega)$-tuples give rise to different shadowing points, which are separated in the metric $	d_{\theta^{T_0^{k+1}(\omega)}\omega}^{T_{N_{k+1}}^{k+1}(\omega)-T_0^{k+1}(\omega)}$. Similar as lemma \ref{lemma separated metric D1}, we have the following lemma.
\begin{lemma}\label{lemma separated metric Dk+1}
	If $(x_0^{k+1},\dots,x_{N_{k+1}(\omega)-1}^{k+1}), (z_0^{k+1},\dots,z_{N_{k+1}(\omega)-1}^{k+1})\in\prod_{i=0}^{N_{k+1}(\omega)-1} C_{k+1}(\theta^{T_i^{k+1}(\omega)}\omega)$ are different $N_{k+1}(\omega)$-tuples, then
	\begin{displaymath}
		d_{\theta^{T_0^{k+1}(\omega)}\omega}^{T_{N_{k+1}}^{k+1}(\omega)-T_0^{k+1}(\omega)}\left(y(x_0^{k+1},\dots,x_{N_{k+1}(\omega)-1}^{k+1}),y(z_0^{k+1},\dots,z_{N_{k+1}(\omega)-1}^{k+1})\right)> \frac{\eta}{2}-\frac{\eta}{2^{4+k+1}}\times 2.
	\end{displaymath}
	Hence
	\begin{equation}\label{number Dk+1}
		\# D_{k+1}(\theta^{T_0^{k+1}(\omega)}\omega)=\prod_{i=0}^{N_{k+1}(\omega)-1}\# C_{k+1}(\theta^{T_i^{k+1}(\omega)}\omega).
	\end{equation}
\end{lemma}In this way, we have constructed intermediate sets $D_1(\omega)$ and $D_k(\theta^{T_0^k(\omega)}\omega)$ for $k\geq 2$.

\emph{Step 2. Construction of $H_k(\omega),$ centers of balls forming the $k$-th level of fiber Moran-like fractal.}
We would like to construct the centers of balls forming the $k$-th level of fiber Moran-like fractal inductively. Let $H_1(\omega)=D_1(\omega).$ Once $H_k(\omega)$ for $k\geq 1$ is constructed, we can construct $H_{k+1}(\omega)$ inductively as follows. For any $x\in H_k(\omega)$ and $y\in D_{k+1}(\theta^{T_0^{k+1}(\omega)}\omega)$, notice that $$T_{0}^{k+1}(\omega)-T_{N_k(\omega)}^k(\omega)=m_{k+1}+l^{k,k+1}(\omega)\geq m_{k+1}$$ by the construction \eqref{gap larger than mk+1}. Therefore, by theorem \ref{thm fiber specification}, there exists a point $z=z(x,y)\in M$, which $(\omega,{\eta}/{2^{4+k+1}})$-shadows pieces of orbits
\begin{displaymath}
	\left\{x,F_{\omega}(x),\dots,F_{\omega}^{T_{N_k(\omega)}^k(\omega)-1}(x)\Big\},\Big\{y,F_{\theta^{T_0^{k+1}(\omega)}\omega}(y)\dots,F_{\theta^{T_0^{k+1}(\omega)}\omega}^{T_{N_{k+1}(\omega)}^{k+1}(\omega)-T_0^{k+1}(\omega)-1}(y)\right\}
\end{displaymath}
with the space $m_{k+1}+l^{k,k+1}(\omega),$ i.e.,
\begin{equation}\label{z shadow xy}
	d_{\omega}^{T_{N_k(\omega)}^k(\omega)}(x,z)<\frac{\eta}{2^{4+k+1}},\quad d_{\theta^{T_0^{k+1}(\omega)}\omega}^{T_{N_{k+1}(\omega)}^{k+1}(\omega)-T_0^{k+1}(\omega)}\left(y,F_{\omega}^{T_0^{k+1}(\omega)}z\right)<\frac{\eta}{2^{4+k+1}}.
\end{equation}
Collect all these points into the set $H_{k+1}(\omega)=\{z=z(x,y):x\in H_k(\omega),y\in D_{k+1}(\theta^{T_0^{k+1}(\omega)}\omega)\}.$ Similar as the proof of lemma 4.5 in \cite{variationalprincipleforrandomanosovsystems}, we can prove the following lemma.

\begin{lemma}\label{lemma Hk separated}
	For any $k\in\mathbb{N}$, $x\in H_k(\omega),y,y^{\prime}\in D_{k+1}(\theta^{T_0^{k+1}(\omega)}\omega),$ if $y\neq y^{\prime},$ one has
	\begin{align}
		&d_{\omega}^{T_{N_k(\omega)}^k(\omega)}(z(x,y),z(x,y^{\prime}))<\frac{\eta}{2^{4+k+1}}\times 2=\frac{\eta}{2^{4+k}},\label{distance of y small}\\
		&d_{\omega}^{T_{N_{k+1}(\omega)}^{k+1}(\omega)}(z(x,y),z(x,y^{\prime}))> \frac{\eta}{2}-\frac{\eta}{2^{4+k+1}}\times 4\geq \frac{3\eta}{8}.\label{distance of y large}
	\end{align}
\end{lemma}
As a consequence of lemma \ref{lemma Hk separated}, we have
\begin{equation}\label{number Hk}
	\# H_{k+1}(\omega)=\# H_{k}(\omega)\cdot \# D_{k+1}(\theta^{T_0^{k+1}(\omega)}\omega)=\#D_1(\omega)\cdot\prod_{i=2}^{k+1}D_{i}(\theta^{T_0^i(\omega)}\omega).
\end{equation}
Note that points in $H_1(\omega)=D_1(\omega)$ are $(\omega,{3\eta}/{8},T_{N_1(\omega)}^1(\omega))$-separated by lemma \ref{lemma separated metric D1}. Therefore, due to \eqref{distance of y large}, for any $k\in\mathbb{N}$, points in $H_k(\omega)$ are $(\omega,{3\eta}/{8},T_{N_{k}(\omega)}^{k}(\omega))$-separated, i.e.
\begin{equation}\label{Hk are 3eta/8 separated}
	d_{\omega}^{T_{N_{k}(\omega)}^{k}(\omega)}(z,z^\prime)> \frac{3\eta}{8},\ \forall z,z^\prime\in H_k(\omega).
\end{equation}

\emph{Step 3. Construction of fiber Moran-like fractal $\mathfrak{I}(\omega).$} For any $k\geq 1$, we define the $k$-th level of Moran-like fractal by
\begin{equation}\label{def k-th mf}
	\mathfrak{I}_k(\omega)=\bigcup_{x\in H_k(\omega)}\overline{\mathcal{B}}_{T_{N_k(\omega)}^k(\omega)}\left(\omega,x,\frac{\eta}{2^{4+k}}\right),
\end{equation}
where
\begin{displaymath}
	\overline{\mathcal{B}}_{T^k_{N_k(\omega)}(\omega)}\left(\omega,x,\frac{\eta}{2^{4+k}}\right)=\left\{z\in M:d_{\omega}^{T^k_{N_k(\omega)}(\omega)}\left(z,x\right)\leq\frac{\eta}{2^{4+k}}\right\}.
\end{displaymath}
The structure the $k$-th level of fiber Moran-like fractal is given in the next lemma, whose proof is placed in section \ref{sec proof of lemmas}.
\begin{lemma}\label{lemma decreased Ik}
	For every $k\geq 1$, the following statements hold
	\begin{enumerate}
		\item for any $x,x^{\prime}\in H_{k}(\omega)$, $x\neq x^{\prime},$ we have
		$$\overline{\mathcal{B}}_{T_{N_k(\omega)}^k(\omega)}\left(\omega,x,\frac{\eta}{2^{4+k}}\right)\cap \overline{\mathcal{B}}_{T_{N_k(\omega)}^k(\omega)}\left(\omega,x^{\prime},\frac{\eta}{2^{4+k}}\right)=\varnothing.$$
		\item  if $z=z(x,y)\in H_{k+1}(\omega)$ for $x\in H_k(\omega)$ and $y\in D_{k+1}(\theta^{T_0^{k+1}(\omega)}\omega),$ then
		\begin{displaymath}
			\overline{\mathcal{B}}_{T_{N_{k+1}(\omega)}^{k+1}(\omega)}\left(\omega,z,\frac{\eta}{2^{4+k+1}}\right)\subset \overline{\mathcal{B}}_{T_{N_k(\omega)}^k(\omega)}\left(\omega,x,\frac{\eta}{2^{4+k}}\right).
		\end{displaymath}
		As a consequence, $\mathfrak{I}_{k+1}(\omega)\subset\mathfrak{I}_k(\omega).$
	\end{enumerate}
\end{lemma}

Finally, we define the fiber Moran-like fractal to be
\begin{equation}\label{construction of hua H}
	\mathfrak{I}(\omega)=\bigcap_{k\geq 1}\mathfrak{I}_k(\omega).
\end{equation}
By the compactness of $M, \mathfrak{I}(\omega)$ is a \emph{nonempty} closed subset of $M.$ The following lemma shows that the fiber Moran-like fractal is contained in the fiber of the irregular set, and its proof is addressed in section \ref{sec proof of lemmas}.

\begin{lemma}\label{lemma different limits}
	For any $x\in\mathfrak{I}(\omega)$, the following limits exist
	\begin{align}
		&\lim_{k\to\infty}\frac{1}{T_{N_{2k}(\omega)}^{2k}(\omega)}\sum_{i=0}^{T_{N_{2k}(\omega)}^{2k}(\omega)-1}\varphi(\Theta^i(\omega,x))=\alpha_0,\label{limit of even numbers}\\&\lim_{k\to\infty}\frac{1}{T_{N_{2k+1}(\omega)}^{2k+1}(\omega)}\sum_{i=0}^{T_{N_{2k+1}(\omega)}^{2k+1}(\omega)-1}\varphi(\Theta^i(\omega,x))=\alpha_1.\label{limit of odd numbers}
	\end{align}
	Therefore, we have $\mathfrak{I}(\omega)\subset I_{\varphi}(\omega)$.
\end{lemma}

Notice that $\mathfrak{I}(\omega)$ is nonempty for any $\omega\in\widehat{\Omega}_{\gamma}$ defined in \eqref{Omega gamma hat}. Therefore, ${I}_{\varphi}(\omega)\neq\varnothing$ for any $\omega\in\widehat{\Omega}_{\gamma}$ by lemma \ref{lemma different limits}. Since $\widehat{\Omega}_{\gamma}$ has $\mathbb{P}$-full measure and $\mathbb{P}$ is complete on $(\Omega,\mathcal{B}_{\mathbb{P}}(\Omega)),$ we conclude that
\begin{displaymath}
	\mathbb{P}(\{\omega\in\Omega: I_{\varphi}(\omega)\neq\varnothing\})=1.
\end{displaymath}

\subsection{Construction of a suitable probability measure supported on fractals}
In this subsection, we will construct a suitable probability measure supported on the fiber Moran-like fractals.

For every $k\geq 1$, we define a probability measure $\mu_{k,\omega}\in \mathcal{M}(M)$ by
\begin{displaymath}
	\mu_{k,\omega}=\frac{1}{\# H_k(\omega)}\sum_{x\in H_k(\omega)}\delta_x.
\end{displaymath}
Since $\mu_{k,\omega}$ is concentrated on $H_k(\omega)$, we have $\mu_{k,\omega}(\mathfrak{I}_k(\omega))=1$ by construction \eqref{def k-th mf}. The proof of the following two lemmas are given in section \ref{sec proof of lemmas}.
\begin{lemma}\label{lemma limit measure}
	For any continuous function $\psi:M\to\mathbb{R}$, the following limit exists
	\begin{displaymath}
		\lim_{k\to\infty}\int_M\psi(x)\dif\mu_{k,\omega}(x).
	\end{displaymath}
\end{lemma}

By the Riesz representation theorem, there exists some  $\mu_{\omega}\in \mathcal{M}(M)$ such that
\begin{displaymath}
	\int\psi \dif\mu_{\omega}=\lim_{k\to\infty}\int\psi \dif\mu_{k,\omega}, \forall \psi\in C(M,\mathbb{R}).
\end{displaymath}
We note that $\omega\in\widehat{\Omega}_{\gamma}$ is fixed at the beginning of this proof. Hence, the notation $\mu_\omega$ can be used to denote the measure concentrated on $\mathfrak{I}(\omega)$ without ambiguity.

\begin{lemma}\label{lemma I full measure}
	$\mu_{\omega}(\mathfrak{I}(\omega))=1.$ That is, $\mu_{\omega}$ is a probability measure supported on the fractal $\mathfrak{I}(\omega).$
\end{lemma}

\subsection{Apply entropy distribution type argument}
In this subsection, with the help of the following lemma \ref{lemma entropy distribution}, we will the apply entropy distribution principle type argument to prove our main theorem. Readers who are not familiar with the entropy distribution principle argument are referred to \cite[Theorem 3.6]{TakensVerbitsky2003variationalprinciple}.

\begin{lemma}\label{lemma entropy distribution}
	There exists $N(\omega)\in\mathbb{N},$ such that for any $n\geq N(\omega),$
	\begin{equation*}
		\mbox{if $B_n\left(\omega,x,{\eta}/{2^{4}}\right)\cap\mathfrak{I}(\omega)\neq \varnothing$, then $ \mu_{\omega}\left(B_n\left(\omega,x,{\eta}/{2^4}\right)\right)\leq e^{-n(h_{\mu_0}(F)-6\gamma)}$}.
	\end{equation*}
Recall $\mu_0$  is defined in \eqref{def of hmu0}.
\end{lemma}
Let $\Gamma_{\omega}^{\eta/2^4}=\{B_{n_i}(\omega,x_i,{\eta}/{2^4})\}_{i=1}^{\infty}$ be any finite or countable cover of $\mathfrak{I}(\omega)$ with $\min\{n_i\}\geq N(\omega)$. Without loss of generality, we assume that $B_{n_i}(\omega,x_i,{\eta}/{2^4})\cap\mathfrak{I}(\omega)\neq \varnothing$ for every $i$. By lemma \ref{lemma entropy distribution}, we have
	\begin{displaymath}
		\sum_ie^{-n_i(h_{\mu_0}(F)-6\gamma)}\geq \sum_{i}\mu_{\omega}\left[B_{n_i}\left(\omega,x_i,\frac{\eta}{2^{4}}\right)\right]\geq \mu_{\omega}(\mathfrak{I}(\omega))=1>0.
	\end{displaymath}
	Therefore, combining \eqref{def of m(Z,s,omega,N,epsilon)} and \eqref{def of m(Z,s,omega,epsilon)} together, the following holds
	\begin{displaymath}
		m\left(\mathfrak{I}(\omega), h_{\mu_0}(F)-6\gamma,\omega,\frac{\eta}{2^4}\right)\geq \inf_{\Gamma_{\omega}^{\eta/2^4}}\sum_ie^{-n_i(h_{\mu_0}(F)-6\gamma)}\geq 1>0.
	\end{displaymath}
	Hence, by definition \eqref{def of htop(F,Z,omega,epsilon)}, we have
	\begin{displaymath}
		h_{top}\left(F,\mathfrak{I}(\omega),\omega,\frac{\eta}{2^4}\right)\geq h_{\mu_0}(F)-6\gamma.
	\end{displaymath}
   Combining lemma \ref{lemma htop exists} and lemma \ref{lemma different limits}, we conclude that
	\begin{equation}\label{htop I plus 6gamma is greater}
		h_{\mu_0}(F)-6\gamma\leq \sup_{\epsilon>0}h_{top}(F,\mathfrak{I}(\omega),\omega,\epsilon)=h_{top}(F,\mathfrak{I}(\omega),\omega)\leq h_{top}(F,I_{\varphi}(\omega),\omega).
	\end{equation}
	
	We have showed that \eqref{htop I plus 6gamma is greater} holds for arbitrary $\omega\in\widehat{\Omega}_\gamma$. To show \eqref{htop is greater that hmu0}, we pick a countable sequence $\gamma_k\to 0,$ and let $\widehat{\Omega}=\cap_{k}\widehat{\Omega}_{\gamma_k}.$ It follows that $\mathbb{P}(\widehat{\Omega})=1,$ and for any $\omega\in\widehat{\Omega}\subset\widehat{\Omega}_k$, we have
	\begin{displaymath}
		h_{\mu_0}(F)-6\gamma_k\leq h_{top}(F,I_{\varphi}(\omega),\omega).
	\end{displaymath}
  Therefore, for any $\omega\in\widehat{\Omega}$, we arrive at
	\begin{equation}
		h_{top}(F,I_{\varphi}(\omega),\omega)\geq h_{\mu_0}(F).
	\end{equation}

The proof of theorem \ref{thm irregular set} is complete.
\section{Topological complexity of irregular set}\label{sec proof residual}
In this section, we will show that the $\omega$-fiber of the irregular set $I_{\varphi}$ is \emph{residual} for $\mathbb{P}$-a.e. $\omega\in\Omega$. The main proof is inspired by \cite[Theorem 1]{MariaPaulo2021genericityofhistoricbehavior}, and divided into two parts. Besides, the proof of all lemmas is collected in section \ref{sec proof of lemmas}.
\subsection{A general construction of a type of dense subset}\label{subsection general construction}
The main result of this subsection is the following proposition.
\begin{proposition}\label{proposition three dense subsets}
   Let $\varphi\in C(\Omega\times M,\mathbb{R})$ be any continuous observable. Then for $\mathbb{P}$-a.e. $\omega\in\Omega$, the following sets
	\begin{align}
		&I_{\varphi}(\omega)=\left\{x\in M:\lim_{n\to\infty}\frac{1}{n}\sum_{i=0}^{n-1}\varphi(\Theta^i(\omega,x))\mbox{ does not exist}\right\},\nonumber\\ &K_{\varphi,\alpha_0}(\omega)=\left\{x\in M:\lim_{n\to\infty}\frac{1}{n}\sum_{i=0}^{n-1}\varphi(\Theta^i(\omega,x))=\alpha_0\right\},\nonumber\\
		&K_{\varphi,\alpha_1}(\omega)=\left\{x\in M:\lim_{n\to\infty}\frac{1}{n}\sum_{i=0}^{n-1}\varphi(\Theta^i(\omega,x))=\alpha_1\right\}.\nonumber
	\end{align}
are dense in $M$, where $\alpha_0,\alpha_1$ are defined in lemma \ref{lemma different integral values}.
\end{proposition}
\begin{proof}[Proof of Proposition \ref{proposition three dense subsets}]
We are going to prove a slightly stronger version of the proposition \ref{proposition three dense subsets}. That is, for any $(\omega,x)\in\Omega\times M$, we will show that there exists a dense subset $Z$ of $M$, such that for any $z\in Z$, the following equality hold
\begin{equation}\label{limit of difference is zero}
	\lim_{n\to\infty}\left[\frac{1}{n}\sum_{i=0}^{n-1}\varphi(\Theta^i(\omega,z))-\frac{1}{n}\sum_{i=0}^{n-1}\varphi(\Theta^i(\omega,x))\right]=0.
\end{equation}
Notice that we do not require that  $\lim_{n\to\infty}\frac{1}{n}\sum_{i=0}^{n-1}\varphi(\Theta^i(\omega,x))$ exists. In fact, with the validity of \eqref{limit of difference is zero}, one can conclude that
\begin{displaymath}
	\begin{split}
		(\omega,x)\mbox{ is a $\varphi$-irregular point}&\Longleftrightarrow (\omega,z)\mbox{ is a $\varphi$-irregular point},\\ x\in K_{\varphi,\alpha_0}(\omega)&\Longleftrightarrow z\in K_{\varphi,\alpha_0}(\omega),\\x\in K_{\varphi,\alpha_1}(\omega)&\Longleftrightarrow z\in K_{\varphi,\alpha_1}(\omega).
	\end{split}
\end{displaymath}
Therefore, we remain to prove that \eqref{limit of difference is zero} holds for all $z$ in some dense subset $Z$ of $M$.

Let $\{x_j\}_{j=1}^{\infty}\subset M$ be a dense subset of $M$. For any $\epsilon>0$, let $s_k=m(\epsilon/2^k)$ be the gap corresponding to the given precision $\epsilon/2^k$ in the definition of fiber specification property, such that for each $\omega\in\Omega$, any $s_k$-spaced $\omega$-specification is $(\omega,\epsilon/2^k)$-shadowed by some point in $M$. Let $\{n_k\}_{k=1}^{\infty}\subset\mathbb{N}$ be a sequence of integers, such that for any $k\in\mathbb{N}$, the following estimations hold
\begin{equation}\label{estimation of nk}
	\begin{cases}
		n_k\geq \max\{2^{s_k},2^{s_{k+1}}\},\\  n_{k+1}\geq 2^{\sum_{i=1}^{k}(s_i+n_i)}.
	\end{cases}
\end{equation}
Define $l_k=\sum_{i=1}^{k}(s_i+n_i)$ for $k\in\mathbb{N}$, and the following limits exist due to \eqref{estimation of nk}
\begin{align}
	&\lim_{k\to \infty}\frac{s_k}{n_k}=0.\label{lim sk over nk}\\ &\lim_{k\to\infty}\frac{s_{k+1}}{n_k}=0.\label{lim sk+1 over nk}\\&\lim_{k\to\infty}\frac{l_k}{n_{k+1}}=0.\label{lim lk over nk+1}
\end{align}

We define a sequence of points $\{y_j^{k}\}_{k=1}^{\infty}$ depending on $x_j$, $x$ and $\omega$ inductively.
By Lemma \ref{thm fiber specification}, there exists $y_j^1\in M$, which $(\omega,\epsilon/2)$-shadows two pieces of orbits
\begin{equation*}
	\{x_j\},\left\{F_{\omega}^{s_1}x,\dots,F_{\omega}^{l_1-1}x\right\}.
\end{equation*}
That is
\begin{equation}\label{equation shadowing 1st level}
	d_{M}(y_j^1,x_j)<\frac{\epsilon}{2},\quad d_{\theta^{s_1}\omega}^{n_1}(F_{\omega}^{s_1}y_j^1, F_{\omega}^{s_1}x)<\frac{\epsilon}{2}.
\end{equation}
Assume that $y_j^k$ has already been defined, we define $y_j^{k+1}$ utilizing the fiber specification property. For $\epsilon/2^{k+1}>0$, there exists $y_j^{k+1}\in M$ which $(\omega,\epsilon/2^{k+1})$-shadows two pieces of orbits
\begin{equation*}
	\{y_j^k,\dots, F_{\omega}^{l_k-1}y_j^k\},\{F_{\omega}^{l_k+s_{k+1}}x, F_{\omega}^{l_{k+1}-1}x\}.
\end{equation*}
That is
\begin{equation}\label{equation shadowing property yk+1 and yk}
	d_{\omega}^{l_k}(y_j^{k+1}, y_j^k)<\frac{\epsilon}{2^{k+1}},\quad d_{\theta^{l_k+s_{k+1}}\omega}^{n_{k+1}}(F_{\omega}^{l_k+s_{k+1}}y_j^{k+1}, F_{\omega}^{l_k+s_{k+1}}x)<\frac{\epsilon}{2^{k+1}}.
\end{equation}
\begin{lemma}\label{lemma contain and nonempty}
	For every $k\geq 1$, we have
	\begin{equation}\label{equation contain}
		\mathcal{B}_{l_{k+1}}\left(\omega,y_j^{k+1},\frac{\epsilon}{2^{k+1}}\right)\subset\mathcal{B}_{l_k}\left(\omega,y_j^k,\frac{\epsilon}{2^k}\right),
	\end{equation}
and
	\begin{equation}\label{equation nonempty}
	 \bigcap_{k\geq 1}\mathcal{B}_{l_k}\left(\omega,y_j^k,\frac{\epsilon}{2^{k}}\right)\neq\varnothing.
	\end{equation}
\end{lemma}

Define
\begin{equation*}
	\mathcal{Z}_{\varphi}(\omega,x,x_j,\epsilon):=\bigcap_{k\geq1}\mathcal{B}_{l_k}\left(\omega,y_j^k,\frac{\epsilon}{2^{k}}\right).
\end{equation*}

\begin{lemma}\label{lemma limit of difference exists}
	For any $z\in \mathcal{Z}_{\varphi}(\omega,x,x_j,\epsilon)$, the following holds
	\begin{equation*}
		\lim_{n\to\infty}\left[\frac{1}{n}\sum_{i=0}^{n-1}\varphi(\Theta^i(\omega,x))-\frac{1}{n}\sum_{i=0}^{n-1}\varphi(\Theta^i(\omega,z))\right]=0.
	\end{equation*}
\end{lemma}

We pick a countable sequence $\epsilon_k\to 0$, and define
\begin{displaymath}
	\mathcal{Z}_{\varphi}(\omega,x)=\bigcup_{j\geq1}\bigcup_{k\geq 1} \mathcal{Z}_{\varphi}(\omega,x,x_j,\epsilon_k).
\end{displaymath}
\begin{lemma}\label{lemma dense set is abundant}
	For any $(\omega,x)\in\Omega\times M,$ $\mathcal{Z}_{\varphi}(\omega,x)$ is dense in $M$.
\end{lemma}

Combining lemma \ref{lemma limit of difference exists} and lemma \ref{lemma dense set is abundant}, we conclude that there exists a dense subset $Z$ of $M$, such that \eqref{limit of difference is zero} holds for any $z\in Z$. Therefore, it remains to show that for $\mathbb{P}$-a.e $\omega\in\Omega$, the $\omega$-fibers of $I_{\varphi}, K_{\varphi,\alpha_0}$ and $K_{\varphi,\alpha_1}$ are nonempty. Notice that we have shown that
\begin{displaymath}
	\mathbb{P}(\{\omega\in\Omega:I_{\varphi}(\omega)\neq\varnothing\})
\end{displaymath} in the proof of theorem \ref{thm irregular set}. Besides, the authors have already established in \cite[Remark 4.3]{variationalprincipleforrandomanosovsystems} that
for our  target system $(\Omega\times M,\Theta)$, the following fact
\begin{displaymath}
	\mathbb{P}(\pi_{\Omega}(K_{\varphi,\alpha}))=\mathbb{P}\left(\left\{\omega\in\Omega: K_{\varphi,\alpha}(\omega)\neq\varnothing\right\}\right)=1
\end{displaymath}
hold for any $\alpha\in\{\int\varphi\dif\mu:\mu\in I_{\Theta}(\Omega\times M)\}.$ By the construction in lemma \ref{lemma different integral values}, both $\alpha_0$ and $\alpha_1$ are obtained through the integration of $\varphi$ with respect to the $\Theta$-invariant measures $\mu_0$ and $\mu_1$. Therefore, we finally conclude that for $\mathbb{P}$-a.e. $\omega\in\Omega$, the sets $I_{\varphi}(\omega), K_{\varphi,\alpha_0}(\omega)$ and $K_{\varphi,\alpha_1}(\omega)$ are all dense in $M$.
This finishes the proof of proposition \ref{proposition three dense subsets}.
\end{proof}

\subsection{Fibers of the irregular set are residual}
In subsection \ref{subsection general construction}, we have already show that there exists a $\mathbb{P}$-full measure set $\overline{\Omega}$, such that for any $\omega\in \overline{\Omega}$, sets $I_{\varphi}(\omega),K_{\varphi,\alpha_0}(\omega)$ and $K_{\varphi,\alpha_1}(\omega)$ are dense in $M$.
\begin{proposition}\label{proposition residual}
	For any $\omega\in\overline{\Omega}$, the $\omega$-fiber $I_{\varphi}(\omega)$ is residual in $M$.
\end{proposition}

\begin{proof}[Proof of Proposition \ref{proposition residual}]
	Notice that to show  $I_{\varphi}(\omega)$ is residual is equivalent to show $I_{\varphi}(\omega)^c$ is of first category. Let us fix some $\omega\in\overline{\Omega}$ from now on.
	
	Given some sufficiently small constant $\delta$, such that
	\begin{equation}\label{equation delta}
		0<\delta<\frac{|\alpha_0-\alpha_1|}{3}.
	\end{equation}We define
\begin{align*}
   \Lambda_N(\omega)=&\left\{x\in M:\left|\frac{1}{m}\sum_{i=0}^{m-1}\varphi(\Theta^i(\omega,x))-\frac{1}{n}\sum_{i=0}^{n-1}\varphi(\Theta^i(\omega,x))\right|\leq\delta,\quad\forall m,n\geq N\right\}\\
&=\bigcap_{m,n\geq N}\left\{x\in M:\left|\frac{1}{m}\sum_{i=0}^{m-1}\varphi(\Theta^i(\omega,x))-\frac{1}{n}\sum_{i=0}^{n-1}\varphi(\Theta^i(\omega,x))\right|\leq\delta\right\}.
\end{align*}
	Note that points in $I_{\varphi}(\omega)^c$ are such that $\{\frac{1}{n}\sum_{i=0}^{n-1}\varphi(\Theta^i(\omega,x))\}_{n\in\mathbb{N}}$ converges, which implies that
	\begin{displaymath}
		I_{\varphi}(\omega)^c\subset\bigcup_{N\geq 1}\Lambda_{N}(\omega).
	\end{displaymath}
	Since $(\omega,x)\mapsto\varphi(\Theta^n(\omega,x))$ is continuous for any $n\in\mathbb{N}$, the set $\Lambda_N(\omega)$ is closed. Hence, it remains to show $\Lambda_N(\omega)$ has empty interior for any $N\in\mathbb{N}$.
	
	Assume that there exists some $N_0\in \mathbb{N}$, such that the interior of $\Lambda_{N_0}(\omega)$, denoted by $\mathrm{int}(\Lambda_{N_0}(\omega))$, is nonempty. Since $K_{\varphi,\alpha_0}(\omega)$ and $K_{\varphi,\alpha_1}(\omega)$ are dense in $M$, one can pick two sequences, named $\{x_k\}_{k\in\mathbb{N}}$ and $\{y_k\}_{k\in\mathbb{N}}$, in $K_{\varphi,\alpha_0}(\omega)\cap \Lambda_{N_0}(\omega)$ and $K_{\varphi,\alpha_1}(\omega)\cap \Lambda_{N_0}(\omega)$ respectively, such that
	\begin{displaymath}
		\lim_{k\to\infty}x_k=\lim_{k\to\infty} y_k=x^*\in \mathrm{int}(\Lambda_{N_0}(\omega)).
	\end{displaymath}
    Note that
	\begin{equation*}
		\left|\frac{1}{m}\sum_{i=0}^{m-1}\varphi(\Theta^i(\omega,x_k))-\frac{1}{n}\sum_{i=0}^{n-1}\varphi(\Theta^i(\omega,x_k))\right|\leq \delta,\quad\forall m,n\geq N_0.
	\end{equation*}
	Fixing $m\geq N_0$ and setting $n\to\infty$, we have
	\begin{equation}\label{average of xk small}
		\left|\frac{1}{m}\sum_{i=0}^{m-1}\varphi(\Theta^i(\omega,x_k))-\alpha_0\right|\leq\delta.
	\end{equation}
	Similarly, for any $m\geq N_0$, the following inequality holds
	\begin{equation}\label{average of yk small}
		\left|\frac{1}{m}\sum_{i=0}^{n-1}\varphi(\Theta^i(\omega,y_k))-\alpha_1\right|\leq\delta.
	\end{equation}
	Notice that the continuous observable $\varphi$ on compact space $\Omega\times M$ is uniformly continuous, and $F_{\omega}^n:M\to M$ is continuous for any $n\in\mathbb{N}$. Therefore, there exists some $\delta_m>0$, such that
	\begin{equation}\label{average of difference}
		d_M(x,y)<\delta_m\Longrightarrow \left|\frac{1}{m}\sum_{i=0}^{m-1}\varphi(\Theta^i(\omega,x))-\frac{1}{m}\sum_{i=0}^{m-1}\varphi(\Theta^i(\omega,y))\right|<\delta.
	\end{equation}
	Since $\{x_k\}_{k\in\mathbb{N}}$ and $\{y_k\}_{k\in\mathbb{N}}$ have the same limit $x^*$, one can pick $k$ sufficiently large, such that $d_M(x_k,y_k)<\delta_m$. Combining \eqref{average of xk small},\eqref{average of yk small} and \eqref{average of difference}, one conclude that
	\begin{displaymath}
		|\alpha_0-\alpha_1|< 3\delta,
	\end{displaymath}
	which contradicts the inequality \eqref{equation delta}. This ends the proof of proposition \ref{proposition residual}.
\end{proof}

\section{Proof of lemmas and corollaries}\label{sec proof of lemmas}

\begin{proof}[Proof of Lemma \ref{lemma different integral values}]
	 According to \eqref{irregular set is nonempty}, we have
	\begin{equation*}
		\lim_{n\to\infty}\frac{1}{n}\sum_{i=0}^{n-1}\varphi(\Theta^i(\omega^*,x^*))\mbox{ does not exist}.
	\end{equation*}
 Therefore, there exists a subsequence $\{n_k\}_{k=1}^{\infty}$, such that
 \begin{equation*}
 	\lim_{k\to\infty}\frac{1}{n_k}\sum_{i=0}^{n_k-1}\varphi(\Theta^i(\omega^*,x^*))=\alpha^*\neq\int\varphi\dif\mu_0=\alpha_0.
 \end{equation*}
Define
$\mu_k=\frac{1}{n_k}\sum_{i=0}^{n_k-1}\delta_{\Theta^i(\omega^*,x^*)},$
then we have
\begin{equation*}
	\int\varphi\dif\mu_k=\frac{1}{n_k}\sum_{i=0}^{n_k-1}\varphi(\Theta^i(\omega^*,x^*)).
\end{equation*}
By the compactness of $\mathcal{M}(\Omega\times M)$ with respect to the weak* topology, without loss of generality, we assume that there exists some $\mu^*\in \mathcal{M}(\Omega\times M)$, such that $\mu_k$ converges to $\mu^{*}$ in weak* topology. Therefore
\begin{displaymath}
	\int\varphi\dif\mu^*=\lim_{k\to\infty}\int\varphi\dif\mu_k=\alpha^*.
\end{displaymath}
By the standard Krylov and Bogolyubov's argument \cite[Theorem 6.9]{Waltersergodic}, we have $\mu^*\in I_{\Theta}(\Omega\times M)$. Hence
\begin{equation}\label{astar neq a0}
\alpha^*=\int\varphi\dif\mu^*\neq\int\varphi\dif\mu_0=\alpha_0.
\end{equation}

Finally, we remain to show the existence of $\mu_1\in I_{\Theta}(\Omega\times M)$. In fact, notice that the entropy map $\mu\mapsto h_{\mu}(\Theta)$ is affine \cite[Theorem 8.1]{Waltersergodic}. Therefore, the fiber entropy map $\mu\mapsto h_{\mu}(F)$ is affine due to Abramov-Rohlin formula \eqref{Rohlin formula}.
Hence, we have
\begin{displaymath}
	h_{\lambda\mu_0+(1-\lambda)\mu^*}(F)=\lambda h_{\mu_0}(F)+(1-\lambda)h_{\mu^*}(F),\quad\forall \lambda\in [0,1].
\end{displaymath}
Recall that we have $h_{top}(F)<\infty$ in \eqref{entropy is finite}. Therefore, $h_{\lambda\mu_0+(1-\lambda)\mu^*}$ converges to $h_{\mu_0}(F)$ as $\lambda\to 1$. Hence, we can pick some $\lambda_0\in(0,1)$ which is sufficiently close to $1$, such that for the given $\gamma>0$, the following holds
\begin{displaymath}
	h_{\mu_0}(F)-\gamma\leq h_{\lambda_0\mu_0+(1-\lambda_0)\mu^*}(F)\leq h_{\mu_0}(F).
\end{displaymath}
Notice that
\begin{displaymath}
	\begin{split}
		\int\varphi\dif(\lambda_0\mu_0+(1-\lambda_0)\mu^*)&=\lambda_0\int\varphi\dif\mu_0+(1-\lambda_0)\int\varphi\dif\mu^*\\&=\lambda_0\alpha_0+(1-\lambda_0)\alpha^*.
	\end{split}
\end{displaymath}
By (\ref{astar neq a0}), we know that $\lambda_0\alpha_0+(1-\lambda_0)\alpha^*\neq\alpha_0.$ Denote $\mu_1=\lambda_0\mu_0+(1-\lambda_0)\mu^*,$ and we have
\begin{displaymath}
	\int\varphi\dif\mu_1=\alpha_1\neq\alpha_0,
\end{displaymath}and
\begin{displaymath}
	h_{\mu_0}(F)-\gamma\leq h_{\mu_1}(F)\leq h_{\mu_0}(F).
\end{displaymath}
\end{proof}

\begin{proof}[Proof of lemma \ref{lemma decreased Ik}]
$\quad$

	(1)	By \eqref{Hk are 3eta/8 separated}, for any $x\not=x^\prime\in H_k(\omega)$, one has $d_{\omega}^{T_{N_k(\omega)}^k(\omega)}(x,x^{\prime})> {3\eta}/{8},$ hence
	\begin{displaymath}
		\overline{\mathcal{B}}_{T_{N_k(\omega)}^k(\omega)}\Big(\omega,x,\frac{\eta}{2^{4+k}}\Big)\cap \overline{\mathcal{B}}_{T_{N_k(\omega)}^k(\omega)}\Big(\omega,x^{\prime},\frac{\eta}{2^{4+k}}\Big)=\varnothing.
	\end{displaymath}
	
	(2) Let $z=z(x,y)\in H_{k+1}(\omega)$ for $x\in H_k(\omega)$ and $y\in D_{k+1}(\theta^{T_0^{k+1}(\omega)}\omega).$ Pick any point $t\in\overline{\mathcal{B}}_{T_{N_{k+1}(\omega)}^{k+1}(\omega)}(\omega,z,{\eta}/{2^{4+k+1}})$, we have
	\begin{equation*}
		d_{\omega}^{T_{N_k(\omega)}^k(\omega)}(t,x)\leq d_{\omega}^{T_{N_k(\omega)}^k(\omega)}(t,z)+d_{\omega}^{T_{N_k(\omega)}^k(\omega)}(z,x)
		\overset{\eqref{z shadow xy}}<\frac{\eta}{2^{4+k+1}}+\frac{\eta}{2^{4+k+1}}
		=\frac{\eta}{2^{4+k}},
	\end{equation*}
	Therefore, $\overline{\mathcal{B}}_{T_{N_{k+1}(\omega)}^{k+1}(\omega)}\left(\omega,z,{\eta}/{2^{4+k+1}}\right)\subset \overline{\mathcal{B}}_{T_{N_k(\omega)}^k(\omega)}\left(\omega,x,{\eta}/{2^{4+k}}\right).
	$ The proof of lemma \ref{lemma decreased Ik} is complete.
\end{proof}
\begin{proof}[Proof of Lemma \ref{lemma different limits}]
	For any $x\in\mathfrak{I}(\omega)$, we have to estimation the following quantities
	\begin{equation*}
	\left|\sum_{i=0}^{T_{N_{2k}(\omega)}^{2k}(\omega)-1}\varphi(\Theta^i(\omega,x))-T_{N_{2k}(\omega)}^{2k}(\omega)\alpha_0\right|,
	\end{equation*}
and
\begin{equation*}
	\left|\sum_{i=0}^{T_{N_{2k+1}(\omega)}^{2k+1}(\omega)-1}\varphi(\Theta^i(\omega,x))-T_{N_{2k+1}(\omega)}^{2k+1}(\omega)\alpha_1\right|.
\end{equation*}
We are going to utilize the shadowing property, triangle inequality and the estimations for points in $D_k(\theta^{T_0^k(\omega)}\omega)$ and $H_k(\omega)$.
	We divide our proof of lemma \ref{lemma different limits} into several steps.
	
	Let us introduce some notation at the beginning: for any $c>0$, put
	\begin{equation}
		\mathrm{var}(\varphi,c)=\sup\{|\varphi(\omega,x)-\varphi(\omega^{\prime},x^{\prime})|:d((\omega,x),(\omega^{\prime},x^{\prime}))<c\}.
	\end{equation}
	Note that, $\mathrm{var}(\varphi,c)\to 0$ as $c\to 0$ due to the compactness of $\Omega\times M$.
	
	\emph{Step 1: estimation on $D_1(\omega)$.} For any $y=y(x_0^1,\dots,x_{N_1(\omega)-1}^1)\in D_1(\omega)$, we have
	\begin{equation}
		d_{\theta^{T_j^1(\omega)}\omega}^{\hat{n}_1(\theta^{T_j^1(\omega)}\omega)}\left(x_j^1,F_{\omega}^{T_{j}^1(\omega)}y\right)<\frac{\eta}{2^{4+1}},\quad\forall j\in\{0,\dots,N_1(\omega)-1\}.
	\end{equation}
Notice that we have the following decomposition of interval $[0,T_{N_1(\omega)}^1(\omega)-1],$ i.e.
\begin{equation}\label{shadowing property on D1}
	\begin{split}
		&[0,T_0^1(\omega)-1]\cup[T_0^1(\omega),T_0^1+\hat{n}(\theta^{T_0^1(\omega)}\omega)-1]\cup[T_0^1(\omega)+\hat{n}_1(\theta^{T_0^1(\omega)}\omega),T_1^1(\omega)-1]\\&\quad\cup\cdots\cup[T_{N_1(\omega)-1}^1(\omega),T_{N_1(\omega)-1}^1(\omega)+\hat{n}_1(\theta^{T_{N_1(\omega)-1}^1(\omega)}\omega)-1].
	\end{split}
\end{equation}
We use triangle inequality, shadowing property \eqref{shadowing property on D1} and the fact that $$x_j^1\in C_1(\theta^{T_j^1(\omega)}\omega)\subset P(\alpha_{\rho(1)},4\delta_1,\hat{n}_1(\theta^{T_j^1(\omega)}\omega),\theta^{T_j^1(\omega)}\omega),\quad\forall j\in\{0,\dots,N_1(\omega)-1\}$$ on intervals $\cup_{j=0}^{N_1(\omega)-1}[T_j^1(\omega),T_j^1(\omega)+\hat{n}_1(\theta^{T_j^1(\omega)}\omega)-1],$ while on other intervals, we use the trivial estimation $|\varphi-\alpha_{\rho(1)}|\leq 2\|\varphi\|_{C^0}.$
Therefore, the following estimation holds
\begin{equation}\label{estimation on D1}
	\begin{split}		
		&\quad \left|\sum_{i=0}^{T_{N_1(\omega)}^1(\omega)-1}\varphi(\Theta^i(\omega,y))-T_{N_1(\omega)}^1(\omega)\alpha_{\rho(1)}\right|\\
		&\leq\sum_{j=0}^{N_1(\omega)-1}\left|\sum_{i=T_j^1(\omega)}^{T_j^1(\omega)+\hat{n}_1(\theta^{T_j^1(\omega)}\omega)-1}\left(\varphi(\Theta^i(\omega,y))-\varphi(\Theta^{i-T_j^1(\omega)}(\theta^{T_j^1(\omega)}\omega,x_j^1))\right)\right|\\
		&\quad\quad+\sum_{j=0}^{N_1(\omega)-1}\left|\sum_{i=T_j^1(\omega)}^{T_j^1(\omega)+\hat{n}_1(\theta^{T_j^1(\omega)}\omega)-1}\varphi(\Theta^{i-T_j^1(\omega)}(\theta^{T_j^1(\omega)}\omega,x_j^1))-\hat{n}_1(\theta^{T_j^1(\omega)}\omega)\alpha_{\rho(1)}\right|\\
		&\quad\quad+2\left(T_0^1(\omega)+\sum_{j=1}^{N_1(\omega)-1}(m_1+l_j^1(\omega))\right)\|\varphi\|_{C^0}\\
		&\leq\sum_{j=0}^{N_1(\omega)-1}\hat{n}_1(\theta^{T_j^1(\omega)}\omega)\mathrm{var}\left(\varphi,\frac{\eta}{2^{4+1}}\right)+4\sum_{j=0}^{N_1(\omega)-1}\hat{n}_1(\theta^{T_j^1(\omega)}\omega)\delta_1 \\
		&\quad\quad +2\left(T_0^1(\omega)+\sum_{j=1}^{N_1(\omega)-1}(m_1+l_j^1(\omega))\right)\|\varphi\|_{C^0}.
	\end{split}
\end{equation}

\emph{Step 2: estimation on $D_k(\theta^{T_0^k(\omega)}\omega)$ for $k\geq 2$.} For any $y\in D_k(\theta^{T_0^k(\omega)}\omega)$, we have to estimate the following quantity
\begin{displaymath}
	\begin{split}
		   \Bigg|\sum_{i=T_0^k(\omega)}^{T_{N_k(\omega)}^k(\omega)-1}\varphi(\Theta^{i-T_0^k(\omega)}(\theta^{T_0^k(\omega)}\omega,y))-(T_{N_k(\omega)}^k(\omega))-T_0^k(\omega)\alpha_{\rho(k)}\Bigg|.
	\end{split}
\end{displaymath}

For any $y=y(x_0^k,\dots,x_{N_k(\omega)-1}^k)\in D_k(\theta^{T_0^k(\omega)}\omega)$, we have
\begin{equation*}\label{shadowing property on Dk}
	d_{\theta^{T_j^{k}(\omega)}\omega}^{\hat{n}_{k}(\theta^{T_j^{k}(\omega)}\omega)}\Big(x_j^{k}, F_{\theta^{T_0^{k}(\omega)}\omega}^{T_{j}^{k}(\omega)-T_0^{k}(\omega)}y\Big)<\frac{\eta}{2^{4+k}}.
\end{equation*}
Notice that we have the following decompostion of the interval $[T_0^k(\omega),T_{N_k(\omega)}^k(\omega)-1]$, i.e.
\begin{displaymath}
	\begin{split}
		&[T_0^k(\omega),T_0^k(\omega)+\hat{n}_k(\theta^{T_0^k(\omega)}\omega)-1]\cup[T_0^k(\omega)+\hat{n}_k(\theta^{T_0^k(\omega)}\omega),T_1^k(\omega)-1]\\
		&\quad\cup\cdots\cup[T_{N_k(\omega)-1}^k(\omega),T_{N_k(\omega)-1}^k(\omega)+\hat{n}_k(\theta^{T_{N_k(\omega)-1}^k(\omega)}\omega)-1].
	\end{split}
\end{displaymath}
Similar as the estimation \eqref{estimation on D1}, the following estimation holds
\begin{equation*}
	\begin{split}
	\quad\quad\quad&\Bigg|\sum_{i=T_0^k(\omega)}^{T_{N_k(\omega)}^k(\omega)-1}\varphi(\Theta^{i-T_0^k(\omega)}(\theta^{T_0^k(\omega)}\omega,y))-(T_{N_k(\omega)}^k(\omega)-T_0^k(\omega))\alpha_{\rho(k)}\Bigg|\label{eq estimate on Dk}\\
	\leq &\sum_{j=0}^{N_k(\omega)-1}\left|\sum_{i=T_j^k(\omega)}^{T_j^k(\omega)+\hat{n}_k(\theta^{T_j^k(\omega)}\omega)-1}\left(\varphi(\Theta^{i-T_0^k(\omega)}(\theta^{T_0^k(\omega)}\omega,y))-\varphi(\Theta^{i-T_j^k(\omega)}(\theta^{T_j^k(\omega)}\omega,x_j^k))\right)\right|\\
	&\quad+\sum_{j=0}^{N_k(\omega)-1}\left|\sum_{i=T_j^k(\omega)}^{T_j^k(\omega)+\hat{n}_k(\theta^{T_j^k(\omega)}\omega)-1}\varphi(\Theta^{i-T_j^k(\omega)}(\theta^{T_j^k(\omega)}\omega,x_j^k))-\hat{n}_k(\theta^{T_j^k(\omega)}\omega)\alpha_{\rho(k)}\right|\\
	&\quad+2\sum_{j=1}^{N_k(\omega)-1}(m_k+l_j^k(\omega))\|\varphi\|_{C^0}\\
	\leq&\sum_{j=0}^{N_k(\omega)-1}\hat{n}_k(\theta^{T_j^k(\omega)}\omega)\mathrm{var}\left(\varphi,\frac{\eta}{2^{4+k}}\right)+4\sum_{j=0}^{N_k(\omega)-1}\hat{n}_k(\theta^{T_j^k(\omega)}\omega)\delta_k+2\sum_{j=1}^{N_k(\omega)-1}(m_k+l_j^k(\omega))\|\varphi\|_{C^0}.
	\end{split}
\end{equation*}

\emph{Step 3: estimation on $H_k(\omega)$.} We introduce the following notation
\begin{equation}\label{Rk equation}
	\begin{split}
		&R_k(\omega)=\max_{z\in H_{2k-1}(\omega)}\left|\sum_{i=0}^{T_{N_{2k-1}(\omega)}^{2k-1}(\omega)-1}\varphi(\Theta^i(\omega,z))-T_{N_{2k-1}(\omega)}^{2k-1}(\omega)\alpha_1\right|,\\ &W_k(\omega)=\max_{z\in H_{2k}(\omega)}\left|\sum_{i=0}^{T_{N_{2k}(\omega)}^{2k}(\omega)-1}\varphi(\Theta^i(\omega,z))-T_{N_{2k}(\omega)}^{2k}(\omega)\alpha_0\right|.
	\end{split}
\end{equation}

First, we are going to derive the relationship between $R_{k+1}(\omega)$ and $W_k(\omega)$.
For any $z\in H_{2k+1}(\omega)$, there exist $x\in H_{2k}(\omega)$ and $y\in D_{2k+1}(\theta^{T_0^{2k+1}(\omega)}\omega)$, such that $z=z(x,y),$ i.e.
\begin{equation}\label{shadowing property on H2k+1}
	d_{\omega}^{T_{N_{2k}(\omega)}^{2k}(\omega)}(x,z)<\frac{\eta}{2^{4+2k+1}},\quad d_{\theta^{T_0^{2k+1}(\omega)}\omega}^{T_{N_{2k+1}(\omega)}^{2k+1}(\omega)-T_0^{2k+1}(\omega)}\left(y,F_{\omega}^{T_{0}^{2k+1}(\omega)}z\right)<\frac{\eta}{2^{4+2k+1}}.
\end{equation}
Notice that we can decompose the interval $[0,T_{N_{2k+1}(\omega)}^{2k+1}(\omega)]$ into
\begin{displaymath}
	[0,T_{N_{2k}(\omega)}^{2k}(\omega)-1]\cup[T_{N_{2k}(\omega)}^{2k}(\omega),T_0^{2k+1}(\omega)-1]\cup [T_0^{2k+1}(\omega),T_{N_{2k+1}(\omega)}^{2k+1}(\omega)-1].
\end{displaymath}
On the interval $[0,T_{N_{2k}(\omega)}^{2k}(\omega)-1]$, we use triangle inequality, the first inequality of \eqref{shadowing property on H2k+1}, and $W_k(\omega)$. On interval $[T_{N_{2k}(\omega)}^{2k}(\omega),T_0^{2k+1}(\omega)-1]=[T_{N_{2k}(\omega)}^{2k}(\omega),T_{N_{2k}(\omega)}^{2k}(\omega)+m_{2k+1}+l^{2k,2k+1}(\omega)-1]$, we use estimate $|\varphi-\alpha_1|\leq 2\|\varphi\|_{C^0}$. On interval $[T_0^{2k+1}(\omega),T_{N_{2k+1}(\omega)}^{2k+1}(\omega)-1]$, we use triangle inequality, the second inequality of \eqref{shadowing property on H2k+1} and the estimation on $D_{2k+1}(\theta^{T_0^{2k+1}(\omega)}\omega)$. Therefore, we have the following estimation 
\begin{equation}\label{estimation of Rk}
	\begin{split}
		&\quad R_{k+1}(\omega)=\left|\sum_{i=0}^{T_{N_{2k+1}(\omega)}^{2k+1}(\omega)-1}\varphi(\Theta^i(\omega,z))-T_{N_{2k+1}(\omega)}^{2k+1}(\omega)\alpha_1\right|\\
		&\leq \left|\sum_{i=0}^{T_{N_{2k}(\omega)}^{2k}(\omega)-1}\left(\varphi(\Theta^i(\omega,z))-\varphi(\Theta^i(\omega,x))\right)\right|+\left|\sum_{i=0}^{T_{N_{2k}(\omega)}^{2k}(\omega)-1}\varphi(\Theta^i(\omega,x))-T_{N_{2k}(\omega)}^{2k}(\omega)\alpha_1\right| \\
		&\quad+2(m_{2k+1}+l^{2k,2k+1}(\omega))\|\varphi\|_{C^0}+\left|\sum_{i=T_{0}^{2k+1}(\omega)}^{T_{N_{2k+1}(\omega)}^{2k+1}(\omega)-1}\left(\varphi(\Theta^i(\omega,z))-\varphi(\Theta^{i-T_0^{2k+1}(\omega)}(\theta^{T_0^{2k+1}(\omega)}\omega,y))\right)\right|\\
		&\quad+\left|\sum_{i=T_{0}^{2k+1}(\omega)}^{T_{N_{2k+1}(\omega)}^{2k+1}(\omega)-1}\varphi(\Theta^{i-T_0^{2k+1}(\omega)}(\theta^{T_0^{2k+1}(\omega)}\omega,y))-(T_{N_{2k+1}(\omega)}^{2k+1}(\omega)-T_0^{2k+1}(\omega))\alpha_1\right|\\
		&\leq T_{N_{2k}(\omega)}^{2k}(\omega)\mathrm{var}\left(\varphi,\frac{\eta}{2^{4+2k+1}}\right)+W_k(\omega)+T_{N_{2k}(\omega)}^{2k}|\alpha_1-\alpha_0|+2(m_{2k+1}+l^{2k,2k+1}(\omega))\|\varphi\|_{C^0}\\
		&\quad +\left(T_{N_{2k+1}(\omega)}^{2k+1}(\omega)-T_0^{2k+1}(\omega)\right)\mathrm{var}\left(\varphi,\frac{\eta}{2^{4+2k+1}}\right)+\sum_{j=0}^{N_{2k+1}(\omega)-1}\hat{n}_{2k+1}(\theta^{T_j^{2k+1}(\omega)}\omega)\mathrm{var}\left(\varphi,\frac{\eta}{2^{4+2k+1}}\right)\\
		&\quad +4\sum_{j=0}^{N_{2k+1}(\omega)-1}\hat{n}_{2k+1}(\theta^{T_j^{2k+1}(\omega)}\omega)\delta_{2k+1}+2\sum_{j=1}^{N_{2k+1}(\omega)-1}(m_{2k+1}+l_j^{2k+1}(\omega))\|\varphi\|_{C^0}\\
		&\leq W_k(\omega)+ T_{N_{2k}(\omega)}^{2k}(\omega)|\alpha_1-\alpha_0| +2T_{N_{2k+1}(\omega)}^{2k+1}(\omega)\mathrm{var}\left(\varphi,\frac{\eta}{2^{4+2k+1}}\right)\\
		&\quad+2\Big(N_{2k+1}(\omega)m_{2k+1}+l^{2k,2k+1}(\omega)+\sum_{j=1}^{N_{2k+1}(\omega)-1}l_j^{2k+1}(\omega)\Big)\|\varphi\|_{C^0}+4T_{N_{2k+1}(\omega)}^{2k+1}(\omega)\delta_{2k+1}.
	\end{split}
\end{equation}
In other words, we have derived a relationship between $R_{k+1}(\omega)$ and $W_k(\omega)$, i.e.
\begin{equation}\label{estimation of Rk+1 and Wk}
	\begin{split}
		 R_{k+1}(\omega)&\leq W_k(\omega)+T_{N_{2k}(\omega)}^{2k}(\omega)|\alpha_1-\alpha_0|+2T_{N_{2k+1}(\omega)}^{2k+1}(\omega)\mathrm{var}\left(\varphi,\frac{\eta}{2^{4+2k+1}}\right)\\&\quad+2\Big(N_{2k+1}(\omega)m_{2k+1}+l^{2k,2k+1}(\omega)+\sum_{j=1}^{N_{2k+1}(\omega)-1}l_j^{2k+1}(\omega)\Big)\|\varphi\|_{C^0}\\&\quad+4T_{N_{2k+1}(\omega)}^{2k+1}(\omega)\delta_{2k+1}.
	\end{split}
\end{equation}

As for $W_k(\omega)$, by using exactly the same way for estimation in \eqref{estimation of Rk}, we have
\begin{equation}\label{estimation of Wk and Rk}
	\begin{split}
		W_k(\omega)&\leq R_k(\omega)+T_{N_{2k-1}(\omega)}^{2k-1}(\omega)|\alpha_1-\alpha_0|+2T_{N_{2k}(\omega)}^{2k}(\omega)\mathrm{var}\left(\varphi,\frac{\eta}{2^{4+2k}}\right)\\&\quad+2\Big(N_{2k}(\omega)m_{2k}+l^{2k-1,2k}(\omega)+\sum_{j=1}^{N_{2k}(\omega)-1}l_j^{2k}(\omega)\Big)\|\varphi\|_{C^0}\\&\quad+4T_{N_{2k}(\omega)}^{2k}(\omega)\delta_{2k}.
	\end{split}
\end{equation}
Combining \eqref{estimation of Rk+1 and Wk} and \eqref{estimation of Wk and Rk} together, we conclude that
\begin{equation}\label{recursion of Rk+1}
	\begin{split}
		R_{k+1}(\omega)&\leq R_k(\omega)+\left(T_{N_{2k-1}(\omega)}^{2k-1}(\omega)+T_{N_{2k}(\omega)}^{2k}(\omega)\right)|\alpha_1-\alpha_0|+2T_{N_{2k}(\omega)}^{2k}(\omega)\mathrm{var}\left(\varphi,\frac{\eta}{2^{4+2k}}\right)\\&\quad+2T_{N_{2k+1}(\omega)}^{2k+1}(\omega)\mathrm{var}\left(\varphi,\frac{\eta}{2^{4+2k+1}}\right)+4\left(T_{N_{2k}(\omega)}^{2k}(\omega)\delta_{2k}+T_{N_{2k+1}(\omega)}^{2k+1}(\omega)\delta_{2k+1}\right)\\&\quad+2\Big(N_{2k}(\omega)m_{2k}+N_{2k+1}(\omega)m_{2k+1}+l^{2k-1,2k}(\omega)\\&\quad\quad+l^{2k,2k+1}(\omega)+\sum_{j=1}^{N_{2k}(\omega)-1}l_j^{2k}(\omega)+\sum_{j=1}^{N_{2k+1}(\omega)-1}l_j^{2k+1}(\omega)\Big)\|\varphi\|_{C^0}.\\
	\end{split}
\end{equation}
Similarly, one can obtain that
\begin{align}\label{recursion of Wk+1}
	\qquad\quad W_{k+1}(\omega)&\leq W_k(\omega)+\left(T_{N_{2k}(\omega)}^{2k}(\omega)+T_{N_{2k+1}(\omega)}^{2k+1}(\omega)\right)|\alpha_1-\alpha_0|+2T_{N_{2k+1}(\omega)}^{2k+1}(\omega)\mathrm{var}\left(\varphi,\frac{\eta}{2^{4+2k+1}}\right)\nonumber\\&\quad+2T_{N_{2k+2}(\omega)}^{2k+2}(\omega)\mathrm{var}\left(\varphi,\frac{\eta}{2^{4+2k+2}}\right)+4\left(T_{N_{2k+1}(\omega)}^{2k+1}(\omega)\delta_{2k+1}+T_{N_{2k+2}(\omega)}^{2k+2}(\omega)\delta_{2k+2}\right)\nonumber\\&\quad+2\Big(N_{2k+1}(\omega)m_{2k+1}+N_{2k+2}(\omega)m_{2k+2}+l^{2k,2k+1}(\omega)\\&\quad\quad+l^{2k+1,2k+2}(\omega)+\sum_{j=1}^{N_{2k+1}(\omega)-1}l_j^{2k+1}(\omega)+\sum_{j=1}^{N_{2k+2}(\omega)-1}l_j^{2k+2}(\omega)\Big)\|\varphi\|_{C^0}.\nonumber
\end{align}

Next, we claim that
\begin{align}
	&\lim_{k\to\infty}\frac{R_{k+1}(\omega)}{T_{N_{2k+1}(\omega)}^{2k+1}(\omega)}=0,\label{quotient of Rk+1 and T2k+1}\\&\lim_{k\to\infty}\frac{W_{k+1}(\omega)}{T_{N_{2k+2}(\omega)}^{2k+2}(\omega)}=0.\label{quotient of Wk+1 and T2k+2}
\end{align}
We are going to illustrate how to prove \eqref{quotient of Rk+1 and T2k+1}, and the same proof can be applied to \eqref{quotient of Wk+1 and T2k+2} directly without any essential change.

Combining \eqref{recursion of Rk+1},\eqref{estimation on D1} and the fact that $D_1(\omega)=H_1(\omega)$, one has the following estimation for $R_{k+1}(\omega)$
\begin{equation}\label{estimation of Rk+1}
	\begin{split}
		R_{k+1}(\omega)&\leq 4\left[\sum_{i=1}^{2k+1}T_{N_i(\omega)}^i(\omega)\mathrm{var}\left(\varphi,\frac{\eta}{2^{4+i}}\right)+2\delta_i\right]+2\sum_{i=1}^{2k}T_{N_i(\omega)}^i(\omega)|\alpha_1-\alpha_0|\\&\quad+4\Big(\sum_{i=1}^{2k+1}N_i(\omega)m_i+\sum_{i=0}^{2k}l^{i,i+1}(\omega)+\sum_{i=1}^{2k+1}\sum_{j=1}^{N_i(\omega)-1}l_j^i(\omega)\Big)\|\varphi\|_{C^0}.
	\end{split}
\end{equation}
In fact, by \eqref{estimate 1 bad point} and \eqref{estimate of k bad point}, one can obtain the following estimation of $T_{N_{k-1}\omega}^{k-1}(\omega)$ inductively, i.e.
\begin{equation}
	\begin{split}
		T_{N_{k-1}(\omega)}^{k-1}(\omega)&\leq N_1(\omega)(\hat{n}_1^M+m_1)\prod_{i=1}^{k-1}(1-\xi_i)^{-1}+N_2(\omega)(\hat{n}_2^M+m_2)\prod_{i=2}^{k-1}(1-\xi_i)^{-1}\\&\quad+\cdots+N_{k-1}(\omega)(\hat{n}_{k-1}^M+m_{k-1})(1-\xi_{k-1})^{-1}.
	\end{split}
\end{equation}
Considering the construction \eqref{construction of Nk} of $N_k(\omega)$, we have
\begin{equation*}
	0\leq\limsup_{k\to\infty}\frac{T_{N_{k-1}(\omega)}^{k-1}(\omega)}{T_{N_k(\omega)}^k(\omega)}\leq \limsup_{k\rightarrow\infty}\frac{ \sum_{j=1}^{k-1}(N_j(\omega)\cdot(\hat{n}_j^M+m_j)\cdot\prod_{i=j}^{k-1}(1-\xi_i)^{-1})}{N_k(\omega)}= 0.
	\end{equation*}
Therefore, one concludes that
\begin{equation}\label{quotient term 1}
	\lim_{k\to\infty}\frac{T_{N_{k-1}(\omega)}^{k-1}}{T_{N_k(\omega)}^k(\omega)}=0
\end{equation}It follows by Stolz's theorem that
	\begin{equation}\label{quotient term 2}
		\begin{split}
			&\quad\lim_{k\to\infty}\frac{\sum_{i=1}^{2k+1}T_{N_i(\omega)}^i(\omega)[\mathrm{var}(\varphi,{\eta}/{2^{4+i}})+2\delta_i]}{T_{N_{2k+1}(\omega)}^{2k+1}(\omega)}\\&=\lim_{k\to\infty}\frac{2T_{N_{2k+1}(\omega)}^{2k+1}(\omega)[\mathrm{var}(\varphi,{\eta}/{2^{4+2k+1}})+2\delta_{2k+1}]}{T_{N_{2k+1}(\omega)}^{2k+1}(\omega)-T_{N_{2k}(\omega)}^{2k}(\omega)}=0.
		\end{split}
	\end{equation}since the sequence $\{T_{N_k(\omega)}^k(\omega)\}_{k=1}^{\infty}$ is strictly increasing and  $T_{N_k(\omega)}^k(\omega)\to\infty$ as $k\to\infty$.
	Similarly, the following limit holds
	\begin{equation}\label{quotient term 3}
		\lim_{k\to\infty}\frac{\sum_{i=1}^{2k}T_{N_i(\omega)}^{i}(\omega)}{T_{N_{2k+1}(\omega)}^{2k+1}(\omega)}=0
	\end{equation}Besides, notice that $\hat{n}_k(\omega)\geq 2^{m_k}$ for all $k\in\mathbb{N}$, so by the construction of $T_{N_k(\omega)}^k(\omega)$, we have $$T_{N_k(\omega)}^k(\omega)\geq \sum_{i=1}^k N_i(\omega)\cdot 2^{m_i}.$$ Hence, by using Stolz's theorem again, we have
	\begin{equation*}
		\begin{split}
		\limsup_{k\to\infty}\frac{\sum_{i=1}^{2k+1} N_i(\omega)m_i}{T_{N_{2k+1}(\omega)}^{2k+1}(\omega)}&\leq \lim_{k\to\infty}\frac{\sum_{i=1}^{2k+1} N_i(\omega)m_i}{\sum_{i=1}^{2k+1}N_i(\omega)\cdot 2^{m_i}}\\&= \lim_{k\to\infty}\frac{N_{2k+1}(\omega)m_{2k+1}}{N_{2k+1}(\omega)\cdot 2^{m_{2k+1}}}=0.
		\end{split}
	\end{equation*}
  Therefore, we conclude that
  \begin{equation}\label{quotient term 4}
  	\lim_{k\to\infty}\frac{\sum_{i=1}^{2k+1} N_i(\omega)m_i}{T_{N_{2k+1}(\omega)}^{2k+1}(\omega)}=0.
  \end{equation}
	Finally, using  (\ref{estimate of k bad point}) and \eqref{quotient term 1}, we have
	\begin{align*}
		&\ \ \ \ \limsup_{k\to\infty}\frac{\sum_{i=0}^{2k}l^{i,i+1}(\omega)+\sum_{i=1}^{2k+1}\sum_{j=1}^{N_i(\omega)-1}l_j^i(\omega)}{T_{N_{2k+1}(\omega)}^{2k+1}(\omega)}\\
		&=\limsup_{k\to\infty}\frac{\sum_{i=0}^{2k-1}l^{i,i+1}(\omega)+\sum_{i=1}^{2k}\sum_{j=1}^{N_i(\omega)-1}l_j^i(\omega)+l^{2k,2k+1}(\omega)+\sum_{j=1}^{N_{2k+1}(\omega)-1}l_j^{2k+1}(\omega)}{T_{N_{2k+1}(\omega)}^{2k+1}(\omega)}\\\\
		&\leq\limsup_{k\to\infty}\frac{T_{N_{2k}(\omega)}^{2k}(\omega)+l^{2k,2k+1}(\omega)+\sum_{j=1}^{N_{2k+1}(\omega)-1}l_j^{2k+1}(\omega)}{T_{N_{2k+1}(\omega)}^{2k+1}(\omega)}\\
		&\leq\limsup_{k\to\infty}\frac{T_{N_{2k}(\omega)}^{2k}(\omega)}{T_{N_{2k+1}(\omega)}^{2k+1}(\omega)}+\xi_{2k+1}=0,
	\end{align*}
which implies that
\begin{equation}\label{quotient term 5}
	\lim_{k\to\infty}\frac{\sum_{i=0}^{2k}l^{i,i+1}(\omega)+\sum_{i=1}^{2k+1}\sum_{j=1}^{N_i(\omega)-1}l_j^i(\omega)}{T_{N_{2k+1}(\omega)}^{2k+1}(\omega)}=0.
\end{equation}
Combining \eqref{quotient term 2},\eqref{quotient term 3},\eqref{quotient term 4} and \eqref{quotient term 5} together, we arrive at
\begin{displaymath}
	\lim_{k\to\infty}\frac{R_{k+1}(\omega)}{T_{N_{2k+1}}^{2k+1}(\omega)}=0.
\end{displaymath}

\emph{Step 4: we prove that \eqref{limit of even numbers} and \eqref{limit of odd numbers} holds.}

For any $x\in\mathfrak{I}(\omega)$, there exists some $z\in H_{2k}(\omega)$, such that
\begin{equation*}
	d_{\omega}^{T_{N_{2k}(\omega)}^{2k}(\omega)}(z,x)<\frac{\eta}{2^{4+2k}}.
\end{equation*}
Therefore, we have
\begin{displaymath}
	\begin{split}
		\left|\sum_{i=0}^{T_{N_{2k}(\omega)}^{2k}(\omega)-1}\varphi(\Theta^i(\omega,x))-\alpha_0\right|&\leq\sum_{i=0}^{T_{N_{2k}(\omega)}^{2k}(\omega)-1}\left|\varphi(\Theta^i(\omega,x))-\varphi(\Theta^i(\omega,z))\right|+W_{k}(\omega)\\&\leq T_{N_{2k}(\omega)}^{2k}(\omega)\mathrm{var}\left(\varphi,\frac{\eta}{2^{4+2k}}\right)+W_k(\omega)
	\end{split}
\end{displaymath}
Divide both sides  by $T_{N_{2k}(\omega)}^{2k}(\omega)$, and we finally arrive at \eqref{limit of even numbers} as $k\to\infty$. \eqref{limit of odd numbers} can be obtained in the same way. The proof of lemma \eqref{lemma different limits} is complete.
\end{proof}

\begin{proof}[Proof of Lemma \ref{lemma limit measure}]
	It suffices to show that $\{\int_M\psi \dif\mu_{k,\omega}\}_{k=1}^{\infty}$ is a Cauchy sequence. Given $\delta>0$, we can find $K$ sufficiently large such that $\mathrm{var}(\psi,{\eta}/{2^{4+K}})<\delta$. Let $k_2>k_1\geq K$ be any integers, then
	\begin{equation*}
		\int_M\psi \dif\mu_{k_i,\omega}=\frac{1}{\# H_{k_i}(\omega)}\sum_{x\in H_{k_i}(\omega)}\psi(x)\mbox{ for }i=1,2.
	\end{equation*}
	For any $x\in H_{k_1}(\omega),$ denote $Z(x)$ to be the collection of points $z$ in $H_{k_2}(\omega)$ such that $z$ descends from $x$, i.e., there is a sequence $z_{k_1+1}\in H_{k_1+1}(\omega),\dots,z_{k_2-1}\in H_{k_2-1}(\omega)$ satisfying
	\begin{equation}\label{def of descend series}
		\begin{split}
			z&=z(z_{k_2-1},y_{k_2}),\ \exists\ y_{k_2}\in D_{k_2}(\theta^{T_0^{k_2}(\omega)}\omega),\\
			z_{k_2-1}&=z(z_{k_2-2}, y_{k_2-1}),\ \exists
			\ y_{k_2-1}\in D_{k_2}(\theta^{T_0^{k_2-1}(\omega)}\omega),\\
			&\cdots\\
			z_{k_1+1}&=z(x,y_{k_1+1}),\ \exists\ y_{k_1+1}\in D_{k_1+1}(\theta^{T_0^{k_1+1}(\omega)}\omega).
		\end{split}
	\end{equation}
	It follows that $\#Z(x)=\# D_{k_1+1}(\theta^{T_0^{k_1+1}(\omega)}\omega)\cdots\# D_{k_2}(\theta^{T_0^{k_2}(\omega)}\omega)$, and by \eqref{number Hk},
	\begin{equation}\label{Hk2=HK1ZX}
		\# H_{k_2}(\omega)=\# H_{k_1}(\omega)\cdot\# Z(x),\quad\forall x\in H_{k_1}(\omega).
	\end{equation}
	By \eqref{z shadow xy}, for any $z\in Z(x),$ we have
	\begin{equation}\label{distance descend}
		\begin{split}
			d_M(x,z)
			&\leq d_{\omega}^{T_{N_{k_1}(\omega)}^{k_1}(\omega)}(x,z)\\
			&\leq d_{\omega}^{T_{N_{k_1}(\omega)}^{k_1}(\omega)}(x,z_{k_1+1})+d_{\omega}^{T_{N_{k_1+1}(\omega)}^{k_1+1}(\omega)}(z_{k_1+1},z_{k_1+2})+\cdots+ d_{\omega}^{T_{N_{k_2-1}(\omega)}^{k_2-1}(\omega)}(z_{k_2-1},z)\\
			&\leq \frac{\eta}{2^{4+k_1+1}}+\frac{\eta}{2^{4+k_1+2}}+\cdots+\frac{\eta}{2^{4+k_2}}\\
			&\leq\frac{\eta}{2^{4+k_1}}.
		\end{split}
	\end{equation}
	Hence
	\begin{displaymath}
		\begin{split}
			\left|\int_M\psi\dif\mu_{k_1,\omega}-\int_M\psi \dif\mu_{k_2,\omega}\right|&\overset{\eqref{Hk2=HK1ZX}}\leq \frac{1}{\# H_{k_2}(\omega)}\sum_{x\in H_{k_1}(\omega)}\sum_{z\in Z(x)}|\psi(x)-\psi(z)|\\&\overset{\eqref{distance descend}}\leq \mathrm{var}\Big(\varphi,\frac{\eta}{2^{4+k_1}}\Big)\leq \mathrm{var}\Big(\varphi,\frac{\eta}{2^{4+K}}\Big)<\delta.
		\end{split}
	\end{displaymath}
	As a consequence, $\{\int_M\psi \dif\mu_{k,\omega}\}_{k=1}^{\infty}$ is a Cauchy sequence. The proof of lemma \ref{lemma limit measure} is complete.
\end{proof}
\begin{proof}[Proof of Lemma \ref{lemma I full measure}]
	For any $k\in\mathbb{N}, p\geq 0$, since $\mathfrak{I}_{k+p}(\omega)\subset \mathfrak{I}_k(\omega)$ and $\mu_{k+p,\omega}(\mathfrak{I}_{k+p}(\omega))=1,$ therefore, $\mu_{k+p,\omega}(\mathfrak{I}_k(\omega))=1.$ Note that $\mathfrak{I}_k(\omega)$ is a closed set, by the Portmanteau theorem, we have
	\begin{displaymath}
		\mu_{\omega}(\mathfrak{I}_k(\omega))\geq \limsup_{p\to\infty}\mu_{k+p,\omega}(\mathfrak{I}_k(\omega))=1.
	\end{displaymath}
	Since $\mathfrak{I}(\omega)=\cap_{k\geq 1}\mathfrak{I}_k(\omega),$ it follows that $\mu_{\omega}(\mathfrak{I}(\omega))=1.$ The proof of lemma \eqref{lemma I full measure} is complete.
\end{proof}

\begin{proof}[Proof of Lemma \ref{lemma entropy distribution}]
	We first define $N(\omega)$ in the statement of Lemma \ref{lemma entropy distribution}.
	
	\emph{Pick of $N(\omega)$.} For $\omega\in\widehat{\Omega}_\gamma$, any $k\in\mathbb{N}$ and $j=\{0,...,N_{k}(\omega)-1\}$, by the previous construction, we have $2^{m_k}\leq \hat{n}_k(\theta^{T_j^k(\omega)}\omega)\leq \hat{n}_k^M$ and $N_k(\omega)\geq 2^{N_1(\omega)(\hat{n}_1^M+m_1)+\cdots+ N_{k-1}(\omega)(\hat{n}_{k-1}^M+m_{k-1})}$. Note also that \eqref{mk tends to infinity} holds. Therefore, for any $\omega\in\widehat{\Omega}_\gamma$, there exists some $k_1(\omega)\in\mathbb{N},$ such that for any $k\geq k_1(\omega)$ and any $i\in\{0,1,...,N_{k+1}(\omega)-1\}$, one has
	\begin{equation}\label{estimation h minus 5gamma and h minus 5.5gamma}
		\begin{split}
			&\quad\exp\left[(h_{\mu_0}(F)-5\gamma)\Big(\sum_{j=0}^{N_1(\omega)-1}\hat{n}_1(\theta^{T_j^1(\omega)}\omega)+\cdots+\sum_{j=0}^{N_k(\omega)-1}\hat{n}_k(\theta^{T_j^k(\omega)}\omega)+\sum_{j=0}^{i}\hat{n}_{k+1}(\theta^{T_j^{k+1}(\omega)}\omega)\Big)\right]\\
			&\geq \exp\left[\Big(h_{\mu_0}(F)-\frac{11\gamma}{2}\Big)\Big(\sum_{j=0}^{N_1(\omega)-1}(\hat{n}_1(\theta^{T_j^1(\omega)}\omega)+m_1)+\cdots+\sum_{j=0}^{N_k(\omega)-1}(\hat{n}_k(\theta^{T_j^k(\omega)}\omega)+m_k)\right.\\
			&\left.\quad+\sum_{j=0}^{i}(\hat{n}_{k+1}(\theta^{T_j^{k+1}(\omega)}\omega)+m_{k+1})\Big)\right].
		\end{split}
	\end{equation}
	Note also that by (\ref{quotient term 1}), we have $T_{N_{k-1}(\omega)}^{k-1}(\omega)/T_{N_{k}(\omega)}^{k}(\omega)\to 0$ as $k\to\infty,$ hence there exists some $k_2(\omega)>k_1(\omega)$, such that for any $k\geq k_2(\omega)$, we have
	\begin{equation}\label{estimation quotient is less than} \frac{T_{N_{k-1}(\omega)}^{k-1}(\omega)}{T_{N_k(\omega)}^k(\omega)}+\xi_k+\xi_{k+1}+\frac{\hat{n}_{k+1}^M+m_{k+1}}{N_k(\omega)}<\frac{{\gamma}/{2}}{(h_{\mu_0}(F)-5\gamma)-\gamma/2}.
	\end{equation}
	Define $N(\omega)=T_{N_{k_2(\omega)}(\omega)}^{k_2(\omega)}(\omega)+1$.
	
	Now, we start proving lemma \ref{lemma entropy distribution} for $n\geq N(\omega)$. Note that $n\geq N(\omega),$ then there exists some $k\geq k_2(\omega)$ such that	$T_{N_k(\omega)}^k(\omega)<n\leq T_{N_{k+1}(\omega)}^{k+1}(\omega).$	For any open set $B_n(\omega,x,{\eta}/{2^4})$, by the weak$^*$ convergence of measure, we have
	\begin{displaymath}
		\begin{split} \mu_{\omega}\left[B_n\left(\omega,x,\frac{\eta}{2^4}\right)\right]\leq & \liminf_{p\to\infty}\mu_{k+p,\omega}\left[B_n\left(\omega,x,\frac{\eta}{2^4}\right)\right].
		\end{split}
	\end{displaymath}
	Next, we wish to estimate
	\begin{equation}\label{mu k+p omega Bn}
		\mu_{k+p,\omega}\left[B_n\left(\omega,x,\frac{\eta}{2^4}\right)\right]=\frac{1}{\# H_{k+p}(\omega)}\cdot \#\Big\{z\in H_{k+p}(\omega):z\in B_n\Big(\omega,x,\frac{\eta}{2^4}\Big)\Big\}.
	\end{equation}
    Our proof is divided into $3$ cases corresponding to the value of $n$.
	\begin{enumerate}
		\item [Case 1.]$T_{N_k(\omega)}^k(\omega)<n\leq T_{N_k(\omega)}^k+l^{k,k+1}(\omega)+m_{k+1}=T_{0}^{k+1}(\omega);$
		\item[Case 2.] There exists $j\in \{0,\dots,N_{k+1}(\omega)-1\},$ such that $T_{j}^{k+1}(\omega)<n\leq T_j^{k+1}(\omega)+\hat{n}_{k+1}(\theta^{T_j^{k+1}(\omega)}\omega);$
		\item[Case 3.] There exists   $j\in\{0,\dots,N_{k+1}(\omega)-1\},$ such that $$T_j^{k+1}(\omega)+\hat{n}_{k+1}(\theta^{T_j^{k+1}(\omega)}\omega)<n\leq T_{j+1}^{k+1}(\omega)=T_j^{k+1}+\hat{n}_{k+1}(\theta^{T_j^{k+1}(\omega)}\omega)+m_{k+1}+l^{k+1}_{j+1}(\omega).$$
	\end{enumerate}
	
	\emph{Proof of Case 1.} We divide the proof of case $1$ into $3$ steps.

	Firstly, we show $\#\{z\in H_k(\omega):z\in B_n(\omega,x,{\eta}/{2^4})\}\leq 1$. In fact, if there are $z_1\neq z_2\in H_k(\omega),$ such that $z_1,z_2\in B_n(\omega,x,{\eta}/{2^4}),$ then
	\begin{equation*}
		d_\omega^n(z_1,z_2)<\frac{\eta}{2^4}\times 2=\frac{\eta}{2^3}.
	\end{equation*}
	However, we have $d_{\omega}^n(z_1,z_2)\geq d_{\omega}^{T_{N_k(\omega)}^k(\omega)}(z_1,z_2)> {3\eta}/{8}$ by \eqref{Hk are 3eta/8 separated}, which leads to a contradiction.
	
	Secondly, we show for any $p\geq 1 $,
	\begin{equation}\label{case 1 step 2}
		\#\left\{z\in H_{k+p}(\omega):z\in B_n(\omega,x,\frac{\eta}{2^4})\right\}\leq \# D_{k+1}(\theta^{T_0^{k+1}(\omega)}\omega)\cdots\# D_{k+p}(\theta^{T_0^{k+p}(\omega)}\omega).
	\end{equation}
	If there are different points $z_1, z_2\in H_{k+p}(\omega)\cap B_n(\omega,x,{\eta}/{2^4})$ such that $z_1$ descends from $x_1\in H_k(\omega)$ and $z_2$ descends from $x_2\in H_k(\omega)$  defined as \eqref{def of descend series}, then we claim that $x_1=x_2$. In fact, if $x_1\neq x_2,$ by \eqref{Hk are 3eta/8 separated}, we have $d_{\omega}^{T_{N_k(\omega)}^k(\omega)}(x_1,x_2)>{3\eta}/{8}.$ However, we also have
	\begin{align*}
		d_{\omega}^{T_{N_k(\omega)}^k(\omega)}(x_1,x_2)&\leq d_{\omega}^{T_{N_k(\omega)}^k(\omega)}(x_1,z_1)+d_{\omega}^{T_{N_k(\omega)}^k(\omega)}(z_1,x)+d_{\omega}^{T_{N_k(\omega)}^k(\omega)}(x,z_2)+d_{\omega}^{T_{N_k(\omega)}^k(\omega)}(z_2,x_2)\\
		&\overset{\eqref{distance descend}}< \frac{\eta}{2^{4+k}}+\frac{\eta}{2^{4}}+\frac{\eta}{2^{4}}+\frac{\eta}{2^{4+k}}\\
		&\leq \frac{\eta}{4},
	\end{align*}which leads to a contradiction.
	As a consequence of \eqref{mu k+p omega Bn} and \eqref{case 1 step 2}, we have
	\begin{displaymath} \mu_{k+p,\omega}\left[B_n\left(\omega,x,\frac{\eta}{2^4}\right)\right]\leq \frac{\# D_{k+1}(\theta^{T_0^{k+1}}\omega)\cdots\# D_{k+p}(\theta^{T_0^{k+p}}\omega)}{\# H_{k+p}(\omega)}=\frac{1}{\# H_k(\omega)}.
	\end{displaymath}
	
	Thirdly, we claim that $\# H_k(\omega)\geq \exp[n(h_{\mu_0}(F)-6\gamma)]$. In fact, we have
	\begin{align*}
		\# H_k(\omega)&=\# D_1(\omega)\cdot \# D_2(\theta^{T_0^2(\omega)}\omega)\cdots\# D_k(\theta^{T_0^k(\omega)}\omega)=\prod_{i=1}^{k}\prod_{j=0}^{N_i(\omega)-1}\# C_i(\theta^{T_j^i(\omega)}\omega)\\
		&= \prod_{i=1}^{k}\prod_{j=0}^{N_i-1}M\left(\alpha_{\rho(i)},4\delta_i,\hat{n}_i(\theta^{T_j^i(\omega)}\omega),\frac{\eta}{2},\theta^{T_j^i(\omega)}\omega\right)\\&\overset{\eqref{estimation for odd even k}}\geq \prod_{i=1}^{k}\prod_{j=0}^{N_i-1}\exp [\hat{n}_i(\theta^{T_j^i(\omega)}\omega)(h_{\mu_0}(F)-5\gamma)]\\
		&\overset{\eqref{estimation h minus 5gamma and h minus 5.5gamma}}\geq \exp\left[\left(h_{\mu_0}(F)-\frac{11\gamma}{2}\right)\Big(\sum_{j=0}^{N_1-1}(\hat{n}_1(\theta^{T_j^1(\omega)}\omega)+m_1)+\cdots+\sum_{j=0}^{N_k-1}(\hat{n}_k(\theta^{T_j^k(\omega)}\omega)+m_k)\Big)\right].
	\end{align*}
Notice the following fact
	\begin{align*}
		&\quad  n-\Big(\sum_{j=0}^{N_1(\omega)-1}(\hat{n}_1(\theta^{T_j^1(\omega)}\omega)+m_1)+\cdots+\sum_{j=0}^{N_k(\omega)-1}(\hat{n}_k(\theta^{T_j^k(\omega)}\omega)+m_k)\Big)\\
		&\leq  T_{N_{k-1}(\omega)}^{k-1}(\omega)+\xi_kT_{N_k(\omega)}^k(\omega)+m_{k+1}+\xi_{k+1}n,
	\end{align*}and therefore,
	\begin{align*}
		&\quad\frac{n-\left(\sum_{j=0}^{N_1(\omega)-1}(\hat{n}_1(\theta^{T_j^1(\omega)}\omega)+m_1)+\cdots+\sum_{j=0}^{N_k(\omega)-1}(\hat{n}_k(\theta^{T_j^k(\omega)}\omega)+m_k)\right)}{n}\\
		&\leq \frac{T_{N_{k-1}(\omega)}^{k-1}(\omega)+\xi_kT_{N_k(\omega)}^k(\omega)+m_{k+1}+\xi_{k+1}n}{n}\\
		&\leq  \frac{T_{N_{k-1}(\omega)}^{k-1}(\omega)}{T_{N_{k}(\omega)}^{k}(\omega)}+\xi_k+\xi_{k+1}+\frac{m_{k+1}}{N_k(\omega)}\\&
		\overset{\eqref{estimation quotient is less than}}{\leq} \frac{\gamma/2}{h_{\mu_0}(F)-{11}/{2}\gamma}.
	\end{align*}So we have
	\begin{equation*}
		\Big(h_{\mu_0}(F)-\frac{11\gamma}{2}\Big)\Big(\sum_{j=0}^{N_1-1}(\hat{n}_1(\theta^{T_j^1(\omega)}\omega)+m_1)+\cdots+\sum_{j=0}^{N_k-1}(\hat{n}_k(\theta^{T_j^k(\omega)}\omega)+m_k)\Big)\geq n(h_{\mu_0}(F)-6\gamma).
	\end{equation*}As a consequence, we obtain $\#H_k(\omega)\geq \exp[n(h_{\mu_0}(F)-6\gamma)]$.
	Hence, in case 1, we have
	\begin{equation*}
		\mu_{k+p,\omega}\left[B_n\left(\omega,x,\frac{\eta}{2^4}\right)\right]\leq \exp[-n(h_{\mu_0}(F)-6\gamma)].
	\end{equation*}

	\emph{Proof of Case 2.} We divide the proof of case $2$ into $4$ steps.
	
	Firstly, we show that $\#\{z\in H_k(\omega):z\in B_n(\omega,x,{\eta}/{2^3})\}\leq 1.$ If there are $z_1\neq z_2\in H_k(\omega)\cap B_n(\omega,x,{\eta}/{2^3}),$ then
	\begin{displaymath}
		d_\omega^n(z_1,z_2)<2\times \frac{\eta}{2^3}=\frac{\eta}{4}.
	\end{displaymath}
	But by \eqref{Hk are 3eta/8 separated}, we have $d_\omega^n(z_1,z_2)\geq d_{\omega}^{T_{N_k(\omega)}^k(\omega)}(z_1,z_2)> {3\eta}/{8},$ which leads to a contradiction.
	
	Secondly, we prove that $$\#\{z\in H_{k+1}(\omega):z\in B_n(\omega,x,{\eta}/{2^3})\}\leq \prod_{i=j}^{N_{k+1}(\omega)-1} \# C_{k+1}(\theta^{T_i^{k+1}(\omega)}\omega).$$ If there exist $z_1\neq z_2\in H_{k+1}(\omega)\cap B_n(\omega,x,{\eta}/{2^3}),$ with
	\begin{displaymath}
		\begin{split}
			z_1&=z(x_1,y_1), x_1\in H_k(\omega),y_1\in D_{k+1}(\theta^{T_0^{k+1}(\omega)}\omega), y_1=y\big(a_0^{k+1},\dots,a_{N_{k+1}(\omega)-1}^{k+1}\big),\\	
			z_2&=z(x_2,y_2), x_2\in H_k(\omega),y_2\in D_{k+1}(\theta^{T_0^{k+1}(\omega)}\omega), y_2=y\big(b_0^{k+1},\dots,b_{N_{k+1}(\omega)-1}^{k+1}\big),
		\end{split}
	\end{displaymath}
	where $a_i^{k+1},b_i^{k+1}\in C_{k+1}(\theta^{T_i^{k+1}(\omega)}\omega)$ for $i\in\{0,\dots,N_{k+1}(\omega)-1\}.$ We claim that $x_1=x_2$ and $a_i^{k+1}=b_i^{k+1}$ for $i\in \{0,1,\dots,j-1\}.$
	In fact, if $x_1\neq x_2\in H_k(\omega)$, by \eqref{Hk are 3eta/8 separated}, we have $d_{\omega}^{T_{N_k(\omega)}^k(\omega)}(x_1,x_2)>{3\eta}/{8},$. But
	\begin{displaymath}
		\begin{split}
			d_{\omega}^{T_{N_k(\omega)}^k(\omega)}(x_1,x_2)&\leq d_{\omega}^{T_{N_k(\omega)}^k(\omega)}(x_1,z_1)+d_{\omega}^{T_{N_k(\omega)}^k(\omega)}(z_1,x)+d_{\omega}^{T_{N_k(\omega)}^k(\omega)}(x,z_2)+d_{\omega}^{T_{N_k(\omega)}^k(\omega)}(z_2,x_2)\\
			&\overset{\eqref{z shadow xy}}\leq \frac{\eta}{2^{4+k+1}}+\frac{\eta}{2^3}+\frac{\eta}{2^3}+\frac{\eta}{2^{4+k+1}}\\&\leq \frac{5\eta}{16},
		\end{split}
	\end{displaymath}
	which leads to a contradiction. Hence $x_1=x_2.$ Next, we prove that $a_i^{k+1}=b_i^{k+1}$ for $i\in\{0,1,\dots,j-1\}$. If $j=0,$ there is nothing to prove. Suppose $j\geq 1$ and there exists $i$ with $0\leq i\leq j-1$ such that $a_{i}^{k+1}\neq b_i^{k+1}.$ On the one hand, by \eqref{eq y shadowing xk+1} and \eqref{z shadow xy}, one has
	\begin{displaymath}
		\begin{split}
			&\quad d_{\theta^{T_i^{k+1}(\omega)}\omega}^{\hat{n}_{k+1}(\theta^{T_i^{k+1}(\omega)}\omega)}\Big(a_i^{k+1},b_i^{k+1}\Big)\\
			&\leq d_{\theta^{T_i^{k+1}(\omega)}\omega}^{\hat{n}_{k+1}(\theta^{T_i^{k+1}(\omega)}\omega)}\Big(a_i^{k+1},F_{\theta^{T_0^{k+1}(\omega)}\omega}^{T_i^{k+1}(\omega)-T_0^{k+1}(\omega)}y_1\Big)+d_{\theta^{T_i^{k+1}(\omega)}\omega}^{\hat{n}_{k+1}(\theta^{T_i^{k+1}(\omega)}\omega)}\Big(F_{\theta^{T_0^{k+1}(\omega)}\omega}^{T_i^{k+1}(\omega)-T_0^{k+1}(\omega)}y_1,F_{\omega}^{T_i^{k+1}(\omega)}z_1\Big)\\
			&\quad+  d_{\theta^{T_i^{k+1}(\omega)}\omega}^{\hat{n}_{k+1}(\theta^{T_i^{k+1}(\omega)}\omega)}\Big(F_{\omega}^{T_i^{k+1}(\omega)}z_1,F_{\omega}^{T_i^{k+1}(\omega)}z_2\Big)+  d_{\theta^{T_i^{k+1}(\omega)}\omega}^{\hat{n}_{k+1}(\theta^{T_i^{k+1}(\omega)}\omega)}\Big(F_{\omega}^{T_i^{k+1}(\omega)}z_2,F_{\theta^{T_0^{k+1}(\omega)}\omega}^{T_i^{k+1}(\omega)-T_0^{k+1}(\omega)}y_2\Big)   \\&\quad+   d_{\theta^{T_i^{k+1}(\omega)}\omega}^{\hat{n}_{k+1}(\theta^{T_i^{k+1}(\omega)}\omega)}\Big(F_{\theta^{T_0^{k+1}(\omega)}\omega}^{T_i^{k+1}(\omega)-T_0^{k+1}(\omega)}y_2,b_i^{k+1}\Big)\\
			&\leq \frac{\eta}{2^{4+k+1}}+\frac{\eta}{2^{4+k+1}}+\frac{\eta}{2^3}\times 2+\frac{\eta}{2^{4+k+1}}+\frac{\eta}{2^{4+k+1}}\leq \frac{\eta}{2^2}+\frac{\eta}{2^{3+k}}<\frac{3\eta}{8}.
		\end{split}
	\end{displaymath}
	On the other hand, $a_i^{k+1}\neq b_i^{k+1}\in C_{k+1}(\theta^{T_i^{k+1}}\omega)$ are $(\theta^{T_i^{k+1}}\omega,{\eta}/{2},\hat{n}_{k+1}(\theta^{T_i^{k+1}}\omega))$-separated, which leads to a contradiction. Hence, there are at most $\prod_{i=j}^{N_{k+1}(\omega)-1} \# C_{k+1}(\theta^{T_i^{k+1}(\omega)}\omega)$ points lying in $H_{k+1}(\omega)\cap B_n(\omega,x,{\eta}/{2^3}).$
	
	Thirdly, we show for any $p\geq 1$,
	\begin{equation*}
		\# \left\{z\in H_{k+p}(\omega):z\in B_n(\omega,x,\frac{\eta}{2^4})\right\}\leq \left( \prod_{i=j}^{N_{k+1}(\omega)-1}\# C_{k+1}(\theta^{T_i^{k+1}(\omega)}\omega)\right)\cdot \left(\prod_{i=2}^{p}\# D_{k+i}(\theta^{T_0^{k+i}(\omega)}\omega)\right).
	\end{equation*} We prove it by showing that $z\in H_{k+p}(\omega)\cap B_n(\omega,x,{\eta}/{2^4})$ must descend from the points of $H_{k+1}(\omega)\cap B_n(\omega,x,{\eta}/{2^3}).$ Suppose that we have $z_1\in H_{k+1}(\omega)$ and $z_p\in H_{k+p}(\omega)\cap B_{n}(\omega,x,{\eta}/{2^4}),$ where $z_p$ descends from $z_1$.
	Denote $z_p=z(z_{p-1},y_p)$ for $z_{p-1}\in H_{k+p-1}(\omega),y_p\in D_{k+p}(\theta^{T_0^{k+p}}\omega),\dots,z_2=z(z_1,y_2)$ for $y_2\in D_{k+2}(\theta^{T_0^{k+2}}\omega).$ Then by \eqref{z shadow xy}, one has
	\begin{displaymath}
		\begin{split}
			d_{\omega}^n(z_1,z_p)&\leq d_{\omega}^{T_{N_{k+1}(\omega)}^{k+1}(\omega)}(z_1,z_2)+d_{\omega}^{T_{N_{k+2}(\omega)}^{k+2}(\omega)}(z_2,z_3)+\cdots+ d_{\omega}^{T_{N_{k+p-1}(\omega)}^{k+p-1}(\omega)}(z_{p-1},z_p)\\
			&\leq \frac{\eta}{2^{4+k+2}}+\frac{\eta}{2^{4+k+3}}+\cdots+\frac{\eta}{2^{k+p}}\\
			&\leq\frac{\eta}{2^{4+k+1}}.
		\end{split}
	\end{displaymath}
	Hence $d_{\omega}^n(x,z_1)\leq d_{\omega}^n(x,z_p)+d_{\omega}^n(z_1,z_p)<{\eta}/{2^4}+{\eta}/{2^{4+k+1}}<{\eta}/{2^3},$ which implies that $z_1\in B_n(\omega,x,{\eta}/{2^3}).$ Therefore
	\begin{displaymath}
		\begin{split}
			&\quad	\# \Big\{z\in H_{k+p}(\omega):z\in B_n\Big(\omega,x,\frac{\eta}{2^4}\Big)\Big\}\\
			&\leq \# H_{k+1}(\omega)\cap B_n\left(\omega,x,\frac{\eta}{2^3}\right)\# D_{k+2}(\theta^{T_0^{k+2}(\omega)}\omega)\cdots\# D_{k+p}(\theta^{T_0^{k+p}(\omega)}\omega)\\
			&\leq\left(\prod_{i=j}^{N_{k+1}(\omega)-1}\# C_{k+1}(\theta^{T_j^{k+1}(\omega)}\omega)\right)\cdot\prod_{i=2}^p D_{k+i}(\theta^{T_0^{k+i}(\omega)}\omega),
		\end{split}
	\end{displaymath}
	It follows by \eqref{number Dk+1} and \eqref{number Hk} that
	\begin{equation}\label{estimation muk+p is less than fraction}
		\begin{split}
			\mu_{k+p,\omega}\left[B_n\left(\omega,x,\frac{\eta}{2^4}\right]\right)&\leq\frac{\left(\prod_{i=j}^{N_{k+1}(\omega)-1}\# C_{k+1}(\theta^{T_i^{k+1}(\omega)}\omega)\right)\cdot\left(\prod_{i=2}^{p}D_{k+i}(\theta^{T_0^{k+i}(\omega)}\omega)\right)}{\# H_{k+p}(\omega)}\\&=\frac{1}{\# H_k(\omega)\prod_{i=0}^{j-1}\# C_{k+1}(\theta^{T_i^{k+1}(\omega)}\omega)},
		\end{split}
	\end{equation}with the convention that $\prod_{i=0}^{-1}=1.$
	
	Fourthly, we show
	\begin{equation}\label{estimate number hkck1}
		\# H_k(\omega)\prod_{i=0}^{j-1}\# C_{k+1}(\theta^{T_i^{k+1}}\omega)\geq \exp[n(h_{\mu_0}(F)-6\gamma)].
	\end{equation}  In fact, we have
    \begin{small}
    		\begin{align*}
    		&\quad\# H_k(\omega)\cdot\prod_{i=0}^{j-1} C_{k+1}(\theta^{T_i^{k+1}}\omega)=\left(\prod_{t=1}^{k}\prod_{i=0}^{N_t(\omega)-1}\# C_t(\theta^{T_i^t(\omega)}\omega)\right)\cdot \left(\prod_{i=0}^{j-1} \# C_{k+1}(\theta^{T_i^{k+1}(\omega)}\omega)\right)\nonumber\\
    		&=\prod_{t=1}^{k}\prod_{i=0}^{N_t(\omega)-1}M\left(\alpha_{\rho(t)},4\delta_t,\hat{n}_t(\theta^{T_i^t(\omega)}\omega),\frac{\eta}{2},\theta^{T_i^t(\omega)}\omega\right)\cdot \prod_{i=0}^{j-1}M\left(\alpha_{\rho(k+1)},4\delta_{k+1},\hat{n}_{k+1}(\theta^{T_i^{k+1}(\omega)}\omega),\frac{\eta}{2},\theta^{T_i^{k+1}(\omega)}\omega\right) \nonumber\\
    		&\overset{\eqref{estimation for odd even k}}{\geq}\exp\left[(h_{\mu_0}(F)-5\gamma)\Big(\sum_{i=0}^{N_1-1}\hat{n}_1(\theta^{T_i^1(\omega)}\omega)+\cdots\right.\\&\left.	\quad\quad\quad\quad\quad\quad\quad\quad\quad\quad\quad\quad\quad\quad+\sum_{i=0}^{N_k-1}\hat{n}_k(\theta^{T_i^k(\omega)}\omega)+\sum_{i=0}^{j-1}\hat{n}_{k+1}(\theta^{T_i^{k+1}(\omega)}\omega)\Big)\right]\nonumber\\
    		&\overset{\eqref{estimation h minus 5gamma and h minus 5.5gamma}}\geq \exp\left[\Big(h_{\mu_0}(F)-\frac{11\gamma}{2}\Big)\Big(\sum_{i=0}^{N_1(\omega)-1}(\hat{n}_1(\theta^{T_i^1(\omega)}\omega)+m_1)+\cdots+\sum_{i=0}^{N_k(\omega)-1}(\hat{n}_k(\theta^{T_i^k(\omega)}\omega)+m_k)\right.\nonumber\\&\left.
    		\quad\quad\quad\quad\quad\quad\quad\quad\quad\quad\quad\quad\quad\quad+\sum_{i=0}^{j-1}(\hat{n}_{k+1}(\theta^{T_i^{k+1}(\omega)}\omega)+m_{k+1})\Big)\right].
    	\end{align*}
    \end{small}
Notice the following fact
	\begin{align*}
		&\quad n-\left(\sum_{i=0}^{N_1(\omega)-1}(\hat{n}_1(\theta^{T_i^1(\omega)}\omega)+m_1)+\cdots+\sum_{i=0}^{N_k(\omega)-1}(\hat{n}_k(\theta^{T_i^k(\omega)}\omega)+m_k)\right.\\
		&\quad\quad\left.+\sum_{i=0}^{j-1}(\hat{n}_{k+1}(\theta^{T_i^{k+1}(\omega)}\omega)+m_{k+1})\right)\\
		 &\leq T_{N_{k-1}(\omega)}^{k-1}(\omega)+l^{k-1,k}(\omega)+\sum_{i=1}^{N_k(\omega)-1}l_i^k(\omega)+l^{k,k+1}(\omega)+\sum_{i=1}^{j}l_i^{k+1}(\omega)+n-T_{j}^{k+1}(\omega)\nonumber\\
		&\leq T_{N_{k-1}(\omega)}^{k-1}(\omega)+\xi_kn+\xi_{k+1}n+\hat{n}_{k+1}^M
		\leq \left(\frac{T_{N_{k-1}(\omega)}^{k-1}(\omega)}{T_{N_{k}(\omega)}^{k}(\omega)}+\xi_k+\xi_{k+1}+\frac{\hat{n}_{k+1}^M}{N_k(\omega)}\right)n\\&
		\overset{\eqref{estimation quotient is less than}}{\leq}  \frac{\gamma/2}{h_{\mu_0}(F)-{11\gamma}/{2}}\cdot n,
	\end{align*}
	with the convention that $\sum_{i=1}^{0}=0$. As a consequence, we have
    	\begin{align*}
			&\quad \Big(h_{\mu_0}(F)-\frac{11\gamma}{2}\Big)\Big(\sum_{i=0}^{N_1(\omega)-1}(\hat{n}_1(\theta^{T_i^1(\omega)}\omega)+m_1)+\cdots+\sum_{i=0}^{N_k(\omega)-1}(\hat{n}_k(\theta^{T_i^k(\omega)}\omega)+m_k)\nonumber\\&\quad\quad\quad\quad\quad\quad\quad\quad\quad\quad\quad\quad\quad+
			\sum_{i=0}^{j-1}(\hat{n}_{k+1}(\theta^{T_i^{k+1}(\omega)}\omega)+m_{k+1})\Big)\geq n(h_{\mu_0}(F)-6\gamma)\nonumber.
		\end{align*}
 Therefore, in case 2, by \eqref{estimation muk+p is less than fraction} and \eqref{estimate number hkck1}, we have

	\begin{displaymath} \mu_{k+p,\omega}\left[B_n\left(\omega,x,\frac{\eta}{2^4}\right)\right]\leq \exp[-n(h_{\mu_0}(F)-6\gamma)].
	\end{displaymath}

	\emph{Proof of Case 3.}
	We also divide our proof into $4$ steps.
	
	Firstly, exactly the same as the proof of step 1 in case 2, we can prove that $\#\{z\in H_k(\omega):z\in B_n(\omega,x,{\eta}/{2^3})\}\leq 1.$
	
	Secondly, we show that $\#\{z\in H_{k+1}(\omega):z\in B_n(\omega,x,{\eta}/{2^3})\}\leq \prod_{i=j+1}^{N_{k+1}(\omega)-1} \# C_{k+1}(\theta^{T_i^{k+1}(\omega)}\omega)$ with the convention that $\prod_{i=N_{k+1}(\omega)}^{N_{k+1}(\omega)-1}=1$. If there are $z_1\neq z_2\in H_{k+1}(\omega)\cap B_n(\omega,x,{\eta}/{2^3}),$ with
	\begin{displaymath}
		\begin{split}
			z_1&=z(x_1,y_1), x_1\in H_k(\omega),y_1\in D_{k+1}(\theta^{T_0^{k+1}(\omega)}\omega), y_1=y\big(a_0^{k+1},\dots,a_{N_{{k+1}}(\omega)-1}^{k+1}\big),\\	
			z_2&=z(x_2,y_2), x_2\in H_k(\omega),y_2\in D_{k+1}(\theta^{T_0^{k+1}(\omega)}\omega), y_2=y\big(b_0^{k+1},\dots,b_{N_{{k+1}}(\omega)-1}^{k+1}\big),
		\end{split}
	\end{displaymath}
	where $a_i^{k+1},b_i^{k+1}\in C_{k+1}(\theta^{T_i^{k+1}(\omega)}\omega)$ for $i\in\{0,\dots,N_{k+1}(\omega)-1\}.$ We claim that $x_1=x_2$ and $a_i^{k+1}=b_i^{k+1}$ for $i\in \{0,1,\dots,j\}.$ In fact, the same proof as in step 2 of case 2 indicates that $x_1=x_2$ and $a_i^{k+1}=b_i^{k+1}$ for $i\in\{0,1,\dots,j-1\}$. It remains to show that $a_j^{k+1}=b_j^{k+1}.$ If $a_j^{k+1}\neq b_j^{k+1}$, on the one hand, by \eqref{eq y shadowing xk+1} and \eqref{z shadow xy}, one has
	\begin{displaymath}
		\begin{split}
			&\quad d_{\theta^{T_j^{k+1}(\omega)}\omega}^{\hat{n}_{k+1}(\theta^{T_j^{k+1}(\omega)}\omega)}\Big(a_j^{k+1},b_j^{k+1}\Big)\\
			&\leq d_{\theta^{T_j^{k+1}(\omega)}\omega}^{\hat{n}_{k+1}(\theta^{T_j^{k+1}(\omega)}\omega)}\Big(a_j^{k+1},F_{\theta^{T_0^{k+1}(\omega)}\omega}^{T_j^{k+1}(\omega)-T_0^{k+1}(\omega)}y_1\Big)+d_{\theta^{T_j^{k+1}(\omega)}\omega}^{\hat{n}_{k+1}(\theta^{T_j^{k+1}(\omega)}\omega)}\Big(F_{\theta^{T_0^{k+1}(\omega)}\omega}^{T_j^{k+1}(\omega)-T_0^{k+1}(\omega)}y_1,F_{\omega}^{T_j^{k+1}(\omega)}z_1\Big) \\ &\quad+  d_{\theta^{T_j^{k+1}(\omega)}\omega}^{\hat{n}_{k+1}(\theta^{T_j^{k+1}(\omega)}\omega)}\Big(F_{\omega}^{T_j^{k+1}(\omega)}z_1,F_{\omega}^{T_j^{k+1}(\omega)}z_2\Big) +  d_{\theta^{T_j^{k+1}(\omega)}\omega}^{\hat{n}_{k+1}(\theta^{T_j^{k+1}(\omega)}\omega)}\Big(F_{\omega}^{T_j^{k+1}(\omega)}z_2,F_{\theta^{T_0^{k+1}}\omega}^{T_j^{k+1}(\omega)-T_0^{k+1}(\omega)}y_2\Big)\\   &\quad+   d_{\theta^{T_i^{k+1}(\omega)}\omega}^{\hat{n}_{k+1}(\theta^{T_j^{k+1}(\omega)}\omega)}\Big(F_{\theta^{T_0^{k+1}(\omega)}\omega}^{T_j^{k+1}(\omega)-T_0^{k+1}(\omega)}y_2,b_j^{k+1}\Big)\\
			&\leq \frac{\eta}{2^{4+k+1}}+\frac{\eta}{2^{4+k+1}}+\frac{\eta}{2^3}\times 2+\frac{\eta}{2^{4+k+1}}+\frac{\eta}{2^{4+k+1}}<\frac{3\eta}{8}.
		\end{split}
	\end{displaymath}
	On the other hand, $a_j^{k+1}\neq b_j^{k+1}\in C_{k+1}(\theta^{T_j^{k+1}(\omega)}\omega)$ are $(\theta^{T_j^{k+1}(\omega)}\omega,{\eta}/{2},\hat{n}_{k+1}(\theta^{T_j^{k+1}(\omega)}\omega))$-separated, which leads to a contradiction.
	
	Thirdly, we show that for any $p\geq 2$,
	\begin{equation*}
		\# \left\{z\in H_{k+p}(\omega):z\in B_n(\omega,x,\frac{\eta}{2^4})\right\}\leq \left( \prod_{i=j+1}^{N_{k+1}(\omega)-1}\# C_{k+1}(\theta^{T_i^{k+1}(\omega)}\omega)\right)\cdot\left(\prod_{i=2}^{p}\# D_{k+i}(\theta^{T_0^{k+i}(\omega)}\omega)\right),
	\end{equation*}with the convention that $\prod_{i=N_{k+1}(\omega)}^{N_{k+1}(\omega)-1}=1$.
	Exactly same proof as in step $3$ of cases $2$ indicates that $z\in H_{{k+p}}(\omega)\cap B_n(\omega,x,{\eta}/{2^4})$ must descends from some point in $H_{k+1}(\omega)\cap B_n(\omega,x,{\eta}/{2^3}).$
	Therefore,
	\begin{displaymath}
		\begin{split}
			&\quad\# \Big\{z\in H_{k+p}(\omega):z\in B_n\Big(\omega,x,\frac{\eta}{2^4}\Big)\Big\}\\
			&\leq \# H_{k+1}(\omega)\cap B_n\left(\omega,x,\frac{\eta}{2^3}\right)\# D_{k+2}(\theta^{T_0^{k+2}(\omega)}\omega)\cdots\# D_{k+p}(\theta^{T_0^{k+p}(\omega)}\omega)\\
			&\leq\left(\prod_{i=j+1}^{N_{k+1}(\omega)-1}\# C_{k+1}(\theta^{T_j^{k+1}(\omega)}\omega)\right)\cdot\left(\prod_{i=2}^p D_{k+i}(\theta^{T_0^{k+i}(\omega)}\omega)\right).
		\end{split}
	\end{displaymath}
	It follows that
	\begin{equation}
		\begin{split}
			\mu_{k+p,\omega}\left[B_n\left(\omega,x,\frac{\eta}{2^4}\right)\right]&\leq\frac{\prod_{i=j+1}^{N_{k+1}(\omega)-1}\# C_{k+1}(\theta^{T_i^{k+1}(\omega)}\omega)\cdot\prod_{i=2}^{p}D_{k+i}(\theta^{T_0^{k+i}(\omega)}\omega)}{\# H_{k+p}(\omega)}\\&=\frac{1}{\# H_k(\omega)\prod_{i=0}^{j}\# C_{k+1}(\theta^{T_i^{k+1}(\omega)}\omega)}.
		\end{split}
	\end{equation}
	
	Fourthly, we show $\# H_k(\omega)\prod_{i=0}^{j}\# C_{k+1}(\theta^{T_i^{k+1}(\omega)}\omega)\geq \exp((h_{\mu_0}(F)-5\gamma)n).$
	Using (\ref{estimation h minus 5gamma and h minus 5.5gamma}), we have
	\begin{small}
		\begin{displaymath}
			\begin{split}
				&\quad\# H_k(\omega)\cdot\prod_{i=0}^{j} C_{k+1}(\theta^{T_i^{k+1}(\omega)}\omega)=\prod_{t=1}^{k}\prod_{i=0}^{N_t(\omega)-1}\# C_t(\theta^{T_i^t(\omega)}\omega)\cdot \prod_{i=0}^{j} \# C_{k+1}(\theta^{T_i^{k+1}(\omega)}\omega)\nonumber\\
				&=\prod_{t=1}^{k}\prod_{i=0}^{N_t(\omega)-1}M\left(\alpha_{\rho(t)},4\delta_t,\hat{n}_t(\theta^{T_i^t(\omega)}\omega),\frac{\eta}{2},\theta^{T_i^t(\omega)}\omega\right)\cdot \prod_{i=0}^{j}M(\alpha_{\rho(k+1)},4\delta_{k+1},\hat{n}_{k+1}(\theta^{T_i^{k+1}(\omega)}\omega),\frac{\eta}{2},\theta^{T_i^{k+1}(\omega)}\omega) \\
				&\geq\exp\left[(h_{\mu_0}(F)-5\gamma)\Big(\sum_{i=0}^{N_1(\omega)-1}\hat{n}_1(\theta^{T_i^1(\omega)}\omega)+\cdots+\sum_{i=0}^{N_k(\omega)-1}\hat{n}_k(\theta^{T_i^k(\omega)}\omega)+\sum_{i=0}^{j}\hat{n}_{k+1}(\theta^{T_i^{k+1}(\omega)}\omega)\Big)\right]\\&\geq \exp\left[\Big(h_{\mu_0}(F)-\frac{11\gamma}{2}\Big)\Big(\sum_{i=0}^{N_1(\omega)-1}(\hat{n}_1(\theta^{T_i^1(\omega)}\omega)+m_1)+\cdots+\sum_{i=0}^{N_k(\omega)-1}(\hat{n}_k(\theta^{T_i^k(\omega)}\omega)+m_k)\right.\\
				&\left.\quad+\sum_{i=0}^{j-1}(\hat{n}_{k+1}(\theta^{T_i^{k+1}(\omega)}\omega)+m_{k+1})+\hat{n}_{k+1}(\theta^{T_j^{k+1}(\omega)}\omega)\Big)\right].
			\end{split}
		\end{displaymath}
	\end{small}
	 We notice the following fact that in Case 3
	\begin{displaymath}
		\begin{split}
			&\quad n-\left(\sum_{i=0}^{N_1(\omega)-1}(\hat{n}_1(\theta^{T_i^1(\omega)}\omega)+m_1)+\cdots+\sum_{i=0}^{N_k(\omega)-1}(\hat{n}_k(\theta^{T_i^k(\omega)}\omega)+m_k)\right.
			\\ &\quad\quad\quad\quad\left.+\sum_{i=0}^{j-1}(\hat{n}_{k+1}(\theta^{T_i^{k+1}(\omega)}\omega)+m_{k+1})+\hat{n}_{k+1}(\theta^{T_j^{k+1}(\omega)}\omega)\right)\\
			 &\leq T_{N_{k-1}(\omega)}^{k-1}(\omega)+l^{k-1,k}(\omega)+\sum_{i=1}^{N_k(\omega)-1}l_i^k(\omega)+l^{k,k+1}(\omega)+\sum_{i=1}^{j}l_i^{k+1}(\omega)\\
			&\quad\quad+(n-T_{j}^{k+1}(\omega)- \hat{n}_{k+1}(\theta^{T_j^{k+1}(\omega)}\omega))\\
			 &\leq T_{N_{k-1}(\omega)}^{k-1}(\omega)+\xi_kn+\xi_{k+1}n+m_{k+1}\\
			&\leq \left(\frac{T_{N_{k-1}(\omega)}^{k-1}(\omega)}{T_{N_{k}(\omega)}^{k}(\omega)}+\xi_k+\xi_{k+1}+\frac{m_{k+1}}{N_{k}(\omega)}\right)\cdot n\\&
			\overset{\eqref{estimation quotient is less than}}{\leq} \frac{\gamma/2}{h_{\mu_0}(F)-{11\gamma}/{2}}\cdot n.
		\end{split}
	\end{displaymath} As a consequence,
	\begin{align*}
		&\ \Big(h_{\mu_0}(F)-\frac{11\gamma}{2}\Big)\Big(\sum_{i=0}^{N_1(\omega)-1}(\hat{n}_1(\theta^{T_i^1(\omega)}\omega)+m_1)+\cdots+\sum_{i=0}^{N_k(\omega)-1}(\hat{n}_k(\theta^{T_i^k(\omega)}\omega)+m_k)\\
		&\quad+\sum_{i=0}^{j-1}(\hat{n}_{k+1}(\theta^{T_i^{k+1}(\omega)}\omega)+m_{k+1})+\hat{n}_{k+1}(\theta^{T_j^{k+1}(\omega)}\omega)\Big)\geq n(h_{\mu_0}(F)-6\gamma).
	\end{align*}
	Therefore,
	\begin{displaymath}
	\mu_{k+p,\omega}\left[B_n\left(\omega,x,\frac{\eta}{2^4}\right)\right]\leq \exp[-n(h_{\mu_0}(F)-6\gamma)].
	\end{displaymath}
	In all three cases, we conclude that
	\begin{displaymath}
	\mu_{k+p,\omega}\left[B_n\left(\omega,x,\frac{\eta}{2^4}\right)\right]\leq \liminf_{p\to\infty}\mu_{k+p,\omega}\left[B_n\left(\omega,x,\frac{\eta}{2^4}\right)\right]\leq \exp[-n(h_{\mu_0}(F)-6\gamma)].
	\end{displaymath}
	The proof of lemma \ref{lemma entropy distribution} is complete.
\end{proof}
\begin{proof}[Proof of Lemma \ref{lemma contain and nonempty}]
	For any $z\in \mathcal{B}_{l_{k+1}}(\omega,y_j^{k+1},{\epsilon}/{2^{k+1}})$, we have
	\begin{displaymath}
		d_{\omega}^{l_{k+1}}\left(z,y_j^{k+1}\right)<\frac{\epsilon}{2^{k+1}}.
	\end{displaymath}
Therefore, by triangular inequality, one has
\begin{displaymath}
	d_{\omega}^{l_{k}}\left(z,y_j^k\right)\leq d_{\omega}^{l_{k+1}}\left(z,y_{j}^{k+1}\right)+d_{\omega}^{l_k}(y_{j}^{k+1},y_j^k)\overset{ \eqref{equation shadowing property yk+1 and yk}}<\frac{\epsilon}{2^k}.
\end{displaymath}
This ends the proof of \eqref{equation contain}.

Similarly, one can prove that
\begin{equation}\label{equation closure contain}
	\overline{\mathcal{B}}_{l_{k+1}}\left(\omega,y_j^{k+1},\frac{\epsilon}{2^{k+1}}\right)\subset\overline{\mathcal{B}}_{l_k}\left(\omega,y_j^{k},\frac{\epsilon}{2^k}\right).
\end{equation}
By the compactness of $M$, we conclude that
\begin{displaymath}
	\bigcap_{k\geq1}\overline{\mathcal{B}}_{l_k}\left(\omega,y_j^k,\frac{\epsilon}{2^{k}}\right)\neq\varnothing.
\end{displaymath} Pick some $z_0\in\cap_{k\geq1}\overline{\mathcal{B}}_{l_k}(\omega,y_j^k,\epsilon/2^k)$, we have $d_{\omega}^{l_{k+1}}(z_0,y_j^{k+1})\leq\epsilon/2^{k+1}$ for any $k\in\mathbb{N}$. By triangular inequality, one has
\begin{displaymath}
d_{\omega}^{l_k}\left(z_0,y_j^k\right)\leq d_{\omega}^{l_{k+1}}\left(z_0,y_{j}^{k+1}\right)+d_{\omega}^{l_k}\left(y_j^{k+1},y_j^k\right)<\frac{\epsilon}{2^k},\quad\forall k\in\mathbb{N}
\end{displaymath}
Therefore, $z_0\in \cap_{k\geq1}\mathcal{B}_{l_k}(\omega,y_j^k,\epsilon/2^k)$. This ends the proof of \eqref{equation nonempty}.

\end{proof}
\begin{proof}[Proof of Lemma \ref{lemma limit of difference exists}]
  It suffices to estimate
  \begin{displaymath}	\left|\sum_{i=0}^{n-1}\varphi(\Theta^i(\omega,z))-\sum_{i=0}^{n-1}\varphi(\Theta^i(\omega,x))\right|
  \end{displaymath}
for $z\in \mathcal{Z}_{\varphi}(\omega,x,x_j,\epsilon)$.
Our estimations are divided into two parts.

\emph{Estimation for $y_j^k$.}

We introduce
\begin{displaymath}
	L_{k}=\left|\sum_{i=0}^{l_k-1}\varphi(\Theta^i(\omega,y_j^{k}))-\sum_{i=0}^{l_k-1}\varphi(\Theta^i(\omega,x))\right|.
\end{displaymath}
Let us obtain an upper estimation of $L_k$ by induction.

For $k=1$, we have by \eqref{equation shadowing 1st level} and triangular inequality that
\begin{displaymath}
    \begin{split}
    	L_1&=\left|\sum_{i=0}^{s_1+n_1-1}\varphi(\Theta^i(\omega,y_j^1))-\sum_{i=0}^{s_1+n_1-1}\varphi(\Theta^i(\omega,x))\right|\\&\leq2s_1\|\varphi\|_{C^0}+n_1\mathrm{var}\left(\varphi,\frac{\epsilon}{2}\right).
    \end{split}
\end{displaymath}
By the definition of $y_j^{k+1}$ and utilizing \eqref{equation shadowing property yk+1 and yk}, one has
\begin{displaymath}
	\begin{split}
		L_{k+1}&=\left|\sum_{i=0}^{l_{k+1}-1}\varphi(\Theta^i(\omega,y_j^{k+1}))-\sum_{i=0}^{l_{k+1}-1}\varphi(\Theta^i(\omega,x))\right|\\&\leq \left|\sum_{i=0}^{l_{k}-1}\varphi(\Theta^i(\omega,y_j^{k+1}))-\sum_{i=0}^{l_k-1}\varphi(\Theta^i(\omega,y_j^k))\right|+\left|\sum_{i=0}^{l_k-1}\varphi(\Theta^i(\omega,y_j^k))-\sum_{i=0}^{l_k-1}\varphi(\Theta^i(\omega,x))\right|\\&\quad+\left|\sum_{i=l_k}^{l_{k+1}-1}\varphi(\Theta^i(\omega,y_j^{k+1}))-\sum_{i=l_k}^{l_{k+1}-1}\varphi(\Theta^i(\omega,x))\right|\\&\leq l_k\mathrm{var}\left(\varphi,\frac{\epsilon}{2^{k+1}}\right)+L_k+2s_{k+1}\|\varphi\|_{C^0}+n_{k+1}\mathrm{var}\left(\varphi,\frac{\epsilon}{2^{k+1}}\right).
	\end{split}
\end{displaymath}
By induction, we conclude that
\begin{equation}\label{equation upper estimation of Lk}
	L_{k}\leq\sum_{i=1}^{k}l_j\mathrm{var}\left(\varphi,\frac{\epsilon}{2^i}\right)+\sum_{i=1}^{k}2s_i\|\varphi\|_{C^0}.
\end{equation}

Let us analyze the expression of $L_k$. We claim that
\begin{equation}\label{equation behaviour of Lk}
	\lim_{k\to\infty}\frac{L_k}{l_k}=0.
\end{equation}
The following limit is obtained through utilizing Stolz's theorem
\begin{equation}\label{equation Lk over lk 1}
	0\leq \lim_{k\to\infty}\frac{\sum_{i=1}^ks_i}{l_k}=\lim_{k\to\infty}\frac{s_k}{l_k-l_{k-1}}\leq\lim_{k\to\infty}\frac{s_k}{n_k}=0,
\end{equation}
where the last equality is due to \eqref{lim sk over nk}. Besides, one also has
\begin{equation}\label{equation Lk over lk 2}
	0\leq\lim_{k\to\infty}\frac{\sum_{i=1}^{k-1}l_i\mathrm{var}\left(\varphi,{\epsilon}/{2^i}\right)}{l_{k}}=\lim_{k\to\infty}\frac{l_{k-1}\mathrm{var}\left(\varphi,\epsilon/2^{k-1}\right)}{l_k-l_{k-1}}\leq \lim_{k\to \infty}\frac{l_{k-1}}{n_k}\mathrm{var}\left(\varphi,\frac{\epsilon}{2^{k-1}}\right)=0.
\end{equation}
where the last equality is due to \eqref{lim lk over nk+1}.
Therefore, \eqref{equation behaviour of Lk} follows directly from \eqref{equation Lk over lk 1} and \eqref{equation Lk over lk 2}.

\emph{Estimation for $z$.} We divide our estimation for $z$ into two cases.

Case 1: $l_k<n\leq l_k+s_{k+1}.$ Since $z\in\mathcal{Z}_{\varphi}(\omega,x,x_j,\epsilon)$, we have $z\in \mathcal{B}_{l_k}(\omega,y_j^k,\epsilon/2^k)$, i.e.
\begin{equation}\label{equation distance yjk and z}
	d_{\omega}^{l_k}\left(y_j^k,z\right)<\frac{\epsilon}{2^k}.
\end{equation}
By utilizing triangular inequality and \eqref{equation distance yjk and z}
\begin{equation*}
	\begin{split}
		&\,\quad\left|\sum_{i=0}^{n-1}\varphi(\Theta^i(\omega,z))-\varphi(\Theta^i(\omega,x))\right|\\&\leq\left|\sum_{i=0}^{l_k-1}\varphi(\Theta^i(\omega,z))-\sum_{i=0}^{l_k-1}\varphi(\Theta^i(\omega,y_j^k))\right|+\left|\sum_{i=1}^{l_k-1}\varphi(\Theta^i(\omega,y_j^k))-\sum_{i=1}^{l_k-1}\varphi(\Theta^i(\omega,x))\right|\\&\quad+\left|\sum_{i=l_k}^{n-1}\varphi(\Theta^i(\omega,z))-\sum_{i=l_k}^{n-1}\varphi(\Theta^i(\omega,x))\right|\\&\leq l_k\mathrm{var}\left(\varphi,\frac{\epsilon}{2^k}\right)+ L_k+2s_{k+1}\|\varphi\|_{C^0}.
	\end{split}
\end{equation*}
Hence, we have
\begin{equation}\label{estimation case 1}
	\frac{1}{n}\left|\sum_{i=0}^{n-1}\varphi(\Theta^i(\omega,z))-\varphi(\Theta^i(\omega,x))\right|\leq \mathrm{var}\left(\varphi,\frac{\epsilon}{2^k}\right)+\frac{L_k}{l_k}+\frac{2s_{k+1}}{n_k}\|\varphi\|_{C^0}.
\end{equation}

Case 2: $l_k+s_{k+1}<n\leq l_{k+1}.$ Similar with case 1, one has
\begin{displaymath}
	\begin{split}
		&\,\quad\left|\sum_{i=0}^{n-1}\varphi(\Theta^i(\omega,z))-\sum_{i=0}^{n-1}\varphi(\Theta^i(\omega,x))\right|\\&\leq \left|\sum_{i=0}^{l_k-1}\varphi(\Theta^i(\omega,z))-\sum_{i=0}^{l_k-1}\varphi(\Theta^i(\omega,y_j^k))\right|+\left|\sum_{i=1}^{l_k-1}\varphi(\Theta^i(\omega,y_j^k))-\sum_{i=1}^{l_k-1}\varphi(\Theta^i(\omega,x))\right|\\&\quad+\left|\sum_{i=l_k}^{l_k+s_{k+1}-1}\varphi(\Theta^i(\omega,z))-\sum_{i=l_k}^{l_k+s_{k+1}-1}\varphi(\Theta^i(\omega,x))\right|+\left|\sum_{i=l_k+s_{k+1}}^{n-1}\varphi(\Theta^i(\omega,z))-\sum_{i=l_k+s_{k+1}}^{n-1}\varphi(\Theta^i(\omega,y_j^{k+1}))\right|\\&\quad+\left|\sum_{i=l_k+s_{k+1}}^{n-1}\varphi(\Theta^i(\omega,y_j^{k+1}))-\sum_{i=l_k+s_{k+1}}^{n-1}\varphi(\Theta^i(\omega,x))\right|\\&\leq l_k\mathrm{var}\left(\varphi,\frac{\epsilon}{2^k}\right)+L_k+2{s_{k+1}}\|\varphi\|_{C^0}+2[n-(l_k+s_{k+1})]\mathrm{var}\left(\varphi,\frac{\epsilon}{2^{k+1}}\right).
	\end{split}
\end{displaymath}
Hence, dividing both sides by $n$, we conclude that
\begin{equation}\label{estimation case 2}
	\begin{split}
		&\quad\frac{1}{n}\left|\sum_{i=0}^{n-1}\varphi(\Theta^i(\omega,z))-\sum_{i=0}^{n-1}\varphi(\Theta^i(\omega,x))\right|\\&\leq \mathrm{var}\left(\varphi,\frac{\epsilon}{2^k}\right)+\frac{L_k}{l_k}+\frac{2s_{k+1}}{n_k}\|\varphi\|_{C^0}+2\mathrm{var}\left(\varphi,\frac{\epsilon}{2^{k+1}}\right).
	\end{split}
\end{equation}

In a word, combining \eqref{estimation case 1} and \eqref{estimation case 2} together, and sending $n\to\infty$, we arrive at
\begin{displaymath}
	\lim_{n\to\infty}\left[\frac{1}{n}\sum_{i=0}^{n-1}\varphi(\Theta^i(\omega,z))-\sum_{i=0}^{n-1}\varphi(\Theta^i(\omega,x))\right]=0.
\end{displaymath}
This ends the proof of lemma \ref{lemma limit of difference exists}.
\end{proof}

\begin{proof}[Proof of Lemma \ref{lemma dense set is abundant}]
	For any $z\in \mathcal{Z}_{\varphi}(\omega,x,x_j,\epsilon)$, we claim that
	\begin{equation}
	  d_M(z,x_j)<\epsilon.
	\end{equation}
   Notice that $z\in \mathcal{Z}_{\varphi}(\omega,x,x_j,\epsilon)\subset\mathcal{B}_{l_1}(\omega,y_j^1,\epsilon/2)$.Therefore, we have $d_{\omega}^{l_1}(z,y_j^1)<\epsilon/2$. Hence, by the shadowing property \eqref{equation shadowing 1st level}, one has
    \begin{displaymath}
    	d_M(z,x_j)\leq d_{\omega}^{l_1}(z,y_j^1)+d_{M}(y_j^1,x_j)<\epsilon.
    \end{displaymath}

    Notice that $\{x_j\}_{j=1}^{\infty}$ is dense in $M$. For any nonempty open ball $\mathcal{B}(y,\delta)\subset M$, there exists some $j_0\in\mathbb{N}$ such that $x_{j_0}\in \mathcal{B}(y,\delta)$. We pick some sufficiently small $\epsilon_k>$ with $\mathcal{B}(x_{j_0},\epsilon_k)\subset \mathcal{B}(y,\delta)$. Therefore, we have
    \begin{displaymath}
    	\mathcal{Z}_{\varphi}(\omega,x,x_{j_0},\epsilon_k)\subset\mathcal{B}(x_{j_0},\epsilon_k)\subset \mathcal{B}(y,\delta),
    \end{displaymath}
   which implies that $\mathcal{Z}_{\varphi}(\omega,x)\cap \mathcal{B}(y,\delta)\neq\varnothing$. This ends the proof of lemma \ref{lemma dense set is abundant}.
\end{proof}
\begin{proof}[Proof of Corollary \ref{corollary dense}]
    It suffices to show that for any nonempty open set $U\subset \Omega$ and $V\subset M$, we have
    \begin{displaymath}
    	(U\times V)\cap I_{\varphi}\neq\varnothing.
    \end{displaymath}By theorem \ref{thm residual}, there exists a $\mathbb{P}$-full measure set $\overline{\Omega}$, such that $I_{\varphi}(\omega)$ is residual, hence dense in $M$ for any $\omega\in\overline{\Omega}.$ Since $U$ is assigned positive $\mathbb{P}$-measure, there exists some $\omega_0\in U\cap \overline{\Omega}$. Furthermore, $I_{\varphi}(\omega_0)$ is dense in $M$. Therefore, there exists some $x_0\in I_{\varphi}(\omega)\cap V$. In a word, we conclude that
\begin{displaymath}
	(\omega_0,x_0)\in (U\times V)\cap I_{\varphi}.
\end{displaymath}
This ends the proof of corollary \ref{corollary dense}.
\end{proof}
\subsection*{Acknowledgment}
 The second author was partially supported by NNSF of China (12401230).

\bibliographystyle{acm}
\bibliography{Multifractalanalysisbib}
\end{document}